\documentclass[11pt]{article}
\usepackage[T1]{fontenc}
\usepackage[utf8]{inputenc}
\usepackage{geometry}
\usepackage{color}
\usepackage[english]{babel}
\usepackage{mathrsfs}
\usepackage{amsmath}
\usepackage{amsthm}
\usepackage{amssymb}
\usepackage{stmaryrd}
\usepackage{graphicx}
\usepackage{setspace}
\PassOptionsToPackage{normalem}{ulem}
\usepackage{ulem}
\usepackage[unicode=true,pdfusetitle,
 bookmarks=true,bookmarksnumbered=false,bookmarksopen=false,
 breaklinks=false,pdfborder={0 0 1},backref=false,colorlinks=true]
 {hyperref}

\makeatletter

\providecommand{\tabularnewline}{\\}

\theoremstyle{plain}
\newtheorem{thm}{\protect\theoremname}
\theoremstyle{plain}
\newtheorem{prop}[thm]{\protect\propositionname}
\theoremstyle{plain}
\newtheorem{cor}[thm]{\protect\corollaryname}
\theoremstyle{plain}
\newtheorem{lem}[thm]{\protect\lemmaname}
\ifx\proof\undefined
\newenvironment{proof}[1][\protect\proofname]{\par
	\normalfont\topsep6\p@\@plus6\p@\relax
	\trivlist
	\itemindent\parindent
	\item[\hskip\labelsep\scshape #1]\ignorespaces
}{%
	\endtrivlist\@endpefalse
}
\providecommand{\proofname}{Proof}
\fi
\theoremstyle{definition}
\newtheorem{defn}[thm]{\protect\definitionname}

\usepackage{amsfonts}
\usepackage{sectsty}
\usepackage{fancyhdr}
\usepackage{enumitem}
\usepackage{subfigure}\usepackage{multirow}
\usepackage{multicol}\usepackage{url}
\usepackage{array}
\usepackage{dsfont}
\usepackage{xspace}
\usepackage{caption}
\usepackage{algpseudocode}
\usepackage{algorithm}

\makeatother

\providecommand{\corollaryname}{Corollary}
\providecommand{\definitionname}{Definition}
\providecommand{\lemmaname}{Lemma}
\providecommand{\propositionname}{Proposition}
\providecommand{\theoremname}{Theorem}

\begin{document}

\title{High dimensional logistic entropy clustering}

\author{Edouard GENETAY\thanks{CREST, ENSAI, Univ. Rennes, LumenAI }, Adrien
SAUMARD\thanks{CREST, ENSAI, Univ. Rennes} and R{'e}mi COULAUD\thanks{LMO, Univ. Paris-Saclay, SNCF} }
\maketitle
Minimization of the (regularized) entropy of classification probabilities
is a versatile class of discriminative clustering methods. The classification
probabilities are usually defined through the use of some classical
losses from supervised classification and the point is to avoid modelisation of the full data distribution by just optimizing the law of the labels conditioned on
the observations. We give the first theoretical study of such methods,
by specializing to logistic classification probabilities. We prove
that if the observations are generated from a two-component isotropic
Gaussian mixture, then minimizing the entropy risk over a Euclidean
ball indeed allows to identify the separation vector of the mixture.
Furthermore, if this separation vector is sparse, then penalizing
the empirical risk by a $\ell_{1}$-regularization term allows to
infer the separation in a high-dimensional space and to recover its
support, at standard rates of sparsity problems. Our approach is based
on the local convexity of the logistic entropy risk, that occurs if
the separation vector is large enough, with a condition on its norm
that is independent from the space dimension. This local convexity
property also guarantees fast rates in a classical, low-dimensional
setting.


\section{Introduction}

The clustering problem can be described as follows: given a measurable
space $\mathcal{X}$, a sample $(X_{1},...,X_{n})\in\mathcal{X}^{n}$,
and an integer $K\geq2$, define a (random) labelling function $Y:\mathcal{X}\rightarrow\{1,...,K\}$.
In particular, to each data $X_{i}$, associate a label $Y_{i}$.
If the function $Y$ is deterministic, then the task is termed ``hard
clustering''. If the function $Y$ is random, the distribution of
the labels $Y(x)$, for $x\in\mathcal{X}$, being characterized by
the uplets $(\mathbb{P}(Y(x)=1),...,\mathbb{P}(Y(x)=K))$, then the
clustering task is said to be ``soft''. In the soft clustering case,
a common approach - called the modelling approach - is to model the
distribution of the data, typically as a mixture distribution, and
to directly relate the probabilities $(\mathbb{P}(Y(x)=1),...,\mathbb{P}(Y(x)=K))$
to the parameters of the mixture~\cite{MR3967046}. One can then
reduce to a hard clustering by assigning each point $x$ to the maximizer
of classification probabilities (or choose one at random amongst the
maximizers if it is non-unique).  Hard clustering algorithms include
the celebrated K-means~\cite{lloyd1982least,steinhaus1957kmeans,macqueen1967kmeans},
hierachical clustering~\cite{jain1999data}, spectral clustering~\cite{ng2002spectral})
among others.

Particularly developed in the machine learning community for its flexibility
when addressing complex data, the so-called ``discriminative approach''
to clustering amounts to model the classification probabilities $(\mathbb{P}(Y(x)=1),...,\mathbb{P}(Y(x)=K))$,
which can be understood as the conditional probabilities of the labels
with respect to the position $x$. Proceeding this way indeed avoids
the modelling of the whole distribution of data and often reduces
to encode in the classification probabilities, the frontiers separating
the clusters. In general, this is done through the use of classical
learning losses such as the logistic, the Hinge or the Conditional
Random Fields loss~\cite{dai2010minimum,gomes2010discriminative}.
More formally, one puts the constraint of $\mathbb{P}(Y(x)=k)$, $k\in\{1,...,K\}$,
being proportional to $\exp(\ell(\beta_{k},x))$, for a vector $\beta_{k}$
and a loss $\ell$. For instance the logistic loss gives classification
probabilities proportional to $\exp(w_{k}^{t}x+b_{k})$ and the Hinge
loss (for $K=2$) induces probabilities of a form proportional to
$\exp(-[1-(w_{k}^{t}\varphi(x)+b_{k})]_{+})$ for some feature map
$\varphi$ and with $(w_{1},b_{1})=(-w_{2},-b_{2})$ in this binary
case.

In addition, these losses were primarily introduced for supervised
learning and in order to transfer them to the unsupervised setting,
one has to define what would be a desirable (unobserved) label. Arguably,
when classifying data, one would prefer to be as sure as possible
of its cluster choice. This is equivalent to saying that the maximum
of classification probabilities would be as close to one as possible.
Hence, a natural criterion to infer a labelling function, would be
to define $\tilde{Y}$ through the probabilities $\mathbb{P}(\tilde{Y}(x)=k)=Z_{\tilde{\beta}}^{-1}(x)\exp(\ell(\tilde{\beta}_{k},x))$,
with a normalizing constant $Z_{\tilde{\beta}}(x)=\sum$$_{k=1}^{K}\exp(\ell(\tilde{\beta}_{k},x))$,
such that
\begin{equation}
(\tilde{\beta}_{1},...,\tilde{\beta}_{K})\in\arg\max_{(\beta_{1},...,\beta_{K})}\left\{ \frac{1}{n}\sum_{i=1}^{n}\frac{1}{Z_{\beta}(x_{i})}\max_{k\in\left\{ 1,...,K\right\} }[\exp(\ell(\beta_{k},x_{i}))]\right\} .\label{eq:crit_max}
\end{equation}
The associated theoretical target is $\mathbb{P}(Y_{*}(x)=k)=\exp(\ell(\beta_{*,k},x))$
with,
\[
(\beta_{*,1},...,\beta_{*,K})\in\arg\max_{(\beta_{1},...,\beta_{K})}\left\{ \mathbb{E}\left[\frac{1}{Z_{\beta}(X)}\max_{k\in\left\{ 1,...,K\right\} }[\exp(\ell(\beta_{k},X))]\right]\right\} ,
\]
where $X$ follows the unknown - and not modeled - distribution of
data.

But the maximum is not a smooth function and it may cause difficulties
when trying to optimize (\ref{eq:crit_max}). As a smooth proxy, one
can try to minimize the entropy of the classification probabilities,
since it achieves its minimum value when the latter probabilities
are all equal to zero or one. This amounts to search for a labelling
function $\hat{Y}$ satisfying $\mathbb{P}(\hat{Y}(x)=k)=Z_{\hat{\beta}}^{-1}(x)\exp(\ell(\hat{\beta}_{k},x))$
with
\begin{equation}
(\hat{\beta}_{1},...,\hat{\beta}_{K})\in\arg\min_{(\beta_{1},...,\beta_{K})}\left\{ \frac{1}{n}\sum_{i=1}^{n}{\rm Ent}\left\{ \mathbb{P}(Y_{\beta}(x_{i})=1),...,\mathbb{P}(Y_{\beta}(x_{i})=K)\right\} \right\} ,\label{eq:crit_entropy_intro}
\end{equation}
where
\begin{equation}
{\rm Ent}\left\{ \mathbb{P}(Y_{\beta}(x_{i})=1),...,\mathbb{P}(Y_{\beta}(x_{i})=K)\right\} =\sum_{k=1}^{K}-\frac{\exp(\ell(\beta_{k},x_{i}))}{Z_{\beta}(x_{i})}\log\left(\frac{\exp(\ell(\beta_{k},x_{i}))}{Z_{\beta}(x_{i})}\right).\label{eq:def_entropy_intro}
\end{equation}
Often, one has to restrict the search among vectors $(\beta_{1},...,\beta_{K})$
in a compact set, or to add to the entropy a regularization term encoding
the complexity of the vectors $(\beta_{1},...,\beta_{K})$~\cite{gomes2010discriminative,dai2010minimum}.
In this second formulation, the theoretical target $(\beta_{0,1},...,\beta_{0,K})$
of estimation is,

The use of entropy terms in semi-supervised and unsupervised learning
is indeed natural and has been the object of active research~\cite{grandvalet2005semi,gomes2010discriminative,dai2010minimum,sugiyama2011information,sugiyama2014information,shi2012information,li2004minimum,aldana2015clustering,muller2012information}.
Furthermore, this approach is at the core of some state-of-the-art
deep clustering approaches~\cite{jabi2019deep}. Another fruitful
approach in discriminative clustering consists in considering convex
relaxations of some initial, untractable criteria and this methodology
often comes with strong theoretical guarantees~\cite{flammarion2017robust,joulin2010discriminative,bach2007diffrac,peng2007approximating,giraud2021introduction,bunea2016pecok,bunea2020model,giraud2019partial,mixon2017clustering,royer2017adaptive,chen2021cutoff}.

The starting point of our work consists in the following observation:
to our knowledge, no theoretical guarantee - of the type of convergence
rates - exists in the literature for (regularized) minimum entropy
estimators (\ref{eq:crit_entropy_intro}). This a weakness compared
to other approaches, such as convex relaxations techniques for instance.
But from a practical perspective, estimators of the form of (\ref{eq:crit_entropy_intro})
have already proved to be efficient and flexible - allowing for instance
feature maps embedding and the use of deep architectures - and the
lack of theoretical studies needs to be filled.

We consider the unsupervised classification of a bipartite high-dimensional
Gaussian mixture, with sparse means. This framework is indeed a good
benchmark, since on the one hand, it is sufficiently simple to allow
us to understand the nature of the target $(\beta_{0,1},...,\beta_{0,K})$
- with $K=2$ and $\beta_{0,1}=-\beta_{0,2}$ in our bipartite framework
- and to investigate the rate of convergence of estimators of the
form of (\ref{eq:crit_entropy_intro}), suitably regularized by a
$\ell_{1}$-penalty. On the other hand, the two-component high-dimensional
Gaussian mixture has received recently at lot of attention~\cite{bouveyron2014discriminative,azizyanminimax,ndaoud2018sharp,li2017minimax,jin2017phase,fan2018curse,cai2019chime,azizyan2015efficient,jin2016influential,brennan2019average,loffler2020computationally}.
Let us emphasize that our goal is not \textit{a priori} to provide
a state-of-the-art method, specifically designed to solve the high-dimensional
Gaussian mixture clustering, but to explore for the first time the
theoretical behavior of discriminative estimators that minimize the
(regularized) classification entropy and see how they can adapt to
a sparse setting.

\section{Some notations and definitions\label{sec:Notations}}

Let $a:=\left(a_{1},...,a_{d}\right)\in\mathbb{R}^{d}$ and $X$ be
a random variable valued in $\mathbb{R}^{d}$, with distribution $P$.
More precisely $X:=\varepsilon Z$ with $\varepsilon\sim\textrm{Rad}\left(\frac{1}{2}\right)$
and $Z\sim\mathcal{N}\left(a,I_{d}\right)$ a Gaussian vector independent
from $\varepsilon$, with normalized variance equal to the identity
matrix $I_{d}$. Take $n\in\mathbb{N}^{*}$, $X^{(1)},...,X^{(n)}$
are observations of $X$ independent and identically distributed according
to $P$. Our goal is to estimate the labelling function $Y_{*}(x)={\rm sign}(x^{t}a)$,
or its opposite, which gives the same hard clustering. This amounts
to estimate the separation vector $a$. To do this, we will use an
entropy criterion.

Set the logistic probability $p_{\beta}\left(X\right):=1/(1+e^{X^{t}\beta})$
where $\beta\in\mathbb{R}^{d}$ and its complementary probability
$q_{\beta}\left(X\right):=e^{X^{t}\beta}/(1+e^{X^{t}\beta})$. The
\textit{logistic entropy} $\rho_{\beta}$ is defined as follows, $\rho_{\beta}\left(X\right):=\rho\left(\beta^{t}X\right)=-p_{\beta}\left(X\right)\log p_{\beta}\left(X\right)-q_{\beta}\left(X\right)\log q_{\beta}\left(X\right)$.
The associated risk is $\mathcal{R}\left(\beta\right):=\mathbb{E}\left[\rho_{\beta}\left(X\right)\right]$.
The latter expectation will also be denoted $P\rho_{\beta}$ for short.
Let $\left\Vert \cdot\right\Vert _{1}$, $\left\Vert \cdot\right\Vert _{2}$
and $\left\Vert \cdot\right\Vert _{\infty}$ be respectively the $L_{1}$,$L_{2}$
and $L_{\infty}$-norm, and denote $B_{1}\left(0,R\right)$, $B_{2}\left(0,R\right)$
and $B_{\infty}\left(0,R\right)$ their corresponding balls centered
at $0$ with radius $R$ in $\mathbb{R}^{d}$. We consider the minimizer
$\beta_{0}$ of the risk $\mathcal{R}\left(\beta\right)$ over a $L_{2}$-ball
$B_{2}\left(0,R\right)$ - where the radius $R$ will be fixed latter
-, $\beta_{0}\in\underset{\beta\in B_{2}\left(0,R\right)}{\arg\min}\left\{ \mathcal{R}\left(\beta\right)\right\} $,
with excess risk $\mathcal{E}\left(\beta,\beta_{0}\right):=\mathcal{R}\left(\beta\right)-\mathcal{R}\left(\beta_{0}\right)$,
for $\beta\in B_{2}\left(0,R\right)$. The empirical distribution
of $X^{(1)},...,X^{(n)}$ is $P_{n}:=\frac{1}{n}\sum_{i=1}^{n}\delta_{X^{(i)}}$,
where $\delta_{X^{(i)}}$ is the Dirac distribution on $X^{(i)}$,
and the quantity$\hat{\mathcal{R}}_{n}\left(\beta\right):=P_{n}\rho_{\beta}=\frac{1}{n}\sum_{i=1}^{n}\rho_{\beta}\left(X^{(i)}\right)$
is the empirical counterpart of $\mathcal{R}\left(\beta\right)$,
called the empirical risk.

We denote by $\gamma$ the probability density function of a centered
standard real Gaussian variable $\mathcal{N}\left(0,1\right)$. $\Phi$
is its cumulative distribution function and $\Phi^{c}:t\mapsto\int_{t}^{\infty}\gamma\left(u\right)du$
the tail distribution of the density $\gamma$. In addition, we write
$G$ the so-called Gaussian Mill's ratio $G\left(x\right):=\frac{\Phi^{c}\left(x\right)}{\gamma\left(x\right)}$.
In this article $\alpha:x\mapsto-\frac{e^{x}}{\left(1+e^{x}\right)^{2}}\left(1+x\frac{1-e^{x}}{1+e^{x}}\right)$
and $x_{1}$ is the unique element of $\left\{ x>0:\alpha(x)=0\right\} $,
satisfying $x_{1}\in[1.54,1.55]$.

$\forall u,v\in\mathbb{R},u\wedge v:=\min\left(u,v\right)$ and $u\lor v:=\max\left(u,v\right)$.
For a vector $\beta=(\beta_{1},...,\beta_{p})^{t}\in\mathbb{R}^{p}$,
we define its support as the set $S$ of indices such that $S=\{i\in\{1,...,p\};\beta_{i}\neq0\}$.
The vector $\beta$ is said to be $s$-sparse if ${\rm Card}(S)\leq s.$
Furthermore, for a set of indices $I\subset\{1,...,p\}$, we denote
$\beta^{I}\in\mathbb{R}^{p}$ the vector such that $\beta_{i}^{I}=\beta_{i}$
if $i\in I$ and $\beta_{j}^{I}=0$ if $j\not\in I$.

\section{Minimising the risk over a $L_{2}$-ball}

Recall that 
\[
\beta_{0}\in\underset{\beta\in B_{2}\left(0,R\right)}{\arg\min}\left\{ \mathcal{R}\left(\beta\right)\right\} ,
\]
where the radius $R$ will be fixed later. Let us investigate the
geometry of the risk $\mathcal{R}$ defined by the logistic entropy.
\begin{prop}
\label{lem:minimum_of_risk_on_a_ball}The risk is symmetric, $\mathcal{R}\left(\beta\right)=\mathcal{R}\left(-\beta\right)$,
and the risk value $\mathcal{R}\left(\beta\right)$ with $\left\Vert \beta\right\Vert _{2}=r$
fixed is decreasing with respect to $\left|\beta^{t}a\right|$.
\end{prop}
Proposition \ref{lem:minimum_of_risk_on_a_ball} states that the risk
is symmetric around zero, and that its values on a sphere are increasing
with respect to the distance to the line $\mathbb{R}a$. Its proof
can be found in Section \ref{subsec:Proofs_main_results}.
\begin{prop}
\label{lem:minimum_risk_on_a_line}The function $\lambda\mapsto\mathcal{R}\left(\lambda\beta\right)$
is decreasing for $\lambda\in\mathbb{R}_{+}$.
\end{prop}
In Proposition \ref{lem:minimum_risk_on_a_line}, it is proved that
the risk is decreasing on semi-lines starting at zero. For a proof
of this result, see Section \ref{subsec:Proofs_main_results}. From
Propositions \ref{lem:minimum_of_risk_on_a_ball} and \ref{lem:minimum_risk_on_a_line},
we characterize the minimizers of the risk over a $L_{2}$-ball. 
\begin{cor}
\label{prop:unicity_minimum_on_half_ball}The minimum of $\mathcal{R}\left(\beta\right)$
on $B_{2}\left(0,R\right)$ is reached at $\pm\beta_{0}$ where $\beta_{0}:=Ra/\left\Vert a\right\Vert _{2}$.
\end{cor}
From Corollary \ref{prop:unicity_minimum_on_half_ball}, we deduce
that estimating $\beta_{0}$ or its opposite directly gives an estimation
of the best labelling function $Y_{*}$ for our clustering problem.
A look at the proof of Propositions \ref{lem:minimum_of_risk_on_a_ball}
and \ref{lem:minimum_risk_on_a_line} shows that these results, and
hence Corollary \ref{prop:unicity_minimum_on_half_ball}, hold true
in the more general setting where the distribution of $Z$ is only
assumed to be spherically symmetric.

In order to tackle the estimation of a sparse separation vector $a$,
the following property will be helpful.
\begin{thm}
\label{lem:Condition-2} Let $\beta_{0}=Ra/\left\Vert a\right\Vert _{2}$
and let $\Lambda_{min}$ be the smallest eigenvalue of the Hessian
$d_{\beta_{0}}^{2}\mathcal{R}$ at $\beta_{0}$. Take a parameter
\textup{$\nu=0.95$,} $R\geq\sqrt{x_{1}+0.08}$ ($R=1.28$ for instance)
and assume that $\left\Vert a\right\Vert _{2}\geq2R$, then 
\[
\Lambda_{min}\geq\frac{\nu}{4}\left(\Phi^{c}\left(\left\Vert a\right\Vert _{2}-\frac{x_{1}}{R}\right)-\Phi^{c}\left(\left\Vert a\right\Vert _{2}+\frac{x_{1}}{R}\right)\right).
\]
 
\end{thm}
Theorem \ref{lem:Condition-2} states that if the radius $R$ and
the mean vector $a$ are sufficiently large, then the risk defined
by the logistic entropy is locally strongly convex around $\beta_{0}$.
The risk is not convex over the whole $L_{2}$-ball $B_{2}(0,R)$,
but this local convexity is very convenient, since it allows to deduce
a quadratic growth of the excess risk pointed on $\beta_{0}$, as
follows.
\begin{lem}
\label{lem:lemme_D01_remi}Set $\beta_{0}$ the unique minimum of
$\mathcal{R}\left(\cdot\right)$ on $\Psi_{U}:=\left\{ \beta\in B_{2}\left(0,R\right):\beta^{t}U>0\right\} $
where $U$ is a random variable uniformly distributed on the unit
$L^{2}$-ball. Assume that $R\geq\sqrt{x_{1}+0.08}$ and $\left\Vert a\right\Vert _{2}\geq2R$.
We have
\[
\underset{\beta\in\Psi_{U}}{\inf}\frac{\mathcal{E}\left(\beta,\beta_{0}\right)}{\left\Vert \beta-\beta_{0}\right\Vert _{2}^{2}}\geq c_{0}>0
\]
with 
\[
c_{0}=L_{0}\frac{\left(\left\Vert a\right\Vert _{2}-R\right)^{6}}{\left\Vert a\right\Vert _{2}^{8}R^{2}}.\exp\left(-\left\Vert a\right\Vert _{2}R-2R^{2}\right)
\]
for a numerical constant $L_{0}$ ($L_{0}=9\times2^{22}$ holds).
\end{lem}
The quadratic growth of the excess risk stated in Lemma \ref{lem:lemme_D01_remi}
will turn out to be a keystone to prove the oracle inequality for
the excess risk of the minimizer of empirical risk regularized by
a $\ell_{1}$ penalty (see Section \ref{sec:An-oracle-inequality}).
The proof of Lemma \ref{lem:lemme_D01_remi} is postponed to Section
\ref{subsec:Proofs_main_results}.

\section{An oracle inequality in high dimension\label{sec:An-oracle-inequality}}

Recall that $\beta_{0}=Ra/\left\Vert a\right\Vert _{2}$ is a minimizer
of the risk over the $L_{2}$-ball of radius $R$: $\beta_{0}\in\underset{\beta\in B_{2}(0,R)}{\arg\min}\mathcal{R}\left(\beta\right)$.
Set $\Psi_{U}:=\left\{ \beta\in B_{2}\left(0,R\right):\beta^{t}U>0\right\} $
and where $U$ is a random variable uniformly distributed on the unit
Euclidean sphere, independent from the observations. We have $\mathbb{P}(\beta_{0}^{t}U=0)=0$
and so $\beta_{0}$ or its opposite belongs to $\Psi_{U}$. Without
loss of generality, we assume that $\beta_{0}\in\Psi_{U}$ and analyze
the situation conditionnally on the choice of $U$. 

We investigate the behavior of the following estimator,
\begin{equation}
\hat{\beta}:=\underset{\beta\in\Psi_{U}}{\arg\min}\left\{ \mathcal{R}_{n}\left(\beta\right)+\lambda\left\Vert \beta\right\Vert _{1}\right\} .\label{eq:def_beta_hat_penalized}
\end{equation}
Set also the empirical process $V_{n}\left(\beta\right):=\left(P_{n}-P\right)\left(\rho_{\beta}\right)$.
For some $T>1,$ define the event
\begin{equation}
\mathcal{T}:=\left\{ \sup_{\beta\in B_{2}\left(0,R\right)}\frac{\left|V_{n}\left(\beta\right)-V_{n}\left(\beta_{0}\right)\right|}{\left\Vert \beta-\beta_{0}\right\Vert _{1}\lor\lambda_{0}}\leq2T\lambda_{0}\right\} ,\label{eq:def_T}
\end{equation}
where $\lambda_{0}$>0 is to be fixed in the following theorem.
\begin{thm}
\label{thm:main_result} Fix $n\geq2$. Assume that $\beta_{0}$ -
or equivalently $a$ - is $s$-sparse, for some integer $s\geq1$,
and denote $S$ its support. Assume also that $R=\sqrt{x_{1}+0.08}$
and $\left\Vert a\right\Vert _{2}\geq2R$. Set $M_{n}:=\left\Vert a\right\Vert _{\infty}+\sqrt{2\log d}+\sqrt{2\log\left(1+n\right)}$
and 
\[
\lambda_{0}:=3LM_{n}\left(5\sqrt{3\log\left(2d\right)}\log n+4\right)n^{-1/2}.
\]
 When the event $\mathcal{T}$ occurs, it holds: $\forall\lambda>2T\lambda_{0}$,
\begin{equation}
\mathcal{E}\left(\hat{\beta},\beta_{0}\right)+4\left(\lambda-2T\lambda_{0}\right)\left\Vert \hat{\beta}^{S^{c}}\right\Vert _{1}\leq A_{\left\Vert a\right\Vert _{2},R}s\left(T\lambda_{0}+\lambda\right)^{2},\label{eq:oracle_inequality}
\end{equation}
where $A_{\left\Vert a\right\Vert _{2},R}$ is a constant depending
only on $\left\Vert a\right\Vert _{2}$ and $R$. More precisely,
for a numerical constant $A_{0}$, one can take
\[
A_{\left\Vert a\right\Vert _{2},R}=A_{0}\left\Vert a\right\Vert _{2}^{8}\left(\left\Vert a\right\Vert _{2}-R\right)^{-6}R^{2}e^{\left\Vert a\right\Vert _{2}R+2R^{2}}.
\]

Furthermore, the event $\mathcal{T}$ occurs with probability at least
\[
1-\frac{3}{4}\log\left(\frac{4R^{2}nd}{L^{2}M_{n}^{2}}\right)\exp\left(-21\left(T-1\right)^{2}\log\left(2d\right)\log^{2}n\right)-\frac{1}{25T^{2}\log\left(2d\right)n\log^{2}n}.
\]
\end{thm}
According to Theorem~\ref{thm:main_result}, if the regularization
parameter $\lambda$ is equal for instance to $3T\lambda_{0}$, then
the rate of convergence of the excess risk is of the order 
\[
\frac{s\log d\log^{2}n\log\left(d\lor n\right)}{n},
\]
with a pre-factor that only depends on $\left\Vert a\right\Vert _{2}$
and $R$. Thus the estimator $\hat{\beta}$ adapts to sparsity and
is able to estimate $\beta_{0}$ even if $d>>n$. Furthermore, the
rate of convergence of $\left\Vert \hat{\beta}^{S^{c}}\right\Vert _{1}$
would be given by
\[
s\sqrt{\frac{\log d\log^{2}n\log\left(d\lor n\right)}{n}},
\]
with also a pre-factor that only depends on $\left\Vert a\right\Vert _{2}$
and $R$. This means that if $s$ and $d$ are such that this rate
(for a bounded $\left\Vert a\right\Vert _{2}$) goes to zero with
$n$ growing to infinity, then the support $S$ of $\beta_{0}$ is
recovered in the sense that $\left\Vert \hat{\beta}^{S^{c}}\right\Vert _{1}$
goes to zero.

Note however that the dependence in $\left\Vert a\right\Vert _{2}$
is exponential in our bounds. This due to our argument of proof, which
uses the local convexity of the risk around $\beta_{0}$. But when
$\left\Vert a\right\Vert _{2}$ is large, the risk tends to be flat
(see Theorem \ref{lem:Condition-2}). This local convexity argument
is also at the core the approach, developed in \cite{MR2677722},
to the non-convex $\ell_{1}$-penalized loss in mixture regression
(see also \cite[Chapter 9]{buhlmann2011statistics}). Note that the
needed lower bound on $\left\Vert a\right\Vert _{2}$ is independent
from the dimension $d.$

A careful look at the proofs also shows that when the conclusion of
Lemma \ref{lem:lemme_D01_remi} holds, that is the excess risk dominates
the square of the Euclidean distance, then Theorem \ref{thm:main_result}
still holds for a general, bounded and Lipschitz loss.

It is also worth noting that in a classical, non-sparse case where
the dimension is (much) smaller than the sample size, a convergence
bound could also be obtained, by standard empirical process techniques.
Indeed, the loss $\rho$ is bounded and Lipschitz, so the rate of
convergence of the following estimator,
\[
\tilde{\beta}\in\arg\min_{\beta\in B_{2}(0,R)}\left\{ \hat{\mathcal{R}}_{n}(\beta)\right\} ,
\]
is of the order 
\[
\sqrt{\frac{Rd}{n}}+\text{\ensuremath{\sqrt{\frac{\log(1/\delta)}{n}}}+\ensuremath{\frac{\log(1/\delta)}{n}}},
\]
up to a numerical pre-factor and on an event of probability at least
$1-\delta$ for $\delta\in(0,1)$. An important remark is that the
latter rate in $\sqrt{d/n}$ holds without any assumption on $R$
and $\left\Vert a\right\Vert _{2}$, because the local convexity of
the risk on $\beta_{0}$ - that is Theorem \ref{lem:Condition-2}
- is not needed to prove it. If Theorem \ref{lem:Condition-2} furthermore
holds, it is easy to see that the rate is actually $d/n$, up to a
pre-factor. Indeed, Theorem induces a so-called margin relation for
the excess risk, which in turn induces a fast rate, since the loss
is bounded (see for instance \cite{Massart:07}).


Also, one can consider the adaptive selection of the regularization
parameter. For this, a sensible idea is to consider a BIC-type criterion
defined with the active set of the estimators corresponding to different
values of the regularization parameter.

We postpone to a forthcoming addition the practical implementation
of the estimator, together with comparisons in the sparse two-component
Gaussian mixture model with other available algorithms. 

\section{Proofs}

Define the empirical process $V_{n}\left(\beta\right)=\left(P_{n}-P\right)\left(\rho_{\beta}\right)$
and $V_{n}^{trunc}\left(\beta\right)=\left(P_{n}-P\right)\left(\rho_{\beta}I_{\left\{ G\left(X\right)\leq M_{n}\right\} }\right)$
where $G\left(X\right):=\left\Vert X\right\Vert _{\infty}$ and note
that $\rho_{\beta}:\beta\mapsto\rho(\beta^{t}X)$ is $L$-lipschitz
(with $L<2.5$).

\subsection{Proofs of the main results\label{subsec:Proofs_main_results}}
\begin{proof}[Proof of Proposition \ref{lem:minimum_of_risk_on_a_ball}]
Take $X=\varepsilon Z$ where $\varepsilon\sim Rad\left(1/2\right)$
and $Z\sim\mathcal{N}\left(a,I_{d}\right)$, with $a\in\mathbb{R}^{d}$
and also $N\sim\mathcal{N}\left(0,1\right)$. Because expression~(\ref{eq:criterion_form_2})
of Lemma \ref{lem:formula_derivatives} is symmetric in $X$, one
has $\mathcal{R}\left(\beta\right)=\mathcal{R}\left(-\beta\right)$
and

\begin{align*}
\mathcal{R}\left(\beta\right) & =\mathbb{E}\left[\log\left(1+e^{Z^{t}\beta}\right)-\frac{Z^{t}\beta e^{Z^{t}\beta}}{1+e^{Z^{t}\beta}}\right].
\end{align*}
The distribution of the real-valued random variable $Z^{t}\beta$
is $\mathcal{N}\left(\beta^{t}a,\left\Vert \beta\right\Vert _{2}^{2}\right)$
and we assume that $\left\Vert \beta\right\Vert _{2}=r$. The criterion
can be seen as a function of $\mu:=\beta^{t}a$ and $r$:

\begin{align}
\mathcal{R}\left(\beta\right) & =\mathbb{E}\left[\log\left(1+e^{\mu+rN}\right)-\frac{\left(\mu+rN\right)e^{\mu+rN}}{1+e^{\mu+rN}}\right]=:\mathcal{R}\left(\mu,r\right).\label{eq:risk_derivative_mean}
\end{align}
 Its derivative with respect to $\mu$ is:
\begin{align*}
\partial_{\mu}\mathcal{R}\left(\mu,r\right) & =\frac{d}{d\mu}\mathbb{E}\left[\log\left(1+e^{\mu+rN}\right)-\frac{\left(\mu+rN\right)e^{\mu+rN}}{1+e^{\mu+rN}}\right]\\
 & =\mathbb{E}\left[\frac{d}{d\mu}\log\left(1+e^{\mu+rN}\right)-\frac{d}{d\mu}\frac{\left(\mu+rN\right)e^{\mu+rN}}{1+e^{\mu+rN}}\right]\\
 & =\mathbb{E}\left[\frac{e^{\mu+rN}}{1+e^{\mu+rN}}-\left(\frac{e^{\mu+rN}}{1+e^{\mu+rN}}+\frac{\left(\mu+rN\right)e^{\mu+rN}}{1+e^{\mu+rN}}+\left(\mu+rN\right)e^{\mu+rN}\frac{-e^{\mu+rN}}{\left(1+e^{\mu+rN}\right)^{2}}\right)\right]\\
 & =\mathbb{E}\left[-\frac{\left(\mu+rN\right)e^{\mu+rN}}{1+e^{\mu+rN}}+\left(\mu+rN\right)e^{\mu+rN}\frac{e^{\mu+rN}}{\left(1+e^{\mu+rN}\right)^{2}}\right]\\
 & =\mathbb{E}\left[\frac{\left(\mu+rN\right)e^{\mu+rN}}{1+e^{\mu+rN}}\left(\frac{e^{\mu+rN}}{1+e^{\mu+rN}}-1\right)\right]\\
\partial_{\mu}\mathcal{R}\left(\mu,r\right) & =-\mathbb{E}\left[\frac{\left(\mu+rN\right)e^{\mu+rN}}{\left(1+e^{\mu+rN}\right)^{2}}\right].
\end{align*}
Let us define $g:x\mapsto\frac{xe^{x}}{(1+e^{x})^{2}}$ so that $\partial_{\mu}\mathcal{R}\left(\mu,r\right)=-\mathbb{E}\left[g\left(\mu+rN\right)\right]$.
We use the lemma~\ref{lem:technical_lemma_same_sign_as_mu} and the
fact that $g$ is odd and positive on $(0,+\infty)$ to conclude that
$\mathbb{E}\left[g\left(\mu+rN\right)\right]$ has the sign of $\mu$,
which gives the result.
\end{proof}
\begin{proof}[Proof of Proposition \ref{lem:minimum_risk_on_a_line}]
Take $\beta\in\mathbb{R}^{d}$, there is $u\in\mathbb{R}$ such that
$\beta^{t}a=u\left\Vert \beta\right\Vert _{2}$. Recall Identity~(\ref{eq:risk_derivative_mean})
above, where $\mathcal{R}$ can be seen as a function of $\mu$ and
$r$ with $Z^{t}\beta\sim\mathcal{N}\left(\mu,r^{2}\right)$. Then
we have
\[
\frac{\partial\mathcal{R}\left(\lambda\beta\right)}{\partial\lambda}=\frac{\partial\mathcal{R}\left(\lambda\beta^{t}a,\left\Vert \lambda\beta\right\Vert _{2}\right)}{\partial\lambda}=\frac{\partial\mathcal{R}\left(ru,r\right)}{\partial r}\left\Vert \beta\right\Vert _{2}.
\]
We set $\forall u\in\mathbb{R},N_{u}\sim\mathcal{N}\left(u,1\right)$
and Equation~(\ref{eq:risk_derivative_mean}) gives:

\begin{align*}
\frac{\partial\mathcal{R}\left(ru,r\right)}{\partial r} & =\frac{\partial}{\partial r}\mathbb{E}\left[\log\left(1+e^{ru+rN_{0}}\right)-\frac{\left(ru+rN_{0}\right)e^{\left(ru+rN_{0}\right)}}{1+e^{\left(ru+rN_{0}\right)}}\right]\\
 & =\mathbb{E}\left[\frac{\partial}{\partial r}\log\left(1+e^{rN_{u}}\right)-\frac{\partial}{\partial r}\left(\frac{rN_{u}e^{rN_{u}}}{1+e^{rN_{u}}}\right)\right]\\
 & =\mathbb{E}\left[\frac{N_{u}e^{rN_{u}}}{1+e^{rN_{u}}}-\left(\frac{N_{u}e^{rN_{u}}}{1+e^{rN_{u}}}+\frac{rN_{u}\left(N_{u}e^{rN_{u}}\right)}{1+e^{rN_{u}}}+rN_{u}e^{rN_{u}}\frac{-N_{u}e^{rN_{u}}}{\left(1+e^{rN_{u}}\right)^{2}}\right)\right]\\
 & =\mathbb{E}\left[-\frac{rN_{u}\left(N_{u}e^{rN_{u}}\right)}{1+e^{rN_{u}}}+rN_{u}e^{rN_{u}}\frac{N_{u}e^{rN_{u}}}{\left(1+e^{rN_{u}}\right)^{2}}\right]\\
 & =\mathbb{E}\left[\frac{rN_{u}\left(N_{u}e^{rN_{u}}\right)}{1+e^{rN_{u}}}\left(\frac{e^{rN_{u}}}{1+e^{rN_{u}}}-1\right)\right]\\
 & =\mathbb{E}\left[\frac{rN_{u}\left(N_{u}e^{rN_{u}}\right)}{1+e^{rN_{u}}}\left(\frac{-1}{1+e^{rN_{u}}}\right)\right]\\
 & =-\mathbb{E}\left[\frac{rN_{u}^{2}e^{rN_{u}}}{\left(1+e^{rN_{u}}\right)^{2}}\right]<0.
\end{align*}
Hence $\frac{\partial\mathcal{R}\left(\lambda\beta\right)}{\partial\lambda}<0$
as required.
\end{proof}
\begin{proof}[Proof of Theorem \ref{lem:Condition-2}]
We make use of Equation~(\ref{eq:final_condition_on_a_R_nu}) from
Lemma~\ref{lem:risk_bilinear_values_control}: $\forall a\in\mathbb{R}^{d}$,
$R,\nu>0$, 
\begin{align}
R\left(1-\left(R-\left\Vert a\right\Vert _{2}+\frac{x_{1}+\frac{8}{100}}{R}\right)G\left(\frac{x_{1}}{R}+R-\left\Vert a\right\Vert _{2}\right)\right) & \geq\left(1+\nu\right)\frac{e^{x_{1}}}{4}G\left(\left\Vert a\right\Vert _{2}-\frac{x_{1}}{R}\right),\label{eq:final_condition_on_a_R_nu-1}
\end{align}
where, see Section \ref{sec:Notations}, $x_{1}$ is a positive numerical
constant and the function $G$ is the so-called Gaussian Mill's ratio.
By Proposition \ref{prop:G_decreasing}, we also have that $G$ is
decreasing on $\mathbb{R}$. Hence, if Equation (\ref{eq:final_condition_on_a_R_nu-1})
is satisfied for some values of $\left\Vert a\right\Vert _{2},R$
and $\nu$ such that $\left\Vert a\right\Vert _{2}-\left(R+\left(x_{1}+0.08\right)R^{-1}\right)>0$,
then it is satisfied for any triplet $(\left\Vert a\right\Vert _{2}+h,R,\nu)$
with $h>0$. In addition, we know from Lemma~\ref{lem:particular_case_of_Eq}
that $\left\Vert a\right\Vert _{2}=2R\approx2.548$, $R=\sqrt{x_{1}+0.08}\approx1.2741$
and $\nu=0.95$ make Equation~(\ref{eq:final_condition_on_a_R_nu-1})
hold true. Consequently, it also holds true when $\left\Vert a\right\Vert _{2}\geq2R\approx2.548$,
$R=\sqrt{x_{1}+0.08}\approx1.2741$ and $\nu=0.95$.

According to lemma~\ref{lem:risk_bilinear_values_control}, when
Equation (\ref{eq:final_condition_on_a_R_nu-1}) holds, one can control
from below the values of $\left(d_{\beta_{0}}^{2}\mathcal{R}\right)\left(h,h\right)$.
More precisely, 
\begin{align*}
\Lambda_{min} & :=\inf_{\left\Vert h\right\Vert =1}\left(d_{\beta_{0}}^{2}\mathcal{R}\right)\left(h,h\right)\\
 & \geq\underset{\eta=\left\Vert h_{\parallel}\right\Vert }{\inf_{\left\Vert h\right\Vert =1}}\frac{\nu}{4}\left(\eta^{2}\frac{x_{1}^{2}}{R^{2}}+1-\eta^{2}\right)\left(\Phi^{c}\left(\left\Vert a\right\Vert _{2}-\frac{x_{1}}{R}\right)-\Phi^{c}\left(\left\Vert a\right\Vert _{2}+\frac{x_{1}}{R}\right)\right)\\
 & =\frac{\nu}{4}\left(\Phi^{c}\left(\left\Vert a\right\Vert _{2}-\frac{x_{1}}{R}\right)-\Phi^{c}\left(\left\Vert a\right\Vert _{2}+\frac{x_{1}}{R}\right)\right)\underset{=1}{\underbrace{\inf_{0\leq\eta\leq1}\left(\eta^{2}\frac{x_{1}^{2}}{R^{2}}+1-\eta^{2}\right)}}
\end{align*}
because $x_{1}/R\geq1$. This proves the result.
\end{proof}
\begin{proof}[Proof of Lemma \ref{lem:lemme_D01_remi}]
The risk $\mathcal{R}$ admits two minima $\beta_{0}$ and $-\beta_{0}$
on $B_{2}\left(0,R\right)$. We consider 
\[
\Psi_{U}=\left\{ \beta\in B_{2}\left(0,R\right):\beta^{t}U>0\right\} ,
\]
 where $U$ is a random variable uniformly distributed on the unit
$L^{2}$-ball. The probability that $U\in\beta_{0}^{\perp}$ is $0$
then with probability $1$ we have $U\notin\beta_{0}^{\perp}$ and
there is therefore only one vector among $\beta_{0}$ and $-\beta_{0}$
that satisfies $\beta_{0}^{t}U>0$. We call $\beta_{0}$ the vector
satisfying both $\mathcal{R}\left(\beta_{0}\right)$ is the minimum
of $\mathcal{R}\left(\cdot\right)$ and $\beta_{0}^{t}U>0$.\\
Take $\beta\in\Psi_{U}$ and let $\varepsilon\in\left(0,R\right)$,
we are about to control $\mathcal{E}\left(\beta,\beta_{0}\right)$
on $B_{2}\left(\beta_{0},\varepsilon\right)$ and $\left\{ \nu\in B_{2}\left(0,R\right):\beta_{0}^{t}\nu>0\right\} \setminus B_{2}\left(\beta_{0},\varepsilon\right)$
but these two sets may not be included $\Psi_{U}$. To bypass this
issue, remark that the risk $\mathcal{R}$ is symmetric with respect
to $0$. Hence, in the case where $\beta\notin\left\{ \nu\in B_{2}\left(0,R\right):\beta_{0}^{t}\nu>0\right\} $,
we will have $\mathcal{E}\left(\beta,\beta_{0}\right)=\mathcal{E}\left(-\beta,\beta_{0}\right)$
where $-\beta\in\left\{ \nu\in B_{2}\left(0,R\right):\beta_{0}^{t}\nu>0\right\} $.
Consequently, one can always control $\mathcal{E}\left(\cdot,\beta_{0}\right)$
on $\Psi_{U}$ with its values on $\left\{ \nu\in B_{2}\left(0,R\right):\beta_{0}^{t}\nu>0\right\} $,
and without loss of generality we will focus on the control of $\mathcal{E}\left(\cdot,\beta_{0}\right)$
on $\left\{ \nu\in B_{2}\left(0,R\right):\beta_{0}^{t}\nu>0\right\} $.

\uline{Case 1}: $\beta\in B_{2}\left(\beta_{0},\varepsilon\right)$\\
We know from Lemma~\ref{lem:minoration_excess_risk_locally} that
$\forall\beta\in B_{2}\left(\beta_{0},\varepsilon\right)$,
\[
\frac{\mathcal{E}\left(\beta,\beta_{0}\right)}{e^{-\left(\left\Vert a\right\Vert _{2}R-R^{2}/2\right)}\left\Vert \beta-\beta_{0}\right\Vert _{2}^{2}}\geq\frac{1}{16}\left(1+\left(\left\Vert a\right\Vert _{2}-R\right)^{2}\right)-24\left\Vert a\right\Vert ^{4}e^{R^{2}/2}e^{\varepsilon\left\Vert a\right\Vert _{2}}\left\Vert \beta-\beta_{0}\right\Vert _{2}.
\]
When $\left\Vert \beta-\beta_{0}\right\Vert _{2}\leq\frac{1}{2}\frac{\frac{1}{16}\left(1+\left(\left\Vert a\right\Vert _{2}-R\right)^{2}\right)}{24\left\Vert a\right\Vert ^{4}e^{R^{2}/2}e^{\varepsilon\left\Vert a\right\Vert _{2}}}$,
one has
\[
\frac{\mathcal{E}\left(\beta,\beta_{0}\right)}{e^{-\left(\left\Vert a\right\Vert _{2}R-R^{2}/2\right)}\left\Vert \beta-\beta_{0}\right\Vert _{2}^{2}}\geq\frac{1+\left(\left\Vert a\right\Vert _{2}-R\right)^{2}}{16}.
\]
In particular, the latter inequality holds when
\begin{align*}
\varepsilon & \leq\frac{1}{768}\frac{1+\left(\left\Vert a\right\Vert _{2}-R\right)^{2}}{\left\Vert a\right\Vert _{2}^{4}e^{R^{2}/2}}e^{-\varepsilon\left\Vert a\right\Vert _{2}},
\end{align*}
which is satisfied for
\begin{align*}
\varepsilon & \leq\frac{1}{768}\frac{1+\left(\left\Vert a\right\Vert _{2}-R\right)^{2}}{\left\Vert a\right\Vert _{2}^{4}e^{R^{2}/2}}\exp\left(-\frac{1}{384}\frac{1+\left(\left\Vert a\right\Vert _{2}-R\right)^{2}}{\left\Vert a\right\Vert _{2}^{4}e^{R^{2}/2}}\left\Vert a\right\Vert _{2}\right)\\
 & =\frac{1}{768}\frac{1+\left(\left\Vert a\right\Vert _{2}-R\right)^{2}}{\left\Vert a\right\Vert _{2}^{4}}\exp\left(-R^{2}/2-\frac{1+\left(\left\Vert a\right\Vert _{2}-R\right)^{2}}{384\left\Vert a\right\Vert _{2}^{3}e^{R^{2}/2}}\right)=:\varepsilon_{max}.
\end{align*}
Then for all $\beta\in B_{2}\left(\beta_{0},\varepsilon_{max}\right)$,
we have

\[
\frac{\mathcal{E}\left(\beta,\beta_{0}\right)}{\left\Vert \beta-\beta_{0}\right\Vert _{2}^{2}}\geq\frac{1}{32}\left(1+\left(\left\Vert a\right\Vert _{2}-R\right)^{2}\right)e^{-\left(\left\Vert a\right\Vert _{2}R-R^{2}/2\right)}.
\]

\uline{Case }2: $\beta\in\left\{ \nu\in B_{2}\left(0,R\right):\beta_{0}^{t}\nu>0\right\} \setminus B_{2}\left(\beta_{0},\varepsilon_{max}\right)$.

Lemmas~\ref{lem:minimum_of_risk_on_a_ball} and~\ref{lem:minimum_risk_on_a_line}
imply that $\forall\lambda>1,\mathcal{E}\left(\lambda\beta,\beta_{0}\right)<\mathcal{E}\left(\beta,\beta_{0}\right)$
and 
\[
\text{if }\nu\in\left\{ \mu\in B_{2}\left(0,R\right):\left\Vert \mu\right\Vert =\left\Vert \beta\right\Vert \text{ \& }\beta_{0}^{t}\mu>\beta_{0}^{t}\beta\right\} ,\text{ }\mathcal{E}\left(\nu,\beta_{0}\right)<\mathcal{E}\left(\beta,\beta_{0}\right).
\]
 With these two properties, we are always able to control $\mathcal{E}\left(\beta,\beta_{0}\right)$
with another value $\mathcal{E}\left(\nu,\beta_{0}\right)$ where
$\nu\in B_{2}\left(\beta_{0},\varepsilon_{max}\right)$. Indeed, if
$\mathbb{R}\cdot\beta$ intersects $B_{2}\left(\beta_{0},\varepsilon_{max}\right)$,
there there exists $\lambda>1$ such that $\mathcal{E}\left(\beta,\beta_{0}\right)>\mathcal{E}\left(\lambda\beta,\beta_{0}\right)$,
where $\lambda\beta\in B_{2}\left(\beta_{0},\varepsilon_{max}\right)$.
Oherwise, we have $\mathcal{E}\left(\beta,\beta_{0}\right)\geq\mathcal{E}\left(R\frac{\beta}{\left\Vert \beta\right\Vert _{2}},\beta_{0}\right)$
and $\mathcal{E}\left(R\frac{\beta_{0}}{\left\Vert \beta_{0}\right\Vert _{2}},\beta_{0}\right)>\mathcal{E}\left(\beta_{inter},\beta_{0}\right)$
where $\beta_{inter}$ is the rotation of $R\frac{\beta_{0}}{\left\Vert \beta_{0}\right\Vert _{2}}$
towards $\beta_{0}$ so that $\beta_{inter}$ is at the frontier of
$B_{2}\left(\beta_{0},\varepsilon_{max}\right)$. Moreover, $\forall\beta\in\left\{ \nu\in B_{2}\left(0,R\right):\beta_{0}^{t}\nu>0\right\} $
we have $\left\Vert \beta-\beta_{0}\right\Vert _{2}^{2}\leq2R^{2}$.
Consequently, 
\[
\text{for all }\beta\in\left\{ \nu\in B_{2}\left(0,R\right):\beta_{0}^{t}\nu>0\right\} \setminus B_{2}\left(\beta_{0},\varepsilon_{max}\right),
\]
 there exists $\nu\in B_{2}\left(\beta_{0},\varepsilon_{max}\right)$
such that $\left\Vert \nu-\beta_{0}\right\Vert _{2}=\varepsilon_{max}$
and 
\[
\frac{\mathcal{E}\left(\beta,\beta_{0}\right)}{\left\Vert \beta-\beta_{0}\right\Vert _{2}^{2}}\geq\frac{\mathcal{E}\left(\beta,\beta_{0}\right)}{2R^{2}}\geq\frac{\mathcal{E}\left(\nu,\beta_{0}\right)}{2R^{2}}.
\]
Furthermmore, from Case 1 above, we have that $\forall\nu\in B_{2}\left(\beta_{0},\varepsilon_{max}\right)$
such that $\left\Vert \nu-\beta_{0}\right\Vert _{2}=\varepsilon_{max}$,

\[
\mathcal{E}\left(\nu,\beta_{0}\right)\geq\frac{1}{32}\left(1+\left(\left\Vert a\right\Vert _{2}-R\right)^{2}\right)e^{-\left(\left\Vert a\right\Vert _{2}R-R^{2}/2\right)}\varepsilon_{max}^{2}.
\]
Hence, for all $\beta\in\left\{ \nu\in B_{2}\left(0,R\right):\beta_{0}^{t}\nu>0\right\} \setminus B_{2}\left(\beta_{0},\varepsilon_{max}\right)$,
\[
\frac{\mathcal{E}\left(\beta,\beta_{0}\right)}{\left\Vert \beta-\beta_{0}\right\Vert _{2}^{2}}\geq\frac{\left(1+\left(\left\Vert a\right\Vert _{2}-R\right)^{2}\right)e^{-\left(\left\Vert a\right\Vert _{2}R-R^{2}/2\right)}\varepsilon_{max}^{2}}{64R^{2}}.
\]
Finally, from the two cases, we get
\[
\underset{\beta\in\left\{ \nu\in B_{2}\left(0,R\right):\beta_{0}^{t}\nu>0\right\} }{\inf}\frac{\mathcal{E}\left(\beta,\beta_{0}\right)}{\left\Vert \beta-\beta_{0}\right\Vert _{2}^{2}}\geq\frac{\left(1+\left(\left\Vert a\right\Vert _{2}-R\right)^{2}\right)e^{-\left(\left\Vert a\right\Vert _{2}R-R^{2}/2\right)}\varepsilon_{max}^{2}}{64R^{2}}.
\]
Consequently, the result is also true when one takes the infimum over
$\Psi_{U}$:{\small{}
\begin{align*}
\underset{\beta\in\Psi_{U}}{\inf}\frac{\mathcal{E}\left(\beta,\beta_{0}\right)}{\left\Vert \beta-\beta_{0}\right\Vert _{2}^{2}} & \geq\frac{\left(1+\left(\left\Vert a\right\Vert _{2}-R\right)^{2}\right)e^{-\left(\left\Vert a\right\Vert _{2}R-R^{2}/2\right)}}{64R^{2}}\left(\frac{1}{768}\frac{1+\left(\left\Vert a\right\Vert _{2}-R\right)^{2}}{\left\Vert a\right\Vert _{2}^{4}}\exp\left(-R^{2}/2-\frac{1+\left(\left\Vert a\right\Vert _{2}-R\right)^{2}}{384\left\Vert a\right\Vert _{2}^{3}e^{R^{2}/2}}\right)\right)^{2}\\
 & =\frac{\left(1+\left(\left\Vert a\right\Vert _{2}-R\right)^{2}\right)^{3}}{9\cdot2^{22}\left\Vert a\right\Vert _{2}^{8}R^{2}}.\exp\left(-\left\Vert a\right\Vert _{2}R-R^{2}/2-R^{2}-\frac{1+\left(\left\Vert a\right\Vert _{2}-R\right)^{2}}{384\left\Vert a\right\Vert _{2}^{3}e^{R^{2}/2}}\right)\\
 & \geq\frac{\left(\left\Vert a\right\Vert _{2}-R\right)^{6}}{9\cdot2^{22}\left\Vert a\right\Vert _{2}^{8}R^{2}}.\exp\left(-\left\Vert a\right\Vert _{2}R-R^{2}/2-R^{2}-R^{2}/2\right)\\
 & \geq\frac{\left(\left\Vert a\right\Vert _{2}-R\right)^{6}}{9\cdot2^{22}\left\Vert a\right\Vert _{2}^{8}R^{2}}.\exp\left(-\left\Vert a\right\Vert _{2}R-2R^{2}\right).
\end{align*}
}{\small\par}
\end{proof}
We present now the proof of our main result, that is the oracle inequality
stated in Section \ref{sec:An-oracle-inequality}.
\begin{proof}[Proof of Theorem \ref{thm:main_result}]
We know from Lemma~\ref{lem:inegalite_essentielle} that 
\begin{equation}
\mathcal{E}\left(\hat{\beta},\beta_{0}\right)+\lambda\left\Vert \hat{\beta}\right\Vert _{1}\leq\left|V_{n}\left(\hat{\beta}\right)-V_{n}\left(\beta_{0}\right)\right|+\lambda\left\Vert \beta_{0}\right\Vert _{1}\label{eq:essential_inequality}
\end{equation}

We set ourselves in the event $\mathcal{T}$ defined in (\ref{eq:def_T}).
It holds 
\[
\sup_{\beta\in B_{2}\left(0,R\right)}\frac{\left|V_{n}\left(\beta\right)-V_{n}\left(\beta_{0}\right)\right|}{\left\Vert \beta-\beta_{0}\right\Vert _{1}\lor\lambda_{0}}\leq2T\lambda_{0}
\]
 and, as $\hat{\beta}\in B_{2}\left(0,R\right)$, Equation~(\ref{eq:essential_inequality})
gives
\[
\mathcal{E}\left(\hat{\beta},\beta_{0}\right)+\lambda\left\Vert \hat{\beta}\right\Vert _{1}\leq2T\lambda_{0}\left\Vert \hat{\beta}-\beta_{0}\right\Vert _{1}\lor\lambda_{0}+\lambda\left\Vert \beta_{0}\right\Vert _{1}.
\]

Case 1: $\left\Vert \hat{\beta}-\beta_{0}\right\Vert _{1}\lor\lambda_{0}=\lambda_{0}$.
We successively have
\begin{eqnarray*}
\mathcal{E}\left(\hat{\beta},\beta_{0}\right) & \leq & 2T\lambda_{0}^{2}+\lambda\left(\left\Vert \beta_{0}\right\Vert _{1}-\left\Vert \hat{\beta}\right\Vert _{1}\right)\\
 & \leq & 2T\lambda_{0}^{2}+\lambda\left|\left\Vert \beta_{0}\right\Vert _{1}-\left\Vert \hat{\beta}\right\Vert _{1}\right|\\
 & \leq & 2T\lambda_{0}^{2}+\lambda\left\Vert \beta_{0}-\hat{\beta}\right\Vert _{1}.
\end{eqnarray*}
Hence,
\begin{eqnarray*}
\mathcal{E}\left(\hat{\beta},\beta_{0}\right)+2\lambda\left\Vert \beta_{0}-\hat{\beta}\right\Vert _{1} & \leq & 2T\lambda_{0}^{2}+3\lambda\left\Vert \beta_{0}-\hat{\beta}\right\Vert _{1}\\
 & \leq & 2T\lambda_{0}^{2}+3\lambda\lambda_{0}\\
 & \leq & 2\left(T\lambda_{0}+\lambda\right)^{2}.
\end{eqnarray*}
Finally, since $2\left(\lambda-2T\lambda_{0}\right)\leq2\lambda$
and $\left\Vert \hat{\beta}^{S^{c}}\right\Vert _{1}=\left\Vert \beta_{0}^{S^{c}}-\hat{\beta}^{S^{c}}\right\Vert _{1}\leq\left\Vert \beta_{0}-\hat{\beta}\right\Vert _{1}$

\[
\mathcal{E}\left(\hat{\beta},\beta_{0}\right)+2\left(\lambda-2T\lambda_{0}\right)\left\Vert \hat{\beta}^{S^{c}}\right\Vert _{1}\leq3\left(T\lambda_{0}+\lambda\right)^{2}.
\]

Case 2: $\left\Vert \hat{\beta}-\beta_{0}\right\Vert _{1}\lor\lambda_{0}=\left\Vert \hat{\beta}-\beta_{0}\right\Vert _{1}$.
We have $\left\Vert \hat{\beta}\right\Vert _{1}=\left\Vert \hat{\beta}^{S}\right\Vert _{1}+\left\Vert \hat{\beta}^{S^{c}}\right\Vert _{1}$,
$\left\Vert \beta_{0}\right\Vert _{1}=\left\Vert \beta_{0}^{S}\right\Vert _{1}$
and $\left\Vert \hat{\beta}-\beta_{0}\right\Vert _{1}=\left\Vert \hat{\beta}^{S}-\beta_{0}^{S}\right\Vert _{1}+\left\Vert \hat{\beta}^{S^{c}}\right\Vert _{1}=\left\Vert \hat{\beta}^{S}-\beta_{0}\right\Vert _{1}+\left\Vert \hat{\beta}^{S^{c}}\right\Vert _{1}$.
Consequently, it holds successively 

\begin{align*}
\mathcal{E}\left(\hat{\beta},\beta_{0}\right)+\lambda\left\Vert \hat{\beta}\right\Vert _{1} & \leq2T\lambda_{0}\left\Vert \hat{\beta}-\beta_{0}\right\Vert _{1}+\lambda\left\Vert \beta_{0}\right\Vert _{1},\\
\mathcal{E}\left(\hat{\beta},\beta_{0}\right)+\lambda\left\Vert \hat{\beta}^{S}\right\Vert _{1}+\lambda\left\Vert \hat{\beta}^{S^{c}}\right\Vert _{1} & \leq2T\lambda_{0}\left\Vert \hat{\beta}^{S}-\beta_{0}\right\Vert _{1}+2T\lambda_{0}\left\Vert \hat{\beta}^{S^{c}}\right\Vert _{1}+\lambda\left\Vert \beta_{0}^{S}\right\Vert _{1},\\
\mathcal{E}\left(\hat{\beta},\beta_{0}\right)+\lambda\left\Vert \hat{\beta}^{S^{c}}\right\Vert _{1}-2T\lambda_{0}\left\Vert \hat{\beta}^{S^{c}}\right\Vert _{1} & \leq2T\lambda_{0}\left\Vert \hat{\beta}^{S}-\beta_{0}\right\Vert _{1}+\lambda\left\Vert \beta_{0}^{S}\right\Vert _{1}-\lambda\left\Vert \hat{\beta}^{S}\right\Vert _{1},\\
\mathcal{E}\left(\hat{\beta},\beta_{0}\right)+\left(\lambda-2T\lambda_{0}\right)\left\Vert \hat{\beta}^{S^{c}}\right\Vert _{1} & \leq2T\lambda_{0}\left\Vert \hat{\beta}^{S}-\beta_{0}\right\Vert _{1}+\lambda\left\Vert \beta_{0}^{S}-\hat{\beta}^{S}\right\Vert _{1},\\
\mathcal{E}\left(\hat{\beta},\beta_{0}\right)+\left(\lambda-2T\lambda_{0}\right)\left\Vert \hat{\beta}^{S^{c}}\right\Vert _{1} & \leq\left(2T\lambda_{0}+\lambda\right)\left\Vert \beta_{0}-\hat{\beta}^{S}\right\Vert _{1}.
\end{align*}
Since $\beta_{0}-\hat{\beta}^{S}$ has at most $s$ non-zero coordinates,
one has $\left\Vert \beta_{0}-\hat{\beta}^{S}\right\Vert _{1}\leq\sqrt{s}\left\Vert \beta_{0}-\hat{\beta}^{S}\right\Vert _{2}\leq\sqrt{s}\left\Vert \beta_{0}-\hat{\beta}\right\Vert _{2}$.
Hence, for any $c_{0}>0$,

\[
\mathcal{E}\left(\hat{\beta},\beta_{0}\right)+\left(\lambda-2T\lambda_{0}\right)\left\Vert \hat{\beta}^{S^{c}}\right\Vert _{1}\leq\left(T\lambda_{0}+\lambda\right)\sqrt{\frac{s}{c_{0}}}\sqrt{c_{0}}\left\Vert \beta_{0}-\hat{\beta}\right\Vert _{2}.
\]
Now use the fact that $\forall a,b,2ab\leq a^{2}+b^{2}$ to get

\begin{align*}
\mathcal{E}\left(\hat{\beta},\beta_{0}\right)+\left(\lambda-2T\lambda_{0}\right)\left\Vert \hat{\beta}^{S^{c}}\right\Vert _{1} & \leq\left(T\lambda_{0}+\lambda\right)^{2}\frac{s}{2c_{0}}+\frac{c_{0}\left\Vert \beta_{0}-\hat{\beta}\right\Vert _{2}^{2}}{2}.
\end{align*}
So we can use Lemma~\ref{lem:lemme_D01_remi} and have $\mathcal{E}\left(\hat{\beta},\beta_{0}\right)\geq c_{0}\left\Vert \beta_{0}-\hat{\beta}\right\Vert _{2}^{2}$
where $c_{0}=\frac{\left(\left\Vert a\right\Vert _{2}-R\right)^{6}}{9\cdot2^{22}\left\Vert a\right\Vert _{2}^{8}R^{2}}e^{-\left\Vert a\right\Vert _{2}R-2R^{2}}$.
Consequently, for this choice of $c_{0}$,
\begin{align*}
\mathcal{E}\left(\hat{\beta},\beta_{0}\right)+\left(\lambda-2T\lambda_{0}\right)\left\Vert \hat{\beta}^{S^{c}}\right\Vert _{1} & \leq\left(T\lambda_{0}+\lambda\right)^{2}\frac{s}{2c_{0}}+\frac{\mathcal{E}\left(\hat{\beta},\beta_{0}\right)}{2},
\end{align*}
which gives
\[
\mathcal{E}\left(\hat{\beta},\beta_{0}\right)+2\left(\lambda-2T\lambda_{0}\right)\left\Vert \hat{\beta}^{S^{c}}\right\Vert _{1}\leq\left(T\lambda_{0}+\lambda\right)^{2}.\frac{s}{c_{0}}.
\]
Finally, combining the two cases, we obtain
\[
\mathcal{E}\left(\hat{\beta},\beta_{0}\right)+2\left(\lambda-2T\lambda_{0}\right)\left\Vert \hat{\beta}^{S^{c}}\right\Vert _{1}\leq\left(T\lambda_{0}+\lambda\right)^{2}.\max\left(\frac{s}{c_{0}},2\right).
\]
The bound on the probability of the event $\mathcal{T}$ is given
in Theorem \ref{thm:control_proba_main_event}.
\end{proof}

\subsection{Auxiliary results}

Let us first state the following basic lemma, where we compute the
derivatives of the loss and its risk.
\begin{lem}
\label{lem:formula_derivatives} With notations of section~\ref{sec:Notations},
it holds
\begin{eqnarray}
\rho_{\beta}\left(X\right) & = & -\log q_{\beta}\left(X\right)+X^{t}\beta.p_{\beta}\left(X\right)\label{eq:criterion_form_1}\\
\rho_{\beta}\left(X\right) & = & \log\left(1+e^{X^{t}\beta}\right)-\frac{X^{t}\beta e^{X^{t}\beta}}{1+e^{X^{t}\beta}}\label{eq:criterion_form_2}\\
\frac{\partial p_{\beta}\left(X\right)}{\partial\beta_{u}} & = & -X_{u}p_{\beta}\left(X\right)q_{\beta}\left(X\right)\label{eq:p_derivative}\\
\frac{\partial q_{\beta}\left(X\right)}{\partial\beta_{u}} & = & X_{u}p_{\beta}\left(X\right)q_{\beta}\left(X\right)\label{eq:q_derivative}\\
\frac{\partial\rho_{\beta}\left(X\right)}{\partial\beta_{u}} & = & -X^{t}\beta X_{u}p_{\beta}\left(X\right)q_{\beta}\left(X\right)\label{eq:criterion_1st_derivatives}\\
\frac{\partial}{\partial\beta_{v}}\frac{\partial}{\partial\beta_{u}}\rho_{\beta}\left(X\right) & = & =-X_{v}X_{u}\alpha\left(X^{t}\beta\right)\label{eq:criterion_2nd_derivatives}\\
\frac{\partial}{\partial\beta_{w}}\frac{\partial}{\partial\beta_{v}}\frac{\partial}{\partial\beta_{u}}\rho_{\beta}\left(X\right) & = & -X_{w}X_{v}X_{u}\alpha'\left(X^{t}\beta\right)\label{eq:criterion_3rd_derivatives}\\
\left(d_{\beta}\mathcal{R}\right)\left(h\right) & = & \mathbb{E}\left[-X^{t}\beta.p_{\beta}\left(X\right)q_{\beta}\left(X\right)X^{t}h\right]\label{eq:risk_linear_op}\\
\left(d_{\beta}^{2}\mathcal{R}\right)\left(h,k\right) & = & \mathbb{E}\left[X^{t}h.X^{t}k.\alpha\left(X^{t}\beta\right)\right]\label{eq:risk_bilinear_op}\\
\left(d_{\beta}^{3}\mathcal{R}\right)\left(h,k,l\right) & = & \mathbb{E}\left[X^{t}h.X^{t}k.X^{t}l.\alpha'\left(X^{t}\beta\right)\right]\label{eq:risk_trilinear_op}
\end{eqnarray}
\end{lem}
\begin{proof}
Consider $X,\beta\in\mathbb{R}^{d}$, $\rho_{\beta}\left(X\right)$
is defined in section~\ref{sec:Notations}. For simplicity, $p$
and $q$ stand for $p_{\beta}\left(X\right)$ and $q_{\beta}\left(X\right)$
recall that $p_{\beta}\left(X\right)=q_{\beta}\left(X\right)e^{-X\beta}$:
\begin{align*}
\rho_{\beta}\left(X\right) & =-p\log p-q\log q\\
 & =-\left(qe^{-X\beta}\right)\log\left(qe^{-X\beta}\right)-q\log q\\
 & =-qe^{-X^{t}\beta}\log q+X^{t}\beta qe^{-X^{t}\beta}-q\log q\\
 & =-q\left(1+e^{-X^{t}\beta}\right)\log q+X^{t}\beta qe^{-X^{t}\beta}\\
 & =-\log q+X^{t}\beta p\\
 & =\log\left(1+e^{X^{t}\beta}\right)-\frac{X^{t}\beta e^{X^{t}\beta}}{1+e^{X^{t}\beta}}.
\end{align*}
Denote $\beta_{u}$ the $u$-th component of $\beta$. We have,

\begin{align*}
\frac{\partial p_{\beta}\left(X\right)}{\partial\beta_{u}} & =\frac{\partial}{\partial\beta_{u}}\left[\frac{1}{1+e^{X^{t}\beta}}\right]=-\frac{X_{u}e^{X^{t}\beta}}{\left(1+e^{X^{t}\beta}\right)^{2}}\\
 & =-X_{u}p_{\beta}\left(X\right)q_{\beta}\left(X\right)
\end{align*}
and

\begin{align*}
\frac{\partial q_{\beta}\left(X\right)}{\partial\beta_{u}} & =\frac{\partial}{\partial\beta_{u}}\left[1-p_{\beta}\left(X\right)\right]=-\frac{\partial p_{\beta}\left(X\right)}{\partial\beta_{u}}\\
 & =X_{u}p_{\beta}\left(X\right)q_{\beta}\left(X\right).
\end{align*}
Secondly, we use Equation~(\ref{eq:criterion_form_2}) to have

\begin{align*}
\frac{\partial\rho_{\beta}\left(X\right)}{\partial\beta_{u}} & =-\frac{\partial\log q}{\partial\beta_{u}}+\frac{\partial\left(X^{t}\beta p\right)}{\partial\beta_{u}}\\
 & =-\frac{\frac{\partial q}{\partial\beta_{u}}}{q}+\frac{\partial\left(X^{t}\beta\right)}{\partial\beta_{u}}p+X^{t}\beta\frac{\partial p}{\partial\beta_{u}}\\
 & =-\frac{\left(X_{u}pq\right)}{q}+X_{u}p+X^{t}\beta\left(-X_{u}pq\right)\\
 & =-X^{t}\beta X_{u}p_{\beta}\left(X\right)q_{\beta}\left(X\right).
\end{align*}
The second derivatives are
\begin{align*}
\frac{\partial}{\partial\beta_{v}}\frac{\partial}{\partial\beta_{u}}\rho_{\beta}\left(X\right) & =-\frac{\partial\left(X\beta\right)}{\partial\beta_{v}}X_{u}pq-X^{t}\beta X_{u}\frac{\partial p}{\partial\beta_{v}}q-X^{t}\beta X_{u}p\frac{\partial q}{\partial\beta_{v}}\\
 & =-\left(X_{v}\right)X_{u}pq-X^{t}\beta X_{u}\left(-X_{v}pq\right)q-X^{t}\beta X_{u}p\left(X_{v}qp\right)\\
 & =-X_{v}X_{u}pq\left(1-X^{t}\beta\left(q-p\right)\right)\\
 & =-X_{v}X_{u}\frac{e^{X^{t}\beta}}{\left(1+e^{X^{t}\beta}\right)^{2}}\left(1+X^{t}\beta\left(\frac{1-e^{X^{t}\beta}}{1+e^{X^{t}\beta}}\right)\right)\\
 & =X_{v}X_{u}\alpha\left(X^{t}\beta\right).
\end{align*}
The third derivatives are

\begin{align*}
\frac{\partial}{\partial\beta_{w}}\frac{\partial}{\partial\beta_{v}}\frac{\partial}{\partial\beta_{u}}\rho_{\beta}\left(X\right) & =X_{v}X_{u}\frac{\partial}{\partial\beta_{w}}\left[\alpha\left(X^{t}\beta\right)\right]\\
 & =X_{w}X_{v}X_{u}\alpha'\left(X^{t}\beta\right).
\end{align*}
As the derivatives are uniformly bounded with respect to $\beta$,
the theorem of derivation under integral can be applied and it comes
that $\forall h,k,l,\in\mathbb{R}^{d}$,
\begin{align*}
\left(d_{\beta}\mathcal{R}\right)\left(h\right) & =\mathbb{E}\left[-X^{t}\beta.p_{\beta}\left(X\right)q_{\beta}\left(X\right)X^{t}h\right],\\
\left(d_{\beta}^{2}\mathcal{R}\right)\left(h,k\right) & =\mathbb{E}\left[X^{t}h.X^{t}k.\alpha\left(X^{t}\beta\right)\right],\\
\left(d_{\beta}^{3}\mathcal{R}\right)\left(h,k,l\right) & =\mathbb{E}\left[X^{t}h.X^{t}k.X^{t}l.\alpha'\left(X^{t}\beta\right)\right].
\end{align*}
\end{proof}
\begin{lem}
\label{lem:technical_lemma_same_sign_as_mu}Take $r>0$. For any function
$g$ odd on $\mathbb{R}$, positive on $(0,+\infty)$ and when $U$
is a symetric random variable with a density $\gamma$ decreasing
on $\mathbb{R}_{+}$, the quantity $\mathbb{E}\left[g\left(\mu+rU\right)\right]$
has the sign of $\mu$.
\end{lem}
\begin{proof}
Take $r,\mu>0$, $U_{1}$ and $U_{2}$ two independent copies of $U$.
It holds
\begin{align*}
\mathbb{E}\left[g\left(\mu+rU\right)\right] & =\mathbb{E}\left[g\left(\mu+rU_{1}\right)I_{U_{1}>0}+g\left(\mu-rU_{2}\right)I_{U_{2}>0}\right]\\
 & =\mathbb{E}\left[g\left(\mu+rU_{1}\right)I_{U_{1}>0}+\right.\\
 & \qquad\left.g\left(\mu-rU_{2}\right)I_{0<U_{2}<\frac{\mu}{r}}+g\left(\mu-rU_{2}\right)I_{\frac{\mu}{r}<U_{2}<2\frac{\mu}{r}}+g\left(\mu-rU_{2}\right)I_{2\frac{\mu}{r}<U_{2}}\right]\\
\\
 & =\mathbb{E}\left[g\left(\mu+rU_{1}\right)I_{U_{1}>0}+g\left(\mu-rU_{2}\right)I_{2\frac{\mu}{r}<U_{2}}\right.\\
 & \qquad\left.+g\left(\mu-rU_{1}\right)I_{0<U_{2}<\frac{\mu}{r}}+g\left(\mu-rU_{2}\right)I_{\frac{\mu}{r}<U_{2}<2\frac{\mu}{r}}\right]\\
\\
 & =\mathbb{E}\left[g\left(\mu+rU_{1}\right)I_{U_{1}>0}+g\left(\mu-rU_{2}\right)I_{2\frac{\mu}{r}<U_{2}}\right]\\
 & \qquad+\mathbb{E}\left[g\left(\mu-rU_{1}\right)I_{0<U_{2}<\frac{\mu}{r}}+g\left(\mu-rU_{2}\right)I_{\frac{\mu}{r}<U_{2}<2\frac{\mu}{r}}\right].
\end{align*}
Let us compute the sign of $\mathbb{E}\left[g\left(\mu+rU_{1}\right)I_{U_{1}>0}+g\left(\mu-rU_{2}\right)I_{2\frac{\mu}{r}<U_{2}}\right]$:

\begin{align*}
 & \mathbb{E}\left[g\left(\mu+rU_{1}\right)I_{U_{1}>0}+g\left(\mu-rU_{2}\right)I_{2\frac{\mu}{r}<U_{2}}\right]\\
 & =\int_{0}^{\infty}g\left(\mu+rx\right)\gamma\left(x\right)dx+\underset{x=y+\frac{2\mu}{r}}{\underbrace{\int_{2\frac{\mu}{r}}^{\infty}g\left(\mu-rx\right)\gamma\left(x\right)dx}}\\
 & =\int_{0}^{\infty}g\left(\mu+rx\right)\gamma\left(x\right)dx+\int_{0}^{\infty}g\left(\mu-r\left(y+\frac{2\mu}{r}\right)\right)\gamma\left(y+\frac{2\mu}{r}\right)dy\\
 & =\int_{0}^{\infty}g\left(\mu+rx\right)\gamma\left(x\right)dx+\int_{0}^{\infty}g\left(-ry-\mu\right)\gamma\left(y+\frac{2\mu}{r}\right)dy\\
 & =\int_{0}^{\infty}g\left(\mu+rx\right)\gamma\left(x\right)dx+\int_{0}^{\infty}-g\left(\mu+ry\right)\gamma\left(y+\frac{2\mu}{r}\right)dy\\
 & =\int_{0}^{\infty}\underset{>0}{\underbrace{g\left(\mu+rx\right)}}\underset{>0}{\underbrace{\left[\gamma\left(x\right)-\gamma\left(x+\frac{2\mu}{r}\right)\right]}}dx\\
 & >0.
\end{align*}
Let us now compute the sign of $\mathbb{E}\left[g\left(\mu-rU_{1}\right)I_{0<U_{2}<\frac{\mu}{r}}+g\left(\mu-rU_{2}\right)I_{\frac{\mu}{r}<U_{2}<2\frac{\mu}{r}}\right]$:

\begin{align*}
 & \mathbb{E}\left[g\left(\mu-rU_{1}\right)I_{0<U_{2}<\frac{\mu}{r}}+g\left(\mu-rU_{2}\right)I_{\frac{\mu}{r}<U_{2}<2\frac{\mu}{r}}\right]\\
 & =\int_{0}^{\frac{\mu}{r}}g\left(\mu-rx\right)\gamma\left(x\right)dx+\underset{x=\frac{2\mu}{r}-y}{\underbrace{\int_{\frac{\mu}{r}}^{\frac{2\mu}{r}}g\left(\mu-rx\right)\gamma\left(x\right)dx}}\\
 & =\int_{0}^{\frac{\mu}{r}}g\left(\mu-rx\right)\gamma\left(x\right)dx+\int_{0}^{\frac{\mu}{r}}g\left(\mu-r\left(\frac{2\mu}{r}-y\right)\right)\gamma\left(\frac{2\mu}{r}-y\right)dy\\
 & =\int_{0}^{\frac{\mu}{r}}g\left(\mu-rx\right)\gamma\left(x\right)dx+\int_{0}^{\frac{\mu}{r}}g\left(ry-\mu\right)\gamma\left(\frac{2\mu}{r}-y\right)dy\\
 & =\int_{0}^{\frac{\mu}{r}}g\left(\mu-rx\right)\gamma\left(x\right)dx+\int_{0}^{\frac{\mu}{r}}-g\left(\mu-ry\right)\gamma\left(\frac{2\mu}{r}-y\right)dy\\
 & =\int_{0}^{\frac{\mu}{r}}\underset{>0}{\underbrace{g\underset{>0}{\underbrace{\left(\mu-rx\right)}}}}\underset{>0}{\underbrace{\left[\gamma\left(x\right)-\gamma\underset{\frac{2\mu}{r}-x>\frac{\mu}{r}>x}{\underbrace{\left(\frac{2\mu}{r}-x\right)}}\right]}}dx\\
 & >0
\end{align*}
Hence $\mathbb{E}\left[g\left(\mu+rU\right)\right]>0$. If $\mu<0$,
then one has $\mathbb{E}\left[g\left(\mu+rU\right)\right]=-\mathbb{E}\left[g\left(-\mu-rU\right)\right]$
and the previous result applies since $-\mu>0$ and $-U\sim U$. Thus
we find that if $\mu<0,\mathbb{E}\left[g\left(\mu+rU\right)\right]<0$.
\end{proof}
\begin{thm}
\label{thm:control_proba_main_event}Set $\Theta=B_{2}\left(0,R\right)$,
$M_{n}=\left\Vert a\right\Vert _{\infty}+\sqrt{2\log d}+\sqrt{2\log\left(1+n\right)}$
and 
\[
\lambda_{0}=3n^{-1/2}LM_{n}\left(5\sqrt{3\log\left(2d\right)}\log n+4\right).
\]
It holds $\forall n\geq2$, $\forall T\geq1$, 
\[
P\left(\sup_{\beta\in\Theta}\frac{\left|V_{n}\left(\beta\right)-V_{n}\left(\beta_{0}\right)\right|}{\left\Vert \beta-\beta_{0}\right\Vert _{1}\lor\lambda_{0}}>2T\lambda_{0}\right)\leq\frac{3}{4}\log\left(\frac{4R^{2}nd}{L^{2}M_{n}^{2}}\right)\exp\left(-21\left(T-1\right)^{2}\log\left(2d\right)\log^{2}n\right)+\frac{1}{25T^{2}\log\left(2d\right)n\log^{2}n}.
\]
\end{thm}
\begin{proof}
First, the triangular inequality gives 
\[
\left|V_{n}\left(\beta\right)-V_{n}\left(\beta_{0}\right)\right|\leq\left|V_{n}^{trunc}\left(\beta\right)-V_{n}^{trunc}\left(\beta_{0}\right)\right|+\left|V_{n}^{trunc}\left(\beta\right)-V_{n}^{trunc}\left(\beta_{0}\right)-\left(V_{n}\left(\beta\right)-V_{n}\left(\beta_{0}\right)\right)\right|,
\]
and since $\forall a,b,t>0$ on has ``$a+b>2t$'' implies ``either
$a>t$ or $b>t$'', the probability of interest can be controlled
as follows:
\begin{align*}
P\left(\sup_{\beta\in\Theta}\frac{\left|V_{n}\left(\beta\right)-V_{n}\left(\beta_{0}\right)\right|}{\left\Vert \beta-\beta_{0}\right\Vert _{1}\lor\lambda_{0}}>2T\lambda_{0}\right) & \leq P\left(\sup_{\beta\in\Theta}\frac{\left|V_{n}^{trunc}\left(\beta\right)-V_{n}^{trunc}\left(\beta_{0}\right)\right|}{\left\Vert \beta-\beta_{0}\right\Vert _{1}\lor\lambda_{0}}>T\lambda_{0}\right)\\
 & \quad+P\left(\sup_{\beta\in\Theta}\frac{\left|V_{n}^{trunc}\left(\beta\right)-V_{n}^{trunc}\left(\beta_{0}\right)-\left(V_{n}\left(\beta\right)-V_{n}\left(\beta_{0}\right)\right)\right|}{\left\Vert \beta-\beta_{0}\right\Vert _{1}\lor\lambda_{0}}>T\lambda_{0}\right).
\end{align*}
Apply now Lemma~\ref{lem:control_V_tail} to have:

\begin{align*}
P\left(\sup_{\beta\in\Theta}\frac{\left|V_{n}\left(\beta\right)-V_{n}\left(\beta_{0}\right)\right|}{\left\Vert \beta-\beta_{0}\right\Vert _{1}\lor\lambda_{0}}>2T\lambda_{0}\right) & \leq P\left(\sup_{\beta\in\Theta}\frac{\left|V_{n}^{trunc}\left(\beta\right)-V_{n}^{trunc}\left(\beta_{0}\right)\right|}{\left\Vert \beta-\beta_{0}\right\Vert _{1}\lor\lambda_{0}}>T\lambda_{0}\right)\\
 & \quad+P\left(\frac{1}{n}\sum_{i=1}^{n}F\left(X^{(i)}\right)>\frac{T\lambda_{0}}{L}\right).
\end{align*}
From Lemmas~\ref{lem:peeling_device} and~\ref{lem:control_of_F(X)},
we get{\scriptsize{}
\begin{align*}
P\left(\sup_{\beta\in\Theta}\frac{\left|V_{n}\left(\beta\right)-V_{n}\left(\beta_{0}\right)\right|}{\left\Vert \beta-\beta_{0}\right\Vert _{1}\lor\lambda_{0}}>2T\lambda_{0}\right) & \leq\frac{3}{4}\log\left(\frac{4R^{2}nd}{L^{2}M_{n}^{2}}\right)\exp\left(-21\left(T-1\right)^{2}\log\left(2d\right)\log^{2}n\right)+4L^{2}\frac{M_{n}^{2}+\left\Vert a\right\Vert _{\infty}+1}{n^{2}\lambda_{0}^{2}T^{2}}.
\end{align*}
}Furthermore,
\begin{align*}
4L^{2}\frac{M_{n}^{2}+\left\Vert a\right\Vert _{\infty}+1}{n^{2}\lambda_{0}^{2}T^{2}} & \leq4L^{2}\frac{M_{n}^{2}+\left\Vert a\right\Vert _{\infty}+1}{n^{2}\frac{9L^{2}M_{n}^{2}\left(5\sqrt{3\log\left(2d\right)}\log n+4\right)^{2}}{n}T^{2}}\\
 & \leq4\frac{1+\frac{\left\Vert a\right\Vert _{\infty}+1}{M_{n}^{2}}}{9\times25\left(3\log\left(2d\right)\log^{2}n\right)nT^{2}}\\
 & \leq\frac{1}{25T^{2}\log\left(2d\right)n\log^{2}n}.
\end{align*}
Hence,{\scriptsize{}
\begin{align*}
P\left(\sup_{\beta\in\Theta}\frac{\left|V_{n}\left(\beta\right)-V_{n}\left(\beta_{0}\right)\right|}{\left\Vert \beta-\beta_{0}\right\Vert _{1}\lor\lambda_{0}}>2T\lambda_{0}\right) & \leq\frac{3}{4}\log\left(\frac{4R^{2}nd}{L^{2}M_{n}^{2}}\right)\exp\left(-21\left(T-1\right)^{2}\log\left(2d\right)\log^{2}n\right)+\frac{1}{25T^{2}\log\left(2d\right)n\log^{2}n}.
\end{align*}
}{\scriptsize\par}
\end{proof}
\begin{lem}
\label{lem:inegalite_essentielle}Recall that $\beta_{0}=\underset{\beta\in\Psi_{U}}{\arg\min}\left\{ \mathcal{R}\left(\beta\right)\right\} $
and $\hat{\beta}:=\underset{\beta\in\Psi_{U}}{\arg\min}\left\{ \mathcal{\hat{R}}_{n}\left(\beta\right)+\lambda\left\Vert \beta\right\Vert _{1}\right\} $.
It holds
\[
\mathcal{E}\left(\hat{\beta},\beta_{0}\right)+\lambda\left\Vert \hat{\beta}\right\Vert _{1}\leq\left|V_{n}\left(\hat{\beta}\right)-V_{n}\left(\beta_{0}\right)\right|+\lambda\left\Vert \beta_{0}\right\Vert _{1}.
\]
\end{lem}
\begin{proof}
By definition of $\hat{\beta}$, we have: 
\[
\mathcal{\hat{R}}_{n}\left(\hat{\beta}\right)+\lambda\left\Vert \hat{\beta}\right\Vert _{1}\leq\mathcal{\hat{R}}_{n}\left(\beta_{0}\right)+\lambda\left\Vert \beta_{0}\right\Vert _{1}.
\]
Injecting the excess risk on both sides of the inequality gives
\[
\mathcal{E}\left(\hat{\beta},\beta_{0}\right)+\lambda\left\Vert \hat{\beta}\right\Vert _{1}\leq\mathcal{R}(\hat{\beta})-\mathcal{R}(\beta_{0})+\mathcal{\hat{R}}_{n}\left(\beta_{0}\right)-\mathcal{\hat{R}}_{n}\left(\hat{\beta}\right)+\lambda\left\Vert \beta_{0}\right\Vert _{1}
\]
Then the result comes from the inequality: 
\[
\mathcal{R}(\hat{\beta})-\mathcal{R}(\beta_{0})+\mathcal{\hat{R}}_{n}\left(\beta_{0}\right)-\mathcal{\hat{R}}_{n}\left(\hat{\beta}\right)\leq\left|V_{n}\left(\hat{\beta}\right)-V_{n}\left(\beta_{0}\right)\right|.
\]
\end{proof}

\subsection{Some further technical lemmas }
\begin{lem}
\label{lem:orthogonality_implies_indep}Assuming $a\in\mathbb{R}^{d}$,
$Z\sim a+N$, $N\sim\mathcal{N}\left(0,I_{d}\right)$, $\beta_{0}=Ra/\left\Vert a\right\Vert _{2}$,
and $h_{\perp}\in\beta_{0}^{\perp}$, then $h_{\perp}^{t}Z$ and $Z^{t}\beta_{0}$
are two independent Gaussian variables.
\end{lem}
\begin{proof}
Note that $h_{\perp}^{t}a=0.$ We have

\begin{align*}
Cov\left(h_{\perp}^{t}Z,Z^{t}\beta_{0}\right) & =\mathbb{E}\left[\left(h_{\perp}^{t}Z-\mathbb{E}\left[h_{\perp}^{t}Z\right]\right)\left(Z^{t}\beta_{0}-\mathbb{E}\left[Z^{t}\beta_{0}\right]\right)\right]\\
 & =\mathbb{E}\left[\left(h_{\perp}^{t}\left(a+N\right)-\mathbb{E}\left[h_{\perp}^{t}\left(a+N\right)\right]\right)\left(\left(a+N\right)^{t}\beta_{0}-\mathbb{E}\left[\left(a+N\right)^{t}\beta_{0}\right]\right)\right]\\
 & =\mathbb{E}\left[\left(h_{\perp}^{t}N-\mathbb{E}\left[h_{\perp}^{t}N\right]\right)\left(a^{t}\beta_{0}+N^{t}\beta_{0}-\mathbb{E}\left[a^{t}\beta_{0}+N^{t}\beta_{0}\right]\right)\right]\\
 & =\mathbb{E}\left[\left(h_{\perp}^{t}N-h_{\perp}^{t}\mathbb{E}\left[N\right]\right)\left(N^{t}\beta_{0}-\mathbb{E}\left[N^{t}\right]\beta_{0}\right)\right]\\
 & =h_{\perp}^{t}\mathbb{E}\left[NN^{t}\right]\beta_{0}\\
 & =h_{\perp}^{t}\beta_{0}\\
 & =0.
\end{align*}
\end{proof}
\begin{lem}
\label{lem:risk_bilinear_simple_form}With $Z$, $\beta_{0}$, $\alpha$
and $\mathcal{R}$ usual notations and for all $h:=h_{\parallel}+h_{\perp}\in Vect(\beta_{0})\oplus\beta_{0}^{\perp}$
such that $\left\Vert h\right\Vert _{2}=1$ and with $\eta:=\left\Vert h_{\parallel}\right\Vert _{2}$,
we have
\begin{align*}
\left(d_{\beta_{0}}^{2}\mathcal{R}\right)\left(h,h\right) & =\eta^{2}\mathbb{E}\left[\frac{1}{R^{2}}\left(Z^{t}\beta_{0}\right)^{2}\alpha\left(Z^{t}\beta_{0}\right)\right]+\left(1-\eta^{2}\right)\mathbb{E}\left[\alpha\left(Z^{t}\beta_{0}\right)\right].
\end{align*}
\end{lem}
\begin{proof}
We computed $d_{\beta_{0}}^{2}\mathcal{R}$ in Equation~(\ref{eq:risk_bilinear_op})
of Lemma \ref{lem:formula_derivatives}. The function $\alpha$ (see
Section \ref{sec:Notations}) is even, so the entries of the Hessian
$d_{\beta_{0}}^{2}\mathcal{R}$ are
\begin{align*}
\forall u,v\in\left\llbracket 1,d\right\rrbracket ,\left(d_{\beta_{0}}^{2}\mathcal{R}\right)_{u,v} & =\mathbb{E}\left[\left(\varepsilon Z_{v}\right)\left(\varepsilon Z_{u}\right)\alpha\left(\varepsilon Z^{t}\beta_{0}\right)\right]\\
 & =\mathbb{E}\left[Z_{v}Z_{u}\alpha\left(Z^{t}\beta_{0}\right)\right].
\end{align*}
Now, let us use the decomposition $h=h_{\parallel}+h_{\perp}$ and
remark that $h_{\parallel}=\epsilon\eta\frac{\beta_{0}}{\left\Vert \beta_{0}\right\Vert _{2}}=\epsilon\frac{\eta}{R}\beta_{0}$
with $\epsilon\in\left\{ -1,1\right\} $. It comes

\begin{align*}
\left(d_{\beta_{0}}^{2}\mathcal{R}\right)\left(h,h\right) & =\mathbb{E}\left[\left(h^{t}Z\right)^{2}\alpha\left(Z^{t}\beta_{0}\right)\right]\\
 & =\mathbb{E}\left[\left(\left(h_{\parallel}+h_{\perp}\right)^{t}Z\right)^{2}\alpha\left(Z^{t}\beta_{0}\right)\right]\\
 & =\mathbb{E}\left[\left(\left(h_{\parallel}^{t}Z\right)^{2}+2.h_{\parallel}^{t}Z.h_{\perp}^{t}Z+\left(h_{\perp}^{t}Z\right)^{2}\right)\alpha\left(Z^{t}\beta_{0}\right)\right]\\
 & =\mathbb{E}\left[\left(h_{\parallel}^{t}Z\right)^{2}\alpha\left(Z^{t}\beta_{0}\right)\right]+2\mathbb{E}\left[h_{\parallel}^{t}Z.h_{\perp}^{t}Z.\alpha\left(Z^{t}\beta_{0}\right)\right]+\mathbb{E}\left[\left(h_{\perp}^{t}Z\right)^{2}\alpha\left(Z^{t}\beta_{0}\right)\right]\\
 & =\mathbb{E}\left[\left(\epsilon\frac{\eta}{R}\beta_{0}^{t}Z\right)^{2}\alpha\left(Z^{t}\beta_{0}\right)\right]+2\mathbb{E}\left[h_{\parallel}^{t}Z.h_{\perp}^{t}Z.\alpha\left(Z^{t}\beta_{0}\right)\right]+\mathbb{E}\left[\left(h_{\perp}^{t}Z\right)^{2}\alpha\left(Z^{t}\beta_{0}\right)\right]\\
 & =\eta^{2}\mathbb{E}\left[\frac{1}{R^{2}}\left(Z^{t}\beta_{0}\right)^{2}\alpha\left(Z^{t}\beta_{0}\right)\right]+2\mathbb{E}\left[h_{\parallel}^{t}Z.h_{\perp}^{t}Z.\alpha\left(Z^{t}\beta_{0}\right)\right]+\mathbb{E}\left[\left(h_{\perp}^{t}Z\right)^{2}\alpha\left(Z^{t}\beta_{0}\right)\right].
\end{align*}
Also remark that $h_{\perp}^{t}Z$ and $Z^{t}\beta_{0}$ are Gaussian
random variables, because $Z$ is a Gaussian vector, that are independent
due to lemma~\ref{lem:orthogonality_implies_indep}.
\begin{align*}
\left(d_{\beta_{0}}^{2}\mathcal{R}\right)\left(h,h\right) & =\eta^{2}\mathbb{E}\left[\frac{1}{R^{2}}\left(Z^{t}\beta_{0}\right)^{2}\alpha\left(Z^{t}\beta_{0}\right)\right]+2\mathbb{E}\left[h_{\parallel}^{t}Z.h_{\perp}^{t}Z.\alpha\left(Z^{t}\beta_{0}\right)\right]+\mathbb{E}\left[\left(h_{\perp}^{t}Z\right)^{2}\alpha\left(Z^{t}\beta_{0}\right)\right]\\
 & =\eta^{2}\mathbb{E}\left[\frac{1}{R^{2}}\left(Z^{t}\beta_{0}\right)^{2}\alpha\left(Z^{t}\beta_{0}\right)\right]+2\mathbb{E}\left[h_{\parallel}^{t}Z.\alpha\left(Z^{t}\beta_{0}\right)\right]\underset{=0}{\underbrace{\mathbb{E}\left[h_{\perp}^{t}Z\right]}}+\mathbb{E}\left[\left(h_{\perp}^{t}Z\right)^{2}\right]\mathbb{E}\left[\alpha\left(Z^{t}\beta_{0}\right)\right]\\
 & =\eta^{2}\mathbb{E}\left[\frac{1}{R^{2}}\left(Z^{t}\beta_{0}\right)^{2}\alpha\left(Z^{t}\beta_{0}\right)\right]+\mathbb{E}\left[\left(h_{\perp}^{t}Z\right)^{2}\right]\mathbb{E}\left[\alpha\left(Z^{t}\beta_{0}\right)\right].
\end{align*}
Note that $\mathbb{E}\left[\left(h_{\perp}^{t}Z\right)^{2}\right]=\mathbb{E}\left[\left(h_{\perp}^{t}\left(a+N\right)\right)^{2}\right]=\mathbb{E}\left[\left(h_{\perp}^{t}N\right)^{2}\right]=h_{\perp}^{t}\mathbb{E}\left[NN^{t}\right]h_{\perp}=1-\eta^{2}$
hence:
\begin{align*}
\left(d_{\beta_{0}}^{2}\mathcal{R}\right)\left(h,h\right) & =\eta^{2}\mathbb{E}\left[\frac{1}{R^{2}}\left(Z^{t}\beta_{0}\right)^{2}\alpha\left(Z^{t}\beta_{0}\right)\right]+\left(1-\eta^{2}\right)\mathbb{E}\left[\alpha\left(Z^{t}\beta_{0}\right)\right].
\end{align*}
\end{proof}
\begin{lem}
\label{lem:control_of_A_B}For all $h:=h_{\parallel}+h_{\perp}\in Vect(\beta_{0})\oplus\beta_{0}^{\perp}$
such that $\left\Vert h\right\Vert _{2}=1$ and with $\eta:=\left\Vert h_{\parallel}\right\Vert _{2}$.
The two following quantities 
\begin{align*}
A & :=\eta^{2}\mathbb{E}\left[\frac{1}{R^{2}}\left(Z^{t}\beta_{0}\right)^{2}\alpha\left(Z^{t}\beta_{0}\right)\mathbb{I}_{\left\{ Z^{t}\beta_{0}>x_{1}\right\} }\right]+\left(1-\eta^{2}\right)\mathbb{E}\left[\alpha\left(Z^{t}\beta_{0}\right)\mathbb{I}_{\left\{ Z^{t}\beta_{0}>x_{1}\right\} }\right]\\
B & :=\eta^{2}\mathbb{E}\left[\frac{1}{R^{2}}\left(Z^{t}\beta_{0}\right)^{2}\alpha\left(Z^{t}\beta_{0}\right)\mathbb{I}_{\left\{ -x_{1}<Z^{t}\beta_{0}<x_{1}\right\} }\right]+\left(1-\eta^{2}\right)\mathbb{E}\left[\alpha\left(Z^{t}\beta_{0}\right)\mathbb{I}_{\left\{ -x_{1}<Z^{t}\beta_{0}<x_{1}\right\} }\right]
\end{align*}
are controlled by
\begin{align*}
A & >\left(\eta^{2}\frac{x_{1}^{2}}{R^{2}}+1-\eta^{2}\right)\left[R+\left[R\left(\left\Vert a\right\Vert _{2}-R\right)-\left(x_{1}+\frac{8}{100}\right)\right]G\left(\frac{x_{1}}{R}+R-\left\Vert a\right\Vert _{2}\right)\right]\gamma\left(\left\Vert a\right\Vert _{2}-\frac{x_{1}}{R}\right)e^{-x_{1}},\\
B & \leq\frac{1}{4}\left(\eta^{2}\frac{x_{1}^{2}}{R^{2}}+1-\eta^{2}\right)\left(\Phi^{c}\left(\left\Vert a\right\Vert _{2}-\frac{x_{1}}{R}\right)-\Phi^{c}\left(\left\Vert a\right\Vert _{2}+\frac{x_{1}}{R}\right)\right).
\end{align*}
\end{lem}
\begin{proof}
Let us first give an upper bound for the quantity $B$. Recall that,
from Lemma \ref{lem:Study_alpha} we have $-\alpha(x)\in\left(0,\frac{1}{4}\right)$
for $x\in\left[-x_{1},x_{1}\right]$. It holds
\begin{align*}
B & =-\eta^{2}\mathbb{E}\left[\frac{1}{R^{2}}\left(Z^{t}\beta_{0}\right)^{2}\alpha\left(Z^{t}\beta_{0}\right)\mathbb{I}_{\left\{ -x_{1}<Z^{t}\beta_{0}<x_{1}\right\} }\right]-\left(1-\eta^{2}\right)\mathbb{E}\left[\alpha\left(Z^{t}\beta_{0}\right)\mathbb{I}_{\left\{ -x_{1}<Z^{t}\beta_{0}<x_{1}\right\} }\right]\\
 & \leq\eta^{2}\mathbb{E}\left[\frac{1}{R^{2}}\left(Z^{t}\beta_{0}\right)^{2}\frac{1}{4}\mathbb{I}_{\left\{ -x_{1}<Z^{t}\beta_{0}<x_{1}\right\} }\right]+\left(1-\eta^{2}\right)\mathbb{E}\left[\frac{1}{4}\mathbb{I}_{\left\{ -x_{1}<Z^{t}\beta_{0}<x_{1}\right\} }\right]\\
 & =\frac{1}{4}\left(\eta^{2}\frac{x_{1}^{2}}{R^{2}}+1-\eta^{2}\right)\mathbb{P}\left[-x_{1}<Z^{t}\beta_{0}<x_{1}\right]\\
 & =\frac{1}{4}\left(\eta^{2}\frac{x_{1}^{2}}{R^{2}}+1-\eta^{2}\right)\mathbb{P}\left[-x_{1}<R\left\Vert a\right\Vert _{2}+R\mathcal{N}\left(0,1\right)<x_{1}\right]\\
 & =\frac{1}{4}\left(\eta^{2}\frac{x_{1}^{2}}{R^{2}}+1-\eta^{2}\right)\mathbb{P}\left[-\frac{x_{1}}{R}-\left\Vert a\right\Vert _{2}<\mathcal{N}\left(0,1\right)<\frac{x_{1}}{R}-\left\Vert a\right\Vert _{2}\right]\\
 & =\frac{1}{4}\left(\eta^{2}\frac{x_{1}^{2}}{R^{2}}+1-\eta^{2}\right)\mathbb{P}\left[\left\Vert a\right\Vert _{2}-\frac{x_{1}}{R}<\mathcal{N}\left(0,1\right)<\frac{x_{1}}{R}+\left\Vert a\right\Vert _{2}\right]\\
 & =\frac{1}{4}\left(\eta^{2}\frac{x_{1}^{2}}{R^{2}}+1-\eta^{2}\right)\left(\Phi^{c}\left(\left\Vert a\right\Vert _{2}-\frac{x_{1}}{R}\right)-\Phi^{c}\left(\left\Vert a\right\Vert _{2}+\frac{x_{1}}{R}\right)\right).
\end{align*}
Let us now turn to the lower bound for the quantity $A$: 
\begin{align*}
A & =\eta^{2}\mathbb{E}\left[\frac{1}{R^{2}}\left(Z^{t}\beta_{0}\right)^{2}\alpha\left(Z^{t}\beta_{0}\right)\mathbb{I}_{\left\{ Z^{t}\beta_{0}>x_{1}\right\} }\right]+\left(1-\eta^{2}\right)\mathbb{E}\left[\alpha\left(Z^{t}\beta_{0}\right)\mathbb{I}_{\left\{ Z^{t}\beta_{0}>x_{1}\right\} }\right]\\
 & \geq\eta^{2}\mathbb{E}\left[\frac{1}{R^{2}}x_{1}^{2}\alpha\left(Z^{t}\beta_{0}\right)\mathbb{I}_{\left\{ Z^{t}\beta_{0}>x_{1}\right\} }\right]+\left(1-\eta^{2}\right)\mathbb{E}\left[\alpha\left(Z^{t}\beta_{0}\right)\mathbb{I}_{\left\{ Z^{t}\beta_{0}>x_{1}\right\} }\right]\\
 & =\eta^{2}\frac{x_{1}^{2}}{R^{2}}\mathbb{E}\left[\alpha\left(Z^{t}\beta_{0}\right)\mathbb{I}_{\left\{ Z^{t}\beta_{0}>x_{1}\right\} }\right]+\left(1-\eta^{2}\right)\mathbb{E}\left[\alpha\left(Z^{t}\beta_{0}\right)\mathbb{I}_{\left\{ Z^{t}\beta_{0}>x_{1}\right\} }\right]\\
 & =\left(\eta^{2}\frac{x_{1}^{2}}{R^{2}}+1-\eta^{2}\right)\mathbb{E}\left[\alpha\left(Z^{t}\beta_{0}\right)\mathbb{I}_{\left\{ Z^{t}\beta_{0}>x_{1}\right\} }\right].
\end{align*}
We need now to control $\mathbb{E}\left[\alpha\left(Z^{t}\beta_{0}\right)\mathbb{I}_{\left\{ Z^{t}\beta_{0}>x_{1}\right\} }\right]$
from below. We first use Lemma~\ref{lem:beginning_minoration} to
get:

{\small{}
\begin{align*}
 & \mathbb{E}\left[\alpha\left(Z^{t}\beta_{0}\right)\mathbb{I}_{\left\{ Z^{t}\beta_{R}>x_{1}\right\} }\right]\\
 & \geq\int_{x_{1}}^{\infty}\left(\left(x-x_{1}-\frac{8}{100}\right)e^{-x}\frac{e^{-\left(x/R-\left\Vert a\right\Vert _{2}\right)^{2}/2}}{\sqrt{2\pi R^{2}}}\right)dx\\
 & =\int_{x_{1}}^{\infty}xe^{-x}\frac{e^{-\left(x/R-\left\Vert a\right\Vert \right)^{2}/2}}{\sqrt{2\pi R^{2}}}dx-\left(x_{1}+\frac{8}{100}\right)\int_{x_{1}}^{\infty}e^{-x}\frac{e^{-\left(x/R-\left\Vert a\right\Vert \right)^{2}/2}}{\sqrt{2\pi R^{2}}}dx.
\end{align*}
}Using the notations of Lemmas~\ref{lem:good_asympt_behaviour} and~\ref{lem:good_asympt_behaviour-1}
we obtain,

\begin{align}
\mathbb{E}\left[\alpha\left(Z^{t}\beta_{0}\right)\mathbb{I}_{\left\{ Z^{t}\beta_{0}>x_{1}\right\} }\right] & \geq J_{a,R}\left(1,x_{1}\right)-\left(x_{1}+\frac{8}{100}\right)K_{a,R}\left(1,x_{1}\right).\label{eq:beginning_minoration}
\end{align}
Hence, Lemmas~\ref{lem:good_asympt_behaviour} and~\ref{lem:good_asympt_behaviour-1}
give:{\small{}
\begin{align*}
 & \mathbb{E}\left[\alpha\left(Z^{t}\beta_{0}\right)\mathbb{I}_{\left\{ Z^{t}\beta_{0}>x_{1}\right\} }\right]\\
 & \geq R\left(1+\left(\left\Vert a\right\Vert -R\right)G\left(\frac{x_{1}}{R}+R-\left\Vert a\right\Vert \right)\right)\gamma\left(\frac{x_{1}}{R}-\left\Vert a\right\Vert \right)e^{-x_{1}}-\left(x_{1}+\frac{8}{100}\right)\gamma\left(\frac{x_{1}}{R}-\left\Vert a\right\Vert \right)G\left(\frac{x_{1}}{R}+R-\left\Vert a\right\Vert \right)e^{-x_{1}}\\
 & \geq\left[R+\left[R\left(\left\Vert a\right\Vert -R\right)-\left(x_{1}+\frac{8}{100}\right)\right]G\left(\frac{x_{1}}{R}+R-\left\Vert a\right\Vert \right)\right]\gamma\left(\left\Vert a\right\Vert -\frac{x_{1}}{R}\right)e^{-x_{1}}.
\end{align*}
}Finally,

\begin{align*}
A & >\left(\eta^{2}\frac{x_{1}^{2}}{R^{2}}+1-\eta^{2}\right)\left[R+\left[R\left(\left\Vert a\right\Vert -R\right)-\left(x_{1}+\frac{8}{100}\right)\right]G\left(\frac{x_{1}}{R}+R-\left\Vert a\right\Vert \right)\right]\gamma\left(\left\Vert a\right\Vert -\frac{x_{1}}{R}\right)e^{-x_{1}}.
\end{align*}
\end{proof}
\begin{lem}
\label{lem:risk_bilinear_values_control}Take $a\in\mathbb{R}^{d}$,
$R,\nu>0$ and $\beta_{0}:=R\frac{a}{\left\Vert a\right\Vert _{2}}$
if inequality
\begin{align}
R\left(1-\left(R-\left\Vert a\right\Vert _{2}+\frac{x_{1}+\frac{8}{100}}{R}\right)G\left(\frac{x_{1}}{R}+R-\left\Vert a\right\Vert _{2}\right)\right) & \geq\left(1+\nu\right)\frac{e^{x_{1}}}{4}G\left(\left\Vert a\right\Vert _{2}-\frac{x_{1}}{R}\right)\label{eq:final_condition_on_a_R_nu}
\end{align}
is true, then for all $h:=h_{\parallel}+h_{\perp}\in Vect(\beta_{0})\oplus\beta_{0}^{\perp}$
such that $\left\Vert h\right\Vert _{2}=1$ and $\eta:=\left\Vert h_{\parallel}\right\Vert _{2}$,
it also holds
\[
\left(d_{\beta_{0}}^{2}\mathcal{R}\right)\left(h,h\right)>\frac{\nu}{4}\left(\eta^{2}\frac{x_{1}^{2}}{R^{2}}+1-\eta^{2}\right)\left(\Phi^{c}\left(\left\Vert a\right\Vert _{2}-\frac{x_{1}}{R}\right)-\Phi^{c}\left(\left\Vert a\right\Vert _{2}+\frac{x_{1}}{R}\right)\right).
\]
\end{lem}
\begin{proof}
Recall that$\beta_{0}:=Ra/\left\Vert a\right\Vert _{2}$. We proved
in Lemma~\ref{lem:risk_bilinear_simple_form}, that $\left(d_{\beta_{0}}^{2}\mathcal{R}\right)\left(h,h\right)$
is given by the following formula:
\begin{align*}
\left(d_{\beta_{0}}^{2}\mathcal{R}\right)\left(h,h\right) & =\eta^{2}\mathbb{E}\left[\frac{1}{R^{2}}\left(Z^{t}\beta_{0}\right)^{2}\alpha\left(Z^{t}\beta_{R}\right)\right]+\left(1-\eta^{2}\right)\mathbb{E}\left[\alpha\left(Z^{t}\beta_{0}\right)\right].
\end{align*}
We know from Lemma~\ref{lem:Study_alpha} that $\alpha$ is non-positive
on the interval $\left[-x_{1},x_{1}\right]$ and positive otherwise.
Consequently, we study the sign of $\left(d_{\beta_{0}}^{2}\mathcal{R}\right)\left(h,h\right)$
on the partition $\mathbb{R}=\left(-\infty-,x_{1}\right)\bigcup\left[-x_{1},x_{1}\right]\bigcup\left(x_{1},\infty\right)$:
\begin{align*}
\left(d_{\beta_{0}}^{2}\mathcal{R}\right)\left(h,h\right) & =\eta^{2}\mathbb{E}\left[\frac{1}{R^{2}}\left(Z^{t}\beta_{0}\right)^{2}\alpha\left(Z^{t}\beta_{0}\right)\right]+\left(1-\eta^{2}\right)\mathbb{E}\left[\alpha\left(Z^{t}\beta_{0}\right)\right]\\
\\
 & =\underset{A}{\underbrace{\eta^{2}\mathbb{E}\left[\frac{1}{R^{2}}\left(Z^{t}\beta_{0}\right)^{2}\alpha\left(Z^{t}\beta_{0}\right)\mathbb{I}_{\left\{ Z^{t}\beta_{0}>x_{1}\right\} }\right]+\left(1-\eta^{2}\right)\mathbb{E}\left[\alpha\left(Z^{t}\beta_{0}\right)\mathbb{I}_{\left\{ Z^{t}\beta_{0}>x_{1}\right\} }\right]}}\\
 & \quad+\underset{-B}{\underbrace{\eta^{2}\mathbb{E}\left[\frac{1}{R^{2}}\left(Z^{t}\beta_{0}\right)^{2}\alpha\left(Z^{t}\beta_{0}\right)\mathbb{I}_{\left\{ -x_{1}<Z^{t}\beta_{0}<x_{1}\right\} }\right]+\left(1-\eta^{2}\right)\mathbb{E}\left[\alpha\left(Z^{t}\beta_{0}\right)\mathbb{I}_{\left\{ -x_{1}<Z^{t}\beta_{0}<x_{1}\right\} }\right]}}\\
 & \quad+\eta^{2}\mathbb{E}\left[\frac{1}{R^{2}}\left(Z^{t}\beta_{0}\right)^{2}\alpha\left(Z^{t}\beta_{0}\right)\mathbb{I}_{\left\{ Z^{t}\beta_{0}<x_{1}\right\} }\right]+\left(1-\eta^{2}\right)\mathbb{E}\left[\alpha\left(Z^{t}\beta_{0}\right)\mathbb{I}_{\left\{ Z^{t}\beta_{0}<x_{1}\right\} }\right]
\end{align*}
We have found in Lemma~\ref{lem:control_of_A_B} two quantities $a>0$
and $b>0$ such that $A>a$ and $b\geq B$:
\begin{align*}
a:= & \left(\eta^{2}\frac{x_{1}^{2}}{R^{2}}+1-\eta^{2}\right)\left[R+\left[R\left(\left\Vert a\right\Vert _{2}-R\right)-\left(x_{1}+\frac{8}{100}\right)\right]G\left(\frac{x_{1}}{R}+R-\left\Vert a\right\Vert _{2}\right)\right]\gamma\left(\left\Vert a\right\Vert _{2}-\frac{x_{1}}{R}\right)e^{-x_{1}},\\
b:= & \frac{1}{4}\left(\eta^{2}\frac{x_{1}^{2}}{R^{2}}+1-\eta^{2}\right)\left(\Phi^{c}\left(\left\Vert a\right\Vert _{2}-\frac{x_{1}}{R}\right)-\Phi^{c}\left(\left\Vert a\right\Vert _{2}+\frac{x_{1}}{R}\right)\right).
\end{align*}
 If $a>\left(1+\nu\right)b$ for some $\nu>0$ then we have $\left(d_{\beta_{0}}^{2}\mathcal{R}\right)\left(h,h\right)>A-B>a-b>\nu b$.
As 
\[
b<\frac{1}{4}\left(\eta^{2}\frac{x_{1}^{2}}{R^{2}}+1-\eta^{2}\right)\Phi^{c}\left(\left\Vert a\right\Vert _{2}-\frac{x_{1}}{R}\right),
\]
the condition ``$a>\left(1+\nu\right)b$'' is satisfied when these
successive conditions are true:{\small{}
\begin{multline*}
\left(\eta^{2}\frac{x_{1}^{2}}{R^{2}}+1-\eta^{2}\right)\left[R+\left[R\left(\left\Vert a\right\Vert _{2}-R\right)-\left(x_{1}+\frac{8}{100}\right)\right]G\left(\frac{x_{1}}{R}+R-\left\Vert a\right\Vert _{2}\right)\right]\gamma\left(\left\Vert a\right\Vert _{2}-\frac{x_{1}}{R}\right)e^{-x_{1}}\\
>\left(1+\nu\right)\frac{1}{4}\left(\eta^{2}\frac{x_{1}^{2}}{R^{2}}+1-\eta^{2}\right)\Phi^{c}\left(\left\Vert a\right\Vert _{2}-\frac{x_{1}}{R}\right)
\end{multline*}
}{\small\par}

(simplify $\left(\eta^{2}\frac{x_{1}^{2}}{R^{2}}+1-\eta^{2}\right)$and
$R$ in factor in the left-hand side){\small{}
\begin{multline*}
R\left[1-\left(R-\left\Vert a\right\Vert _{2}+\frac{x_{1}+\frac{8}{100}}{R}\right)G\left(\frac{x_{1}}{R}+R-\left\Vert a\right\Vert _{2}\right)\right]\gamma\left(\left\Vert a\right\Vert _{2}-\frac{x_{1}}{R}\right)\\
>\left(1+\nu\right)\frac{e^{x_{1}}}{4}\Phi^{c}\left(\left\Vert a\right\Vert _{2}-\frac{x_{1}}{R}\right)
\end{multline*}
}{\small\par}

(divide by $\gamma\left(\left\Vert a\right\Vert _{2}-\frac{x_{1}}{R}\right)$
and make Mill's ratio appear)

{\small{}
\begin{align*}
R\left(1-\left(R-\left\Vert a\right\Vert _{2}+\frac{x_{1}+\frac{8}{100}}{R}\right)G\left(\frac{x_{1}}{R}+R-\left\Vert a\right\Vert _{2}\right)\right) & \geq\left(1+\nu\right)\frac{e^{x_{1}}}{4}G\left(\left\Vert a\right\Vert _{2}-\frac{x_{1}}{R}\right).
\end{align*}
}To conclude, when the latter inequality is true, one has $\left(d_{\beta_{0}}^{2}\mathcal{R}\right)\left(h,h\right)>\nu b$.
\end{proof}
\begin{lem}
\label{lem:Study_alpha}Study of $\alpha(x)=-\frac{e^{x}}{\left(1+e^{x}\right)^{2}}\left(1+x\frac{1-e^{x}}{1+e^{x}}\right)$.
At $x=0$, $\alpha(0)=-\frac{1}{4}$ is a global minimum, $x_{\alpha_{max}}\in\left[2,3\right]$
is the positive real where $\alpha$ is maximal with value $\alpha_{max}$,
its derivative is bounded $\left\Vert \alpha'\right\Vert _{\infty}\leq0.22$
and by definition of $x_{1}$ (see Section~\ref{sec:Notations}),
$\alpha(x_{1})=0$ with 
\begin{equation}
x_{1}\approx1.54340463.\label{eq:approx_x1}
\end{equation}

\begin{table}[h]
\begin{centering}
\begin{tabular}{|c|c|c|c|c|c|c|c|c|c|c|c|c|c|}
\hline 
x & $0$ &  & $x_{2}$ &  & $1$ &  & $x_{1}$ &  & $2$ &  & $x_{\alpha_{max}}$ &  & $\infty$\tabularnewline
\hline 
\hline 
sign of $f''$ & $-$ & $-$ & $-$ & $-$ &  & $-$ & $-$ & $-$ & $-$ & $-$ & $-$ & $-$ & $-$\tabularnewline
\hline 
variations of $f'$ & $1$ & $\searrow$ & $0$ & $\searrow$ &  & $\searrow$ & $\searrow$ & $\searrow$ & $\searrow$ & $\searrow$ & $\searrow$ & $\searrow$ & $-\infty$\tabularnewline
\hline 
sign of $f'$ & $+$ & $+$ & $0$ & $-$ &  & $-$ & $-$ & $-$ & $-$ & $-$ & $-$ & $-$ & $-$\tabularnewline
\hline 
variations of $f$ & $2$ & $\nearrow$ & $f\left(x_{2}\right)$ & $\searrow$ & 2 & $\searrow$ & $0$ & $\searrow$ & $3-e^{2}$ & $\searrow$ & $\searrow$ & $\searrow$ & $-\infty$\tabularnewline
\hline 
sign of $f$ & $+$ & $+$ & $+$ & $+$ & $+$ & $+$ & $0$ & $-$ & $-$ & $-$ & $-$ & $-$ & $-$\tabularnewline
\hline 
sign of $\alpha$ & $-\frac{1}{4}$ & $-$ & $-$ & $-$ & $-$ & $-$ & $0$ & $+$ & $+$ & $+$ & $\alpha_{max}$ & $+$ & $0$\tabularnewline
\hline 
\end{tabular}
\par\end{centering}
\caption{sign and variation table of $f$, sign table of $\alpha$\label{tab:sign-and-variation-alpha}}
\end{table}

\end{lem}
\begin{proof}
First, remark that $\forall x>0$,
\begin{align*}
\alpha(x)\geq0 & \Leftrightarrow1+x\frac{1-e^{x}}{1+e^{x}}\leq0\\
 & \Leftrightarrow x\left(1-e^{x}\right)\leq-\left(1+e^{x}\right)\\
 & \Leftrightarrow1+e^{x}-xe^{x}+x\leq0\\
 & \Leftrightarrow f(x)\leq0.
\end{align*}
We study $f:x\mapsto1+e^{x}-xe^{x}+x$ for $x\in\mathbb{R}_{+}$ since
$\alpha$ is even. First of all, $f'\left(x\right)=1-xe^{x}$ and
$f''\left(x\right)=-\left(x+1\right)e^{x}$ which gives the sign and
variation table~\ref{tab:sign-and-variation-alpha}. It is obvious
that there exists $x_{1}>0$ such that $f\left(x_{1}\right)=0$. Set
$p_{x}=\left(1-e^{x}\right)^{-1}$ and $q_{x}=e^{x}\left(1-e^{x}\right)^{-1}$.
Note that $p_{x}-q_{x}=\frac{1-e^{x}}{1+e^{x}}=-\tanh\left(\frac{x}{2}\right)$
and $p_{x}q_{x}=\frac{1}{4}\left(\left(p_{x}+q_{x}\right)^{2}-\left(p_{x}-q_{x}\right)^{2}\right)=\frac{1}{4}\left(1-\tanh^{2}\left(\frac{x}{2}\right)\right)$.





\begin{align}
\alpha'(x) & =\frac{d}{dx}\left[-p_{x}q_{x}\left(1+x\left(p_{x}-q_{x}\right)\right)\right]\nonumber \\
 & =-\frac{d}{dx}\left[p_{x}\right]q_{x}\left(1+x\left(p_{x}-q_{x}\right)\right)-p_{x}\frac{d}{dx}\left[q_{x}\right]\left(1+x\left(p_{x}-q_{x}\right)\right)-p_{x}q_{x}\frac{d}{dx}\left[1+x\left(p_{x}-q_{x}\right)\right]\nonumber \\
 & =p_{x}q_{x}q_{x}\left(1+x\left(p_{x}-q_{x}\right)\right)-p_{x}p_{x}q_{x}\left(1+x\left(p_{x}-q_{x}\right)\right)-p_{x}q_{x}\left[\left(p_{x}-q_{x}\right)+x\frac{d}{dx}\left[\left(p_{x}-q_{x}\right)\right]\right]\nonumber \\
 & =p_{x}q_{x}\left[q_{x}\left(1+x\left(p_{x}-q_{x}\right)\right)-p_{x}\left(1+x\left(p_{x}-q_{x}\right)\right)-\left(p_{x}-q_{x}\right)-x\left(-p_{x}q_{x}-p_{x}q_{x}\right)\right]\nonumber \\
 & =p_{x}q_{x}\left[-\left(p_{x}-q_{x}\right)\left(1+x\left(p_{x}-q_{x}\right)\right)-\left(p_{x}-q_{x}\right)+2xp_{x}q_{x}\right]\nonumber \\
 & =p_{x}q_{x}\left[2xp_{x}q_{x}-\left(p_{x}-q_{x}\right)\left(2+x\left(p_{x}-q_{x}\right)\right)\right]\label{eq:alpha_derivative}\\
 & =p_{x}q_{x}\left[\frac{x}{2}\left(1-\tanh^{2}\left(\frac{x}{2}\right)\right)+\tanh\left(\frac{x}{2}\right)\left(2-x\tanh\left(\frac{x}{2}\right)\right)\right]\nonumber \\
 & =p_{x}q_{x}\left[\frac{x}{2}\left(1-3\tanh^{2}\left(\frac{x}{2}\right)\right)+2\tanh\left(\frac{x}{2}\right)\right]\nonumber \\
\alpha'(x) & =\frac{1}{4}\left(1-\tanh^{2}\left(\frac{x}{2}\right)\right)\left[\frac{x}{2}\left(1-3\tanh^{2}\left(\frac{x}{2}\right)\right)+2\tanh\left(\frac{x}{2}\right)\right].\nonumber 
\end{align}

One can see on Figure~\ref{fig:plot_of_alpha_prime} that the maximum
of $\alpha$ is attained at $2\leq x_{\alpha_{max}}\leq3$. The function
$\alpha$ is Lipschitz and one can see graphically on Figure~\ref{fig:plot_of_alpha_prime}
that $\left\Vert \alpha'\right\Vert _{\infty}\leq0.22$.
\end{proof}
\begin{lem}
\label{lem:alpha_concave_x1_3}$\alpha:x\mapsto-\frac{e^{x}}{\left(1+e^{x}\right)^{2}}\left(1+x\frac{1-e^{x}}{1+e^{x}}\right)$
is concave on $\left[x_{1},3\right]$.
\end{lem}
\begin{proof}
The shape of $\alpha$ on $[x_{1},3]$ can be seen on figure~\ref{fig:plot_alpha_x1_3}.
\begin{figure}
\includegraphics[width=0.4\paperwidth]{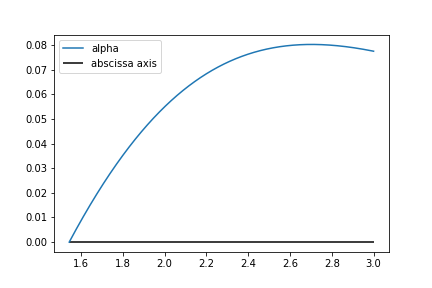}\caption{Plot of $\alpha$ on $[x_{1},3]$. 
\label{fig:plot_alpha_x1_3}}
\end{figure}

We use the following compact notations: $p=\frac{1}{1+e^{x}}$, $q=1-p$,
hence $\alpha(x)=-pq\left(1+x\left(p-q\right)\right)$. Recall that
$\frac{dp}{dx}=-pq$, $\frac{dq}{dx}=pq$ and that $\forall x>0,p<q$.
We proved in Equation~(\ref{eq:alpha_derivative}) that
\begin{align*}
\alpha'(x) & =pq\left[\left(q-p\right)\left(2+x\left(p-q\right)\right)+2xpq\right]\\
 & =p(1-p)\left[\left(1-2p\right)\left(2+x\left(p-q\right)\right)+2xp(1-p)\right]
\end{align*}

In this proof we will also need the variations of $\varpi:x\mapsto1+x\left(p-q\right)$:

\begin{align*}
\varpi'(x) & =\frac{d}{dx}\left(1+x\left(p-q\right)\right)\\
 & =\left(p-q\right)+x\frac{d}{dx}\left(p-q\right)\\
 & =\left(p-q\right)+x\left(-pq-pq\right)\\
 & =\left(p-q\right)-2xpq\\
 & <0
\end{align*}

We will want the sign of $\frac{d^{2}\alpha}{dx^{2}}$. First remark
that 
\begin{align*}
\frac{d}{dx}\left[\frac{\alpha'}{pq}\right] & =\frac{d\alpha'}{dx}\frac{1}{pq}+\alpha'\frac{d}{dx}\left[\frac{1}{pq}\right]\\
 & =\frac{d\alpha'}{dx}\frac{1}{pq}+\alpha'\frac{-1}{\left(pq\right)^{2}}\frac{d\left(pq\right)}{dx}\\
 & =\alpha''\frac{1}{pq}+\alpha'\frac{-1}{\left(pq\right)^{2}}\left(-pqq+ppq\right)\\
\frac{d}{dx}\left[\frac{\alpha'}{pq}\right] & =\alpha''\frac{1}{pq}-\alpha'\frac{1}{pq}\left(p-q\right)
\end{align*}

aalgebraic rearrangment give $\alpha''=p(1-p)\frac{d}{dx}\left[\frac{\alpha'}{pq}\right]+\alpha'\left(2p-1\right)$.
Now compute what is still missing:

\begin{align*}
\frac{d}{dx}\left[\frac{\alpha'}{pq}\right] & =\frac{d}{dx}\left[\frac{pq\left[\left(q-p\right)\left(2+x\left(p-q\right)\right)+2xpq\right]}{pq}\right]\\
 & =\frac{d}{dx}\left[\left(q-p\right)\left(2+x\left(p-q\right)\right)+2xpq\right]\\
 & =\frac{d}{dx}\left(q-p\right)\left(2+x\left(p-q\right)\right)+\left(q-p\right)\frac{d}{dx}\left(2+x\left(p-q\right)\right)+2\frac{d}{dx}\left(xpq\right)\\
 & =\left(pq+pq\right)\left(2+x\left(p-q\right)\right)+\left(q-p\right)\left[\left(p-q\right)+x\left(-pq-pq\right)\right]+2\left(pq-xpqq+xppq\right)\\
 & =2pq\left(2+x\left(p-q\right)\right)+\left(q-p\right)\left[\left(p-q\right)-2xpq\right]+2pq\left(1+x\left(p-q\right)\right)\\
 & =2pq\left(3+2x\left(p-q\right)\right)+\left(q-p\right)\left(p-q\right)-2xpq\left(q-p\right)\\
 & =2pq\left(3+2x\left(p-q\right)\right)-\left(q-p\right)^{2}+2xpq\left(p-q\right)\\
 & =2pq\left(3+2x\left(p-q\right)+x\left(p-q\right)\right)-\left(q-p\right)^{2}\\
 & =6pq\left(1+x\left(p-q\right)\right)-\left(q-p\right)^{2}\\
 & =6pq\varpi(x)-\left(p-q\right)^{2}\\
\frac{d}{dx}\left[\frac{\alpha'}{pq}\right] & =6p(1-p)\varpi(x)-\left(1-2p\right)^{2}
\end{align*}

-Case $x\in\left[x_{1},2\right]$:

we have $p\in\left[0.11,0.18\right]$, hence $1-2p\in\left[0.64,0.78\right]$,
$x\mapsto p$ is decreasing and $p\mapsto p(1-p)$ is increasing on
this intervalles of interest then $p(1-p)\in\left[p_{x=2}(1-p_{x=2}),p_{x=x_{1}}(1-p_{x=x_{1}})\right]\subset[0.09,0.15]$

then $\varpi$ is strictly decreasing and $\varpi(x)\in\left[\varpi(2),\varpi(x_{1})\right]\subset\left[-0.53,0\right]$
($0$ occurs because by definition $x_{1}$ is such that $0=\alpha(x_{1})=pq\varpi(x_{1})$),
all intervalle put together gives in case $x\in\left[x_{1},2\right]$:
\begin{align*}
\alpha'(x) & =\underset{\geq0.09}{\underbrace{p(1-p)}}\left[\underset{\geq0.64}{\underbrace{\left(1-2p\right)}}\left(1+\underset{\geq-0.53}{\underbrace{1+x\left(p-q\right)}}\right)+2\underset{\geq1.5435}{\underbrace{x_{1}}}\underset{\geq0.09}{\underbrace{p(1-p)}}\right]\\
 & \geq0.052
\end{align*}

\begin{align*}
\alpha'(x) & =\underset{\leq0.15}{\underbrace{p(1-p)}}\left[\underset{\leq0.78}{\underbrace{\left(1-2p\right)}}\left(1+\underset{\leq0}{\underbrace{1+x\left(p-q\right)}}\right)+2\underset{\leq1.5436}{\underbrace{x_{1}}}\underset{\leq0.15}{\underbrace{p(1-p)}}\right]\\
 & \leq0.19
\end{align*}

it holds
\begin{align*}
\frac{d}{dx}\left[\frac{\alpha'}{pq}\right] & =2p(1-p)\underset{\leq0}{\underbrace{\varpi(x)}}-\underset{\geq0.4}{\underbrace{\underset{\geq0.64}{\left(\underbrace{1-2p}\right)^{2}}}}\\
 & \leq-0.4
\end{align*}

\begin{align*}
\frac{d}{dx}\left[\frac{\alpha'}{pq}\right] & =2\underset{\leq0.15}{\underbrace{p(1-p)}}\underset{\geq-0.53}{\underbrace{\varpi(x)}}-\underset{\leq0.61}{\underbrace{\underset{\leq0.78}{\left(\underbrace{1-2p}\right)^{2}}}}\\
 & \geq-0.769
\end{align*}

And finally

\begin{align*}
\alpha''(x) & =\underset{\geq0.09}{\underbrace{p(1-p)}}\underset{\leq-0.4}{\underbrace{\frac{d}{dx}\left[\frac{\alpha'}{pq}\right]}}+\underset{\geq0.059}{\underbrace{\alpha'(x)}}\underset{\leq-0.64}{\underbrace{\left(1-2p\right)}}\\
 & \leq-0.07376
\end{align*}
\begin{align*}
\alpha''(x) & =\underset{\leq0.15}{\underbrace{pq}}\underset{\geq-0.769}{\underbrace{\frac{d}{dx}\left[\frac{\alpha'}{pq}\right]}}+\underset{\leq0.19}{\underbrace{\alpha'(x)}}\underset{\geq-0.78}{\underbrace{\left(1-2p\right)}}\\
 & \geq-0.26355
\end{align*}

$\alpha$ is concave on $\left[x_{1},2\right]$.

-We do the same in the case $x\in\left[2,2.5\right]$:

we have $p\in\left[0.075,0.12\right]$, hence $1-2p\in\left[0.76,0.85\right]$,
$x\mapsto p$ is decreasing and $p\mapsto p(1-p)$ is increasing on
this intervalles of interest then $p(1-p)\in\left[p_{x=2.5}(1-p_{x=2.5}),p_{x=2}(1-p_{x=2})\right]\subset[0.069,0.11]$,
$\varpi(x)\in\left[\varpi(3),\varpi(2)\right]\subset\left[-1.125,-0.52\right]$
and

\begin{align*}
\alpha'(x) & =\underset{\geq0.069}{\underbrace{p(1-p)}}\left[\underset{\geq0.76}{\underbrace{\left(1-2p\right)}}\left(1+\underset{\geq-1.125}{\underbrace{1+x\left(p-q\right)}}\right)+2\underset{\geq1.5435}{\underbrace{x_{1}}}\underset{\geq0.069}{\underbrace{p(1-p)}}\right]\\
 & \geq0.008
\end{align*}

\begin{align*}
\alpha'(x) & =\underset{\leq0.12}{\underbrace{p(1-p)}}\left[\underset{\leq0.85}{\underbrace{\left(1-2p\right)}}\left(1+\underset{\le-0.52}{\underbrace{1+x\left(p-q\right)}}\right)+2\underset{\leq1.5436}{\underbrace{x_{1}}}\underset{\leq0.12}{\underbrace{p(1-p)}}\right]\\
 & \leq0.094
\end{align*}

We now have
\begin{align*}
\frac{d}{dx}\left[\frac{\alpha'}{pq}\right] & =2\underset{\geq0.069}{\underbrace{p(1-p)}}\underset{\leq-0.52}{\underbrace{\varpi(x)}}-\underset{\geq0.5776}{\underbrace{\underset{\geq0.76}{\left(\underbrace{1-2p}\right)^{2}}}}\\
 & \leq-0.64936\leq-0.65
\end{align*}

\begin{align*}
\frac{d}{dx}\left[\frac{\alpha'}{pq}\right] & =2\underset{\leq0.11}{\underbrace{p(1-p)}}\underset{\geq-0.53}{\underbrace{\varpi(x)}}-\underset{\leq0.7225}{\underbrace{\underset{\leq0.85}{\left(\underbrace{1-2p}\right)^{2}}}}\\
 & \geq-0.8391\geq-0.84
\end{align*}

And finally
\begin{align*}
\alpha''(x) & =\underset{\geq0.069}{\underbrace{p(1-p)}}\underset{\leq-0.65}{\underbrace{\frac{d}{dx}\left[\frac{\alpha'}{pq}\right]}}+\underset{\geq0.008}{\underbrace{\alpha'(x)}}\underset{\leq-0.76}{\underbrace{\left(1-2p\right)}}\\
 & \leq-0.05093
\end{align*}
\begin{align*}
\alpha''(x) & =\underset{\leq0.11}{\underbrace{pq}}\underset{\geq-0.84}{\underbrace{\frac{d}{dx}\left[\frac{\alpha'}{pq}\right]}}+\underset{\leq0.094}{\underbrace{\alpha'(x)}}\underset{\geq-0.85}{\underbrace{\left(1-2p\right)}}\\
 & \geq-0.1723
\end{align*}

Consequently $\alpha$ is concave on $\left[2,2.5\right]$

-We do the same in the case $x\in\left[2.5,3\right]$:

in that case $p\in\left[0.047,0.076\right]$, hence $1-2p\in\left[0.848,0.906\right]$,
$x\mapsto p$ is decreasing and $p\mapsto p(1-p)$ is increasing on
this intervalles of interest then $p(1-p)\in\left[p_{x=3}(1-p_{x=3}),p_{x=2.5}(1-p_{x=2.5})\right]\subset[0.0447,0.071]$
and $\varpi(x)\in\left[\varpi(3),\varpi(2)\right]\subset\left[-1.72,-1.12\right]$.

\begin{align*}
\alpha'(x) & =\underset{\geq0.0447}{\underbrace{p(1-p)}}\left[\underset{\geq0.848}{\underbrace{\left(1-2p\right)}}\left(1+\underset{\geq-1.72}{\underbrace{1+x\left(p-q\right)}}\right)+2\underset{\geq1.5435}{\underbrace{x_{1}}}\underset{\geq0.0447}{\underbrace{p(1-p)}}\right]\\
 & \geq-0.02113
\end{align*}

\begin{align*}
\alpha'(x) & =\underset{\leq0.071}{\underbrace{p(1-p)}}\left[\underset{\leq0.906}{\underbrace{\left(1-2p\right)}}\left(1+\underset{\leq-1.12}{\underbrace{1+x\left(p-q\right)}}\right)+2\underset{\leq1.5436}{\underbrace{x_{1}}}\underset{\leq0.071}{\underbrace{p(1-p)}}\right]\\
 & \leq0.0079
\end{align*}

We now have
\begin{align*}
\frac{d}{dx}\left[\frac{\alpha'}{pq}\right] & =2\underset{\geq0.0447}{\underbrace{p(1-p)}}\underset{\leq-1.12}{\underbrace{\varpi(x)}}-\underset{\geq0.5776}{\underbrace{\underset{\geq0.848}{\left(\underbrace{1-2p}\right)^{2}}}}\\
 & \leq-0.83
\end{align*}

\begin{align*}
\frac{d}{dx}\left[\frac{\alpha'}{pq}\right] & =2\underset{\leq0.076}{\underbrace{p(1-p)}}\underset{\geq-1,72}{\underbrace{\varpi(x)}}-\underset{\leq0.7225}{\underbrace{\underset{\leq0.906}{\left(\underbrace{1-2p}\right)^{2}}}}\\
 & \geq-1,082
\end{align*}

Concerning $\alpha''$, since $\alpha'\in\left[-0.02113,0.0079\right]$,
the reasonning with an intervalle containing $0$ is a bit different:
$p(1-p)\frac{d}{dx}\left[\frac{\alpha'}{pq}\right]\in\left[-0.077,-0.039\right]$
and $\alpha'(x)\left(1-2p\right)\in[-0.0072,0.0191]$, consequently
\begin{align*}
\alpha''(x) & =p(1-p)\frac{d}{dx}\left[\frac{\alpha'}{pq}\right]+\alpha'(x)\left(1-2p\right)\\
 & \in\left[-0.082,-0.0199\right]
\end{align*}

Consequently $\alpha$ is concave on $\left[2.5,3\right]$
\end{proof}
\begin{lem}
\label{lem:comparison_alpha_minus_phi_to_0}$\forall x\geq3$, $\alpha(x)-\varphi(x)\geq xe^{-x}\left(\frac{x_{1}+0.08-1}{x}-4e^{-x}\right)$
where $\alpha:x\mapsto-\frac{e^{x}}{\left(1+e^{x}\right)^{2}}\left(1+x\frac{1-e^{x}}{1+e^{x}}\right)$
and $\varphi:x\mapsto\left(x-x_{1}-0.08\right)e^{-x}$.
\end{lem}
\begin{proof}
Let us study $\alpha-\varphi$
\begin{align*}
\alpha(x)-\varphi(x) & =-\frac{e^{x}}{\left(1+e^{x}\right)^{2}}\left(1+x\frac{1-e^{x}}{1+e^{x}}\right)-\left(x-x_{1}-0.08\right)e^{-x}\\
 & =-\frac{e^{-2x}}{\left(1+e^{-x}\right)^{3}}\left(1+e^{x}+x-xe^{x}\right)-\left(x-x_{1}-0.08\right)e^{-x}\\
 & =\frac{xe^{-2x}}{\left(1+e^{-x}\right)^{3}}\left(e^{x}-\frac{e^{x}}{x}-1-\frac{1}{x}\right)-x\left(1-\frac{x_{1}+0.08}{x}\right)e^{-x}\\
 & =xe^{-x}\left[\frac{e^{-x}}{\left(1+e^{-x}\right)^{3}}\left(e^{x}-\frac{e^{x}}{x}-1-\frac{1}{x}\right)-\left(1-\frac{x_{1}+0.08}{x}\right)\right]\\
 & =xe^{-x}\left[\frac{1}{\left(1+e^{-x}\right)^{3}}\left(1-\frac{1}{x}-e^{-x}-\frac{e^{-x}}{x}\right)-\left(1-\frac{x_{1}+0.08}{x}\right)\right]
\end{align*}
Set $R(x):=\frac{1}{\left(1+e^{-x}\right)^{3}}-1+3e^{-x}$ and $\delta=x_{1}+0.08$.
We get

{\footnotesize{}
\begin{align*}
\alpha(x)-\varphi(x) & =xe^{-x}\left[\left(1-3e^{-x}+R(x)\right)\left(1-\frac{1}{x}-e^{-x}-\frac{e^{-x}}{x}\right)-\left(1-\frac{\delta}{x}\right)\right]\\
 & =xe^{-x}\left[\left(-1+\delta\right)\frac{1}{x}+\left(-1-3\right)e^{-x}+\left(-1+3\right)\frac{e^{-x}}{x}+3e^{-2x}+3\frac{e^{-2x}}{x}+R(x)\left(1-\frac{1}{x}-e^{-x}-\frac{e^{-x}}{x}\right)\right]\\
 & =xe^{-x}\left[\frac{\delta-1}{x}-4e^{-x}+\underset{\geq0}{\underbrace{2\frac{e^{-x}}{x}+3e^{-2x}+3\frac{e^{-2x}}{x}}}+R(x)\left(1-\frac{1}{x}-e^{-x}-\frac{e^{-x}}{x}\right)\right]\\
 & \geq xe^{-x}\left[\frac{\delta-1}{x}-4e^{-x}+R(x)\left(1-\frac{1}{x}-e^{-x}-\frac{e^{-x}}{x}\right)\right].
\end{align*}
}Let us now discuss the sign of $R(x)\left(1-\frac{1}{x}-e^{-x}-\frac{e^{-x}}{x}\right)$:\\
$R'(x)=3e^{-x}\left(1+e^{-x}\right)^{-4}-3e^{-x}=3e^{-x}\left(\left(1+e^{-x}\right)^{-4}-1\right)<0$,
$R$ is strictly decreasing. Since $\lim_{x\rightarrow+\infty}R(x)=0$,
necessarily $R\geq0$. In addition, since $x\geq3$, 
\[
1-\frac{1}{x}-e^{-x}-\frac{e^{-x}}{x}\geq1-\frac{1}{3}-e^{-x}-\frac{e^{-x}}{3}=\frac{2}{3}\left(1-2e^{-x}\right)\geq0.
\]
 Hence $R(x)\left(1-\frac{1}{x}-e^{-x}-\frac{e^{-x}}{x}\right)\geq0$
for $x\geq3$ and it comes
\[
\forall x\geq3,\alpha(x)-\varphi(x)\geq xe^{-x}\left(\frac{\delta-1}{x}-4e^{-x}\right).
\]
\end{proof}
\begin{lem}
\label{lem:control_alpha_by_phi}The function $\alpha:x\mapsto-\frac{e^{x}}{\left(1+e^{x}\right)^{2}}\left(1+x\frac{1-e^{x}}{1+e^{x}}\right)$
is greater than $\varphi:x\mapsto\left(x-x_{1}-0.08\right)e^{-x}$
on $\left[x_{1},\infty\right[$.
\end{lem}
\begin{proof}
Let us prove that $\alpha\geq\varphi$ by considering four intervals
$\left[x_{1},2\right]$, $\left[2,2.5\right]$, $\left[2.5,3\right]$
and $\left[3,\infty\right]$. We know that $\alpha$ is concave on
$\left[x_{1},3\right]$ according to lemma~\ref{lem:alpha_concave_x1_3}.
It is also the case of $\varphi$ because $\varphi''(x)=\left(x-x_{1}-0.08-2\right)e^{-x}$
which is negative on $\left[x_{1},3\right]$ since $x_{1}+0.08+2\approx3.62$.
Hence, $\alpha$ is above its geometrical chords and $\varphi$ below
its tangents on $\left[x_{1},3\right]$.

\uline{Case 1 on \mbox{$\left[x_{1},2\right]$}:}

\begin{figure}
\includegraphics[width=0.4\paperwidth]{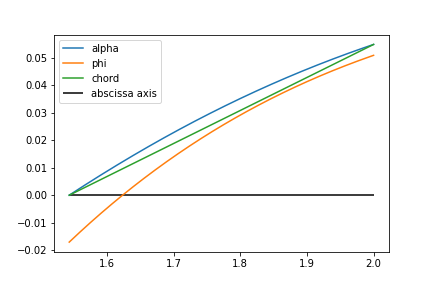}\caption{Plot of $\alpha$, $\varphi$ and its the chord on $[x_{1},2]$.\label{fig:plot_alpha_approx_x1_2}}
\end{figure}
The function $\alpha$ is above $l_{1}:x\mapsto\frac{\alpha\left(2\right)}{2-x_{1}}\left(x-x_{1}\right)$
and $\varphi$ is below $l_{2}:x\mapsto\varphi(1.85)+\varphi'(1.85)\left(x-1.85\right)$.
And as shown on figure~\ref{fig:plot_alpha_approx_x1_2}, $l_{1}\geq l_{2}$
on $\left[x_{1},2\right]$, one can compute the first coordinate of
their intersection point:
\[
x_{intersection}=\frac{\varphi(1.85)+\frac{\alpha\left(2\right)}{2-x_{1}}x_{1}-1.85\varphi'(1.85)}{\frac{\alpha\left(2\right)}{2-x_{1}}-\varphi'(1.85)}.
\]
A numerical computatuion gives $x_{intersection}\approx2.820$. The
two affine functions $l_{1}$ and $l_{2}$ intersect outside the intervalle
$\left[x_{1},2\right]$ and since at $x_{1}$ we have $l_{2}(x_{1})\approx-0.00165<0=l_{1}(x_{1})$,
we can conclude that on $\left[x_{1},2\right]$, $\alpha\geq l_{1}\geq l_{2}\geq\varphi$.

\uline{Case 2 on \mbox{$\left[2,2.5\right]$}:}

\begin{figure}
\includegraphics[width=0.4\paperwidth]{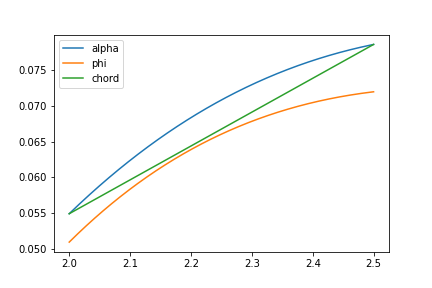}\caption{Plot of $\alpha$ and the chord of $\varphi$ on $[2,2.5]$. 
\label{fig:plot_alpha_approx_2_2.5}}
\end{figure}
The function $\alpha$ is above $l_{1}:x\mapsto\alpha\left(2\right)+\frac{\alpha\left(2.5\right)-\alpha\left(2\right)}{2.5-2}\left(x-2\right)$
and $\varphi$ is below $l_{2}:x\mapsto\varphi(2.2)+\varphi'(2.2)\left(x-2.2\right)$.
$l_{1}\geq l_{2}$ as well, one can check it with $l_{1}(2)\approx0.05493\geq0.05450\approx l_{2}(2)$
and $l_{1}(2.5)\approx0.0785\geq0.0779\approx l_{2}(2.5)$. Consequently
on $\left[2,2.5\right]$, $\alpha\geq l_{1}\geq l_{2}\geq\varphi$.

\uline{Case 3 on \mbox{$\left[2.5,3\right]$}:}

\begin{figure}
\includegraphics[width=0.4\paperwidth]{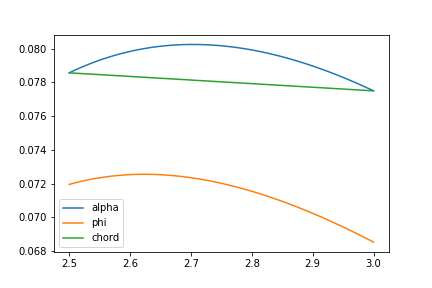}\caption{Plot of $\alpha$ and the chord of $\varphi$ on $[2.5,3]$.\label{fig:plot_alpha_approx_2.5_3}}
\end{figure}
The function $\alpha$ is above $l_{1}:x\mapsto\alpha\left(2.5\right)+\frac{\alpha\left(3\right)-\alpha\left(2.5\right)}{3-2.5}\left(x-2.5\right)$
and $\varphi$ is maximal at $x_{1}+1+0.08$ with approximate value
$0.07256$. And since $l_{1}(2.5)\approx0.07856\geq0.07256$ and $l_{1}(3)\approx0.07750\geq0.07256$,
we can conclude that $l_{1}$ is above the maximum of $\varphi$.
Hence, on $\left[2.5,3\right]$, $\alpha\geq l_{1}\geq\varphi$.

\uline{Case 4 on \mbox{$x\in\left[3,\infty\right]$}:}

\begin{figure}
\includegraphics[width=0.4\paperwidth]{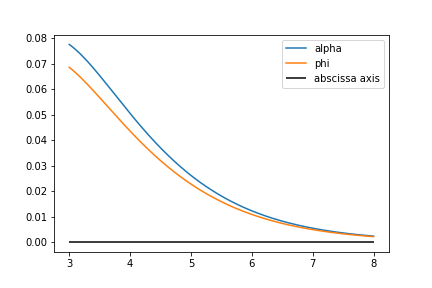}\caption{Plot of $\alpha$ and $\varphi$ on $[3,8]$
\label{fig:plot_alpha_approx_3_infty}}
\end{figure}
Thanks to Lemma~\ref{lem:comparison_alpha_minus_phi_to_0}, we know
that $\forall x\geq3$, $\alpha(x)-\varphi(x)\geq e^{-x}\left(x_{1}+0.08-1-4xe^{-x}\right)$.
Let us study the sign of $f:x\mapsto x_{1}+0.08-1-4xe^{-x}$. For
any $x\geq3$,
\[
f'(x)=-4e^{-x}+4xe^{-x}=4\left(x-1\right)e^{-x}\geq0.
\]
We also have $f(3)\approx0.026>0$. Consequently, $\forall x\geq3,f(x)\geq0$
and $\alpha(x)-\varphi(x)\geq0$ as well.

This completes the proof: $\forall x\geq x_{1},\alpha(x)-\varphi(x)\geq0$.
\end{proof}
\begin{lem}
\label{lem:beginning_minoration}Recall that $\alpha=-\frac{e^{x}}{\left(1+e^{x}\right)^{2}}\left(1+x\frac{1-e^{x}}{1+e^{x}}\right)$,
$Z\sim\mathcal{N}\left(a,I_{d}\right)$ and $\beta_{0}=Ra/\left\Vert a\right\Vert _{2}$,
$a\in\mathbb{R}^{d}$. We have
\begin{align}
\mathbb{E}\left[\alpha\left(Z^{t}\beta_{0}\right)\mathbb{I}_{\left\{ Z^{t}\beta_{0}>x_{1}\right\} }\right] & \geq\int_{x_{1}}^{\infty}\left(\left(x-x_{1}-\frac{8}{100}\right)e^{-x}\frac{e^{-\left(x/R-\left\Vert a\right\Vert _{2}\right)^{2}/2}}{\sqrt{2\pi R^{2}}}\right)dx.\label{eq:2_term_minoration}
\end{align}
\end{lem}
\begin{proof}
Recall that $\alpha(x)\geq\left(x-x_{1}-\frac{8}{100}\right)e^{-x}$
on $[x_{1},\infty[$ according to Lemma~\ref{lem:control_alpha_by_phi}.
Moreover $Z^{t}\beta_{0}\sim\mathcal{N}\left(a^{t}\beta_{0},\left\Vert \beta\right\Vert _{2}\right)$
with $a^{t}\beta_{0}=R\left\Vert a\right\Vert _{2}$, $\left\Vert \beta\right\Vert _{2}=R$.
This gives

{\footnotesize{}
\begin{align*}
\mathbb{E}\left[\alpha\left(Z^{t}\beta_{0}\right)\mathbb{I}_{\left\{ Z^{t}\beta_{0}>x_{1}\right\} }\right] & =\int_{x_{1}}^{\infty}\alpha(x)\frac{1}{\sqrt{2\pi R^{2}}}\exp\left(-\frac{\left(x-R\left\Vert a\right\Vert \right)^{2}}{2R^{2}}\right)dx\\
 & \geq\int_{x_{1}}^{\infty}\left(x-x_{1}-\frac{8}{100}\right)e^{-x}\frac{1}{\sqrt{2\pi R^{2}}}\exp\left(-\frac{\left(x/R-\left\Vert a\right\Vert \right)^{2}}{2}\right)dx.
\end{align*}
}{\footnotesize\par}
\end{proof}
\begin{lem}
\label{lem:good_asympt_behaviour}For any $a,z\in\mathbb{R}^{d},$
$\xi\in\mathbb{R}$ and $R>0$, it holds{\footnotesize{}
\begin{align*}
J_{a,R}\left(\xi,z\right):=\int_{z}^{\infty}\left(xe^{-\xi x}\frac{1}{\sqrt{2\pi R^{2}}}e^{-\left(x/R-\left\Vert a\right\Vert _{2}\right)^{2}/2}\right)dx & =R\left(1+\left(\left\Vert a\right\Vert _{2}-R\xi\right)G\left(\frac{z}{R}+R\xi-\left\Vert a\right\Vert _{2}\right)\right)\gamma\left(\frac{z}{R}-\left\Vert a\right\Vert _{2}\right)e^{-\xi z},
\end{align*}
}where $\gamma:x\rightarrow(2\pi)^{-1/2}e^{-\frac{1}{2}x^{2}}$ is
the standard Gaussian density, $\Phi^{c}$ is the standard Gaussian
tail function and $G:x\mapsto\Phi^{c}\left(x\right)/\gamma(x)$ is
the Gaussian Mill's ratio.
\end{lem}
\begin{proof}
We have

\begin{align}
J_{a,R}\left(\xi,z\right) & =\int_{z}^{\infty}xe^{-\xi x}\frac{1}{\sqrt{2\pi R^{2}}}\exp\left(-\frac{\left(x-R\left\Vert a\right\Vert \right)^{2}}{2R^{2}}\right)dx\nonumber \\
 & =\int_{z}^{\infty}\frac{x}{\sqrt{2\pi R^{2}}}\exp\left(-\frac{\left(x-R\left\Vert a\right\Vert \right)^{2}+2R^{2}\xi x}{2R^{2}}\right)dx.\label{eq:save_computation_time}
\end{align}
Moreover, 
\[
\left(x-R\left\Vert a\right\Vert \right)^{2}+2R^{2}\xi x=\left(x+R^{2}\xi-R\left\Vert a\right\Vert _{2}\right)^{2}+R^{2}\xi\left(2R\left\Vert a\right\Vert _{2}-R^{2}\xi\right).
\]
Hence,

\begin{align}
J_{a,R}\left(\xi,z\right) & =\int_{z}^{\infty}\frac{x}{\sqrt{2\pi R^{2}}}\exp\left(-\frac{\left(x+R\left(R\xi-\left\Vert a\right\Vert _{2}\right)\right)^{2}+R^{2}\left(2\left\Vert a\right\Vert _{2}-R\xi\right)R\xi}{2R^{2}}\right)dx\nonumber \\
 & =e^{-\frac{\left(2\left\Vert a\right\Vert _{2}-R\xi\right)R\xi}{2}}\int_{z}^{\infty}x\frac{1}{\sqrt{2\pi R^{2}}}\exp\left(-\frac{\left(x/R+R\xi-\left\Vert a\right\Vert _{2}\right)^{2}}{2}\right)dx\label{eq:intermediary_integrale_J}
\end{align}
By the change the variable $y=x/R+R\xi-\left\Vert a\right\Vert _{2}$,
we get

\begin{align*}
J_{a,R}\left(\xi,z\right) & =e^{-\frac{\left(2\left\Vert a\right\Vert _{2}-R\xi\right)R\xi}{2}}\int_{z/R+R\xi-\left\Vert a\right\Vert _{2}}^{\infty}R\left(y-\left(R\xi-\left\Vert a\right\Vert _{2}\right)\right)\frac{1}{\sqrt{2\pi}}\exp\left(-\frac{y^{2}}{2}\right)dy\\
 & =e^{-\frac{\left(2\left\Vert a\right\Vert _{2}-R\xi\right)R\xi}{2}}\int_{z/R+R\xi-\left\Vert a\right\Vert _{2}}^{\infty}Ry\frac{1}{\sqrt{2\pi}}\exp\left(-\frac{y^{2}}{2}\right)dy\\
 & \quad-e^{-\frac{\left(2\left\Vert a\right\Vert _{2}-R\xi\right)R\xi}{2}}\int_{z/R+R\xi-\left\Vert a\right\Vert _{2}}^{\infty}R\left(R\xi-\left\Vert a\right\Vert _{2}\right)\frac{1}{\sqrt{2\pi}}\exp\left(-\frac{y^{2}}{2}\right)dy\\
 & =-Re^{-\frac{\left(2\left\Vert a\right\Vert _{2}-R\xi\right)R\xi}{2}}\left[\frac{1}{\sqrt{2\pi}}\exp\left(-\frac{y^{2}}{2}\right)\right]_{z/R+R\xi-\left\Vert a\right\Vert _{2}}^{\infty}\\
 & \quad-R\left(R\xi-\left\Vert a\right\Vert _{2}\right)e^{-\frac{\left(2\left\Vert a\right\Vert _{2}-R\xi\right)R\xi}{2}}\Phi^{c}\left(\frac{z}{R}+R\xi-\left\Vert a\right\Vert _{2}\right)\\
 & =Re^{-\frac{\left(2\left\Vert a\right\Vert _{2}-R\xi\right)R\xi}{2}}\left(\frac{1}{\sqrt{2\pi}}e^{-\frac{1}{2}\left(\frac{z}{R}+R\xi-\left\Vert a\right\Vert _{2}\right)^{2}}+\left(\left\Vert a\right\Vert _{2}-R\xi\right)\Phi^{c}\left(\frac{z}{R}+R\xi-\left\Vert a\right\Vert _{2}\right)\right)\\
 & =R\left[1+\left(\left\Vert a\right\Vert _{2}-R\xi\right)G\left(\frac{z}{R}+R\xi-\left\Vert a\right\Vert _{2}\right)\right]\frac{1}{\sqrt{2\pi}}e^{-\frac{R\xi\left(2\left\Vert a\right\Vert _{2}-R\xi\right)+\left(\frac{z}{R}+R\xi-\left\Vert a\right\Vert _{2}\right)^{2}}{2}}
\end{align*}
and since 
\begin{align}
\left(2\left\Vert a\right\Vert _{2}-R\xi\right)R\xi+\left(z/R+R\xi-\left\Vert a\right\Vert _{2}\right)^{2} & =\left(z/R-\left\Vert a\right\Vert _{2}\right)^{2}+2z\xi,\label{eq:useful_simplification-1}
\end{align}
we finally get the result.
\end{proof}
\begin{lem}
\label{lem:good_asympt_behaviour-1}For any $a,z\in\mathbb{R}^{d},$
$\xi\in\mathbb{R}$ and $R>0$, it holds
\begin{align*}
K_{a,R}\left(\xi,z\right):=\int_{z}^{\infty}e^{-\xi x}\frac{1}{\sqrt{2\pi R^{2}}}e^{-\left(x/R-\left\Vert a\right\Vert _{2}\right)^{2}/2}dx & =\gamma\left(\frac{z}{R}-\left\Vert a\right\Vert _{2}\right)G\left(\frac{z}{R}+R\xi-\left\Vert a\right\Vert _{2}\right)e^{-\xi z}
\end{align*}
where $\gamma$ is the standard Gaussian density, $\Phi^{c}$ is the
standard Gaussian tail function and $G:x\mapsto\Phi^{c}\left(x\right)/\gamma(x)$
is the Gaussian Mill's ratio.
\end{lem}
\begin{proof}
By the same calculation as in Equation~\ref{eq:save_computation_time},
we can write

\begin{align*}
K_{a,R}\left(\xi,z\right) & =e^{-\frac{\left(2\left\Vert a\right\Vert _{2}-R\xi\right)R\xi}{2}}\int_{z}^{\infty}\frac{1}{\sqrt{2\pi R^{2}}}\exp\left(-\frac{\left(x/R+R\xi-\left\Vert a\right\Vert _{2}\right)^{2}}{2}\right)dx.
\end{align*}
By the change the variable $y=x/R+R\xi-\left\Vert a\right\Vert _{2}$,
we get

\begin{align*}
K_{a,R}\left(\xi,z\right) & =e^{-\frac{\left(2\left\Vert a\right\Vert _{2}-R\xi\right)R\xi}{2}}\int_{z/R+R\xi-\left\Vert a\right\Vert _{2}}^{\infty}\frac{1}{\sqrt{2\pi}}\exp\left(-\frac{y^{2}}{2}\right)dy\\
 & =e^{-\frac{\left(2\left\Vert a\right\Vert _{2}-R\xi\right)R\xi}{2}}\Phi^{c}\left(z/R+R\xi-\left\Vert a\right\Vert _{2}\right)\\
 & =\frac{1}{\sqrt{2\pi}}e^{-\frac{\left(2\left\Vert a\right\Vert _{2}-R\xi\right)R\xi+\left(z/R+R\xi-\left\Vert a\right\Vert _{2}\right)^{2}}{2}}.G\left(z/R+R\xi-\left\Vert a\right\Vert _{2}\right).
\end{align*}
By Identity (\ref{eq:useful_simplification-1}), it follows that

\begin{align*}
K_{a,R}\left(\xi,z\right) & =\frac{1}{\sqrt{2\pi}}e^{-\frac{1}{2}\left(z/R-\left\Vert a\right\Vert _{2}\right)^{2}-z\xi}G\left(z/R+R\xi-\left\Vert a\right\Vert _{2}\right),
\end{align*}
as expected. 
\end{proof}
\begin{lem}
\label{lem:rule_for_G}Set $G\left(x\right)=\frac{\Phi^{c}\left(x\right)}{\gamma\left(x\right)}$
the Mill's ratio of the standard gaussian distribution. $G$ satisfies:
$\forall x\in\mathbb{R}$
\begin{align*}
xG\left(x\right)-G'\left(x\right) & =1\\
G''(x)-xG'\left(x\right)-G\left(x\right) & =0\\
G'''(x)-2G'\left(x\right)-xG''\left(x\right) & =0
\end{align*}
\end{lem}
\begin{proof}
$G\left(x\right)=\frac{\Phi^{c}\left(x\right)}{\gamma\left(x\right)}$
and using the fact that $\frac{d\Phi^{c}}{dx}\left(x\right)=-\gamma\left(x\right)$
and $\gamma'\left(x\right)=-x\gamma\left(x\right)$ it comes:
\begin{align*}
G'\left(x\right) & =\frac{-\gamma\left(x\right)\gamma\left(x\right)-\Phi^{c}\left(x\right)\left(-x\gamma\left(x\right)\right)}{\gamma\left(x\right)\gamma\left(x\right)}\\
 & =-1+x\frac{\Phi^{c}\left(x\right)}{\gamma\left(x\right)}\\
 & =-1+xG\left(x\right)
\end{align*}

and $G'=xG-1\Rightarrow G''=G+xG'\Rightarrow G'''=G'+G'+xG''$
\end{proof}
\begin{prop}
\label{prop:control_of_G}The function $G\left(x\right)=\frac{\Phi^{c}\left(x\right)}{\gamma\left(x\right)}$
is known as the Gaussian Mill's ratio and $\forall x\geq0$

\[
0<\frac{2}{x+\sqrt{x^{2}+4}}\leq G\left(x\right)\leq\frac{2}{x+\sqrt{x^{2}+4.\frac{2}{\pi}}}
\]
\begin{figure}
\includegraphics[width=0.4\paperwidth]{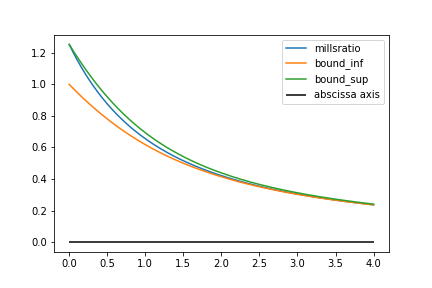}\caption{Plot of the function $G$ on $[0,4]$ 
\label{fig:plot_mill_s_ratio}}
\end{figure}
\end{prop}
\begin{proof}
Focus on the first inequality: the lower bound is due to~\cite{birnbaum1942inequality}
and the upper bound is attributed to Pollak~\cite{MR83529} according
to~\cite{gasull2014approximating} in which one can find the inequality
in the first commentar of Remark~11 p~1848.
\end{proof}
\begin{prop}
\label{prop:G_decreasing}The Gaussian mill's ratio function $G:x\mapsto\frac{\Phi^{c}\left(x\right)}{\gamma\left(x\right)}$
is a strictly decreasing function on $\mathbb{R}$.
\end{prop}
\begin{proof}
We have seen in Lemma~\ref{lem:rule_for_G} that $G'=-1+xG$. Since
$G\geq0$, it is obvious that $G'<0$ on $\left]-\infty;0\right]$.
Furthermore, take $x>0$, then with proposition~\ref{prop:control_of_G}
we have 
\begin{align*}
G'(x) & =-1+xG(x)\\
 & \leq-1+x\frac{2}{x+\sqrt{x^{2}+4.\frac{2}{\pi}}}\\
 & =\frac{2}{1+\sqrt{1+\frac{8}{\pi x^{2}}}}-1
\end{align*}
One can see that $\forall x>0,\frac{2}{1+\sqrt{1+\frac{8}{\pi x^{2}}}}<1$,
hence $G'<0$ everywhere on $\left]0,\infty\right[$.
\end{proof}
\begin{lem}
\label{lem:monotonicity_wrt_a_in_Eq}Define $Eq_{a,b,c,d}:1+aG\left(-b-a\right)\geq cG\left(a+d\right)$
where $a,b,c,d\geq0$ and $G:x\mapsto\frac{\Phi^{c}\left(x\right)}{\gamma\left(x\right)}$
is the Gaussian mill's ratio where $\gamma$ and $\Phi^{c}$ are respectively
the density and the tail function of the standard univariate gaussian.
If $Eq_{a,b,c,d}$ holds true, then $\forall h>0,Eq_{a+h,b,c,d}$
holds true.
\end{lem}
\begin{proof}
Start with $\left(a,b,c,d\right)$ such that $Eq_{a,b,c,d}$ holds
true and take $h>0$. We proved in prop~\ref{prop:G_decreasing}
that $G$ is a decreasing function,then $G\left(-b-a\right)<G\left(-b-\left(a+h\right)\right)$,
then one has $0\leq aG\left(-b-a\right)<\left(a+h\right)G\left(-b-\left(a+h\right)\right)$,
hence

\[
1+\left(a+h\right)G\left(-b-\left(a+h\right)\right)>1+aG\left(-b-a\right)
\]
But $\left(a,b,c,d\right)$ such that $Eq_{a,b,c,d}$ holds true,
therefore 
\[
1+\left(a+h\right)G\left(-b-\left(a+h\right)\right)>cG\left(a+d\right)
\]

Finally, use again the fact that $G$ is decreasing, to have $G\left(a+d\right)>G\left(\left(a+h\right)+d\right)$
and it comes
\[
1+\left(a+h\right)G\left(-b-\left(a+h\right)\right)>cG\left(\left(a+h\right)+d\right)
\]

To conclude, $Eq_{a+h,b,c,d}$ holds also true.
\end{proof}
\begin{lem}
\label{lem:particular_case_of_Eq}The equation $R\left(1-\left(R-\left\Vert a\right\Vert _{2}+\frac{x_{1}+\frac{8}{100}}{R}\right)G\left(\frac{x_{1}}{R}+R-\left\Vert a\right\Vert _{2}\right)\right)\geq\left(1+\nu\right)\frac{e^{x_{1}}}{4}G\left(\left\Vert a\right\Vert _{2}-\frac{x_{1}}{R}\right)$
holds true, in particular, for $R=\sqrt{x_{1}+0.08}\approx1.2741$,
$\left\Vert a\right\Vert _{2}=2R=R+\frac{x_{1}+0.08}{R}\approx2.548$
and $\nu=0.95$.
\end{lem}
\begin{proof}
Replace the corresponding quantities to get as left side $R\left(1-\left(R-\left\Vert a\right\Vert _{2}+\frac{x_{1}+\frac{8}{100}}{R}\right)G\left(\frac{x_{1}}{R}+R-\left\Vert a\right\Vert _{2}\right)\right)=R=\sqrt{x_{1}+0.08}$
and as right side $\left(1+\nu\right)\frac{e^{x_{1}}}{4}G\left(R+\frac{x_{1}+\frac{8}{100}}{R}-\frac{x_{1}}{R}\right)=\left(1+\nu\right)\frac{e^{x_{1}}}{4}G\left(\sqrt{x_{1}+0.08}+\frac{0.08}{\sqrt{x_{1}+0.08}}\right)$.\\
Approximation show that $\sqrt{x_{1}+0.08}\approx1.2741$, $\frac{e^{x_{1}}}{4}\approx1.1701$,
$\sqrt{x_{1}+0.08}+\frac{0.08}{\sqrt{x_{1}+0.08}}\approx1.33700$
and $G\left(1.337\right)\approx0.5552$. On can see it is then enough
to takes $\nu=0.95$ because $\left(1+\nu\right)\frac{e^{x_{1}}}{4}G\left(\sqrt{x_{1}+0.08}+\frac{0.08}{\sqrt{x_{1}+0.08}}\right)\approx1.2668$
(the inequality holds true because $1.2741\geq1.2668$).
\end{proof}
\begin{lem}
\label{lem:Control_risk_linear_op}Recall that $\left(d_{\beta}\mathcal{R}\right)\left(\nu\right)=-\mathbb{E}\left[X^{t}\beta p_{\beta}\left(X\right)q_{\beta}\left(X\right)X^{t}\nu\right]$
for $\beta\in B_{2}\left(0,R\right)$ and $\nu\in\mathbb{R}^{d}$.
Assume that $a^{t}\beta-2\left\Vert \beta\right\Vert _{2}^{2}\geq0$.
If $\left\langle \nu,\beta\right\rangle \leq0$, it holds
\[
\left(d_{\beta}\mathcal{R}\right)\left(\nu\right)\geq\frac{1}{8}e^{-\left(2a^{t}\beta-\left\Vert \beta\right\Vert _{2}^{2}\right)/2}\left\langle -\nu,\frac{\beta}{\left\Vert \beta\right\Vert _{2}^{2}}\right\rangle \left(\left\Vert \beta\right\Vert _{2}^{2}+\left(a^{t}\beta-\left\Vert \beta\right\Vert _{2}^{2}\right)^{2}\right).
\]
\end{lem}
\begin{proof}
For $\nu\in\mathbb{R}^{d}$, decompose it on $\beta^{\perp}\oplus Vect\left(\beta\right)$
as $\nu=\nu_{\perp}+\nu_{\parallel}$, and set $\lambda_{\nu}:=\left\langle \nu,\frac{\beta}{\left\Vert \beta\right\Vert _{2}}\right\rangle $
so that $\nu_{\parallel}=\lambda_{\nu}\frac{\beta}{\left\Vert \beta\right\Vert _{2}}$
. Recall $X\sim\varepsilon_{Rad}Z$ where $\varepsilon_{Rad}$ is
a Rademacher random variable with distribution $\frac{1}{2}\delta_{-1}+\frac{1}{2}\delta_{1}$and
$Z\sim\mathcal{N}\left(a,I_{d}\right)$. As a consequence $Z^{t}\beta\sim\mathcal{N}\left(a^{t}\beta,\left\Vert \beta\right\Vert _{2}^{2}\right)$.
Set also $N\sim\mathcal{N}\left(0,1\right)$, so that $Z^{t}\beta=a^{t}\beta+\left\Vert \beta\right\Vert _{2}N$.
We have, by symmetry in $X$ and independence between $Z^{t}\beta$
and $Z^{t}\nu_{\perp}$,

\begin{align*}
\left(d_{\beta}\mathcal{R}\right)\left(\nu\right) & =-\mathbb{E}\left[X^{t}\beta p_{\beta}\left(X\right)q_{\beta}\left(X\right)X^{t}\nu\right]\\
 & =-\mathbb{E}\left[\frac{Z^{t}\beta e^{-Z^{t}\beta}}{\left(1+e^{-Z^{t}\beta}\right)^{2}}Z^{t}\left(\nu_{\perp}+\nu_{\parallel}\right)\right]\\
 & =-\mathbb{E}\left[\frac{Z^{t}\beta e^{-Z^{t}\beta}}{\left(1+e^{-Z^{t}\beta}\right)^{2}}Z^{t}\nu_{\perp}\right]-\mathbb{E}\left[\frac{Z^{t}\beta e^{-Z^{t}\beta}}{\left(1+e^{-Z^{t}\beta}\right)^{2}}Z^{t}\nu_{\parallel}\right]\\
 & =-\mathbb{E}\left[\frac{Z^{t}\beta e^{-X^{t}\beta}}{\left(1+e^{-X^{t}\beta}\right)^{2}}\right]\underset{=0}{\underbrace{\mathbb{E}\left[Z^{t}\nu_{\perp}\right]}}-\frac{\lambda_{\nu}}{\left\Vert \beta\right\Vert _{2}}\mathbb{E}\left[\frac{Z^{t}\beta e^{-Z^{t}\beta}}{\left(1+e^{-Z^{t}\beta}\right)^{2}}Z^{t}\beta\right]\\
 & =-\frac{\lambda_{\nu}}{\left\Vert \beta\right\Vert _{2}}\mathbb{E}\left[\zeta\left(a^{t}\beta+\left\Vert \beta\right\Vert _{2}N\right)\right],
\end{align*}
where $\zeta:x\mapsto\frac{x^{2}e^{x}}{\left(1+e^{x}\right)^{2}}$.
Note that the function $\zeta$ is even and that a simple calculation
gives $\forall x>0,\zeta(x)\geq\frac{x^{2}}{4}e^{-x}$. If $\lambda_{\nu}\leq0$,
\[
-\frac{\lambda_{\nu}}{\left\Vert \beta\right\Vert _{2}}\mathbb{E}\left[\zeta\left(a^{t}\beta+\left\Vert \beta\right\Vert _{2}N\right)\right]\geq-\frac{\lambda_{\nu}}{\left\Vert \beta\right\Vert _{2}}\int_{0}^{+\infty}\frac{x^{2}}{4}e^{-x}\frac{1}{\sqrt{2\pi\left\Vert \beta\right\Vert _{2}^{2}}}\exp\left(-\frac{\left(x-a^{t}\beta\right)^{2}}{2\left\Vert \beta\right\Vert _{2}^{2}}\right)dx.
\]
Set $N_{a,\beta}\sim\mathcal{N}\left(a^{t}\beta-2\left\Vert \beta\right\Vert _{2}^{2},\left\Vert \beta\right\Vert _{2}^{2}\right)$.
This gives

\begin{align*}
\mathbb{E}\left[\zeta\left(a^{t}\beta+\left\Vert \beta\right\Vert _{2}^{2}N\right)\right] & \geq\int_{0}^{+\infty}\frac{x^{2}}{4}e^{-x}\frac{1}{\sqrt{2\pi\left\Vert \beta\right\Vert _{2}^{2}}}\exp\left(-\frac{\left(x-a^{t}\beta\right)^{2}}{2\left\Vert \beta\right\Vert _{2}^{2}}\right)dx\\
 & =\int_{0}^{+\infty}\frac{x^{2}}{4}\frac{1}{\sqrt{2\pi\left\Vert \beta\right\Vert _{2}^{2}}}\exp\left(-\frac{\left(x-\left(a^{t}\beta-\left\Vert \beta\right\Vert _{2}^{2}\right)\right)^{2}+\left\Vert \beta\right\Vert _{2}^{2}\left(2a^{t}\beta-\left\Vert \beta\right\Vert _{2}^{2}\right)}{2\left\Vert \beta\right\Vert _{2}^{2}}\right)dx\\
 & =e^{-\left(2a^{t}\beta-\left\Vert \beta\right\Vert _{2}^{2}\right)/2}\int_{0}^{+\infty}\frac{x^{2}}{4}\frac{1}{\sqrt{2\pi\left\Vert \beta\right\Vert _{2}^{2}}}\exp\left(-\frac{\left(x-\left(a^{t}\beta-2\left\Vert \beta\right\Vert _{2}^{2}\right)\right)^{2}}{2\left\Vert \beta\right\Vert _{2}^{2}}\right)dx\\
 & \geq\frac{1}{8}e^{-\left(2a^{t}\beta-\left\Vert \beta\right\Vert _{2}^{2}\right)/2}\mathbb{E}\left[N_{a,\beta}^{2}\right]\\
 & =\frac{1}{8}e^{-\left(2a^{t}\beta-\left\Vert \beta\right\Vert _{2}^{2}\right)/2}\left(\mathbb{V}\left[N_{a,\beta}^{2}\right]+\mathbb{E}\left[N_{a,\beta}\right]^{2}\right)\\
 & =\frac{1}{8}e^{-\left(2a^{t}\beta-\left\Vert \beta\right\Vert _{2}^{2}\right)/2}\left(\left\Vert \beta\right\Vert _{2}^{2}+\left(a^{t}\beta-2\left\Vert \beta\right\Vert _{2}^{2}\right)^{2}\right),
\end{align*}
where in the second inequality, we used the fact that $a^{t}\beta-2\left\Vert \beta\right\Vert _{2}^{2}\geq0$.
Therefore
\[
\left(d_{\beta}\mathcal{R}\right)\left(\nu\right)\geq-\frac{\lambda_{\nu}}{8\left\Vert \beta\right\Vert _{2}}e^{-\left(2a^{t}\beta-\left\Vert \beta\right\Vert _{2}^{2}\right)/2}\left(\left\Vert \beta\right\Vert _{2}^{2}+\left(a^{t}\beta-\left\Vert \beta\right\Vert _{2}^{2}\right)^{2}\right).
\]
\end{proof}
\begin{defn}
\label{def:tensor_operator_norm}The operator norm $\left\Vert \cdot\right\Vert _{op}$
on the trilinear symmetric operator space with respect to $\left\Vert \cdot\right\Vert _{2}$
is defined as: for all $T$ symetric trilinear operator, $\left\Vert T\right\Vert _{op}:=\underset{u\in\partial B_{2}\left(0,1\right)}{\sup}\left|T\left(u,u,u\right)\right|$
as shown in equation~(2) in both~\cite{zhang2012best} and~\cite{qi2012spectral}.
\end{defn}
\begin{lem}
\label{lem:risk_trilinear_op_norm}With trilinear symmetric operator
defined above, the third derivative of the risk satisfies: $\forall\beta\in\mathbb{R}^{d}$,
\[
\left\Vert d_{t\beta+(1-t)\beta_{0}}^{3}\mathcal{R}\right\Vert _{op}\leq8e^{-\left(a^{t}\beta-\left\Vert \beta\right\Vert _{2}^{2}\right)}\sqrt{2\left(\left\Vert a\right\Vert ^{6}+\mathbb{E}\left[N_{0}^{6}\right]\right)\left(\left[\left\Vert \beta\right\Vert _{2}^{2}+\left[a^{t}\beta-2\left\Vert \beta\right\Vert _{2}^{2}\right]^{2}\right]+\left[a^{t}\beta-2\left\Vert \beta\right\Vert _{2}^{2}\right]+1\right)},
\]
 where $N_{0}\sim\mathcal{N}\left(0,1\right)$.
\end{lem}
\begin{proof}
if $u\in\partial B_{2}\left(0,1\right)$, then for $N\sim\mathcal{N}\left(0,I_{d}\right)$,
$N^{t}u\sim N_{0}\sim\mathcal{N}\left(0,1\right)$, and it is known
that $\mathbb{E}\left[\left|N_{0}\right|\right]=\sqrt{\frac{2}{\pi}}$
and $\mathbb{E}\left[\left|N_{0}\right|^{3}\right]=3\mathbb{E}\left[\left|N_{0}\right|\right]\mathbb{V}\left[\left|N_{0}\right|\right]+\mathbb{E}\left[\left|N_{0}\right|\right]^{3}$.
Owing to Equation~(\ref{eq:risk_trilinear_op}) and Cauchy-Schwarz
inequality, we have

\begin{align*}
\left|d_{\beta}^{3}\mathcal{R}\left(u,u,u\right)\right| & =\left|\mathbb{E}\left[\left(X^{t}u\right)^{3}.\alpha'\left(X^{t}\beta\right)\right]\right|\\
 & \leq\sqrt{\mathbb{E}\left[\left(X^{t}u\right)^{6}\right]\mathbb{E}\left[\left(\alpha'\left(X^{t}\beta\right)\right)^{2}\right]}.
\end{align*}
On the one hand, using the fact that $\forall a,b>0,\forall n\in\mathbb{N},\left(a+b\right)^{n}\leq2^{n-1}\left(a^{n}+b^{n}\right)$
we get 
\begin{align*}
\mathbb{E}\left[\left(X^{t}u\right)^{6}\right] & =\mathbb{E}\left[\left(a^{t}u+N^{t}u\right)^{6}\right]\\
 & =\mathbb{E}\left[\left(a^{t}u+N_{0}\right)^{6}\right]\\
 & \leq\mathbb{E}\left[2^{5}\left(\left(a^{t}u\right)^{6}+N_{0}^{6}\right)\right]\\
 & \leq32\left(\left\Vert a\right\Vert _{2}^{6}+\mathbb{E}\left[N_{0}^{6}\right]\right).
\end{align*}
On the other hand, we have already proved in Lemma~\ref{lem:Study_alpha}
that {\small{}$\forall x,\alpha'(x)=p_{x}q_{x}\left[\frac{x}{2}\left(1-3\tanh^{2}\left(\frac{x}{2}\right)\right)+2\tanh\left(\frac{x}{2}\right)\right]$.
Hence,}{\small\par}

\begin{align*}
\left|\alpha'(x)\right| & \leq p_{x}q_{x}\left|\frac{x}{2}\left(1-3\tanh^{2}\left(\frac{x}{2}\right)\right)+2\tanh\left(\frac{x}{2}\right)\right|\\
 & \leq p_{x}q_{x}\left(\frac{x}{2}\left|1-3\tanh^{2}\left(\frac{x}{2}\right)\right|+2\right)\\
 & \leq\frac{e^{x}}{\left(1+e^{x}\right)^{2}}\left(\frac{x}{2}\left(3\tanh^{2}\left(\frac{x}{2}\right)+1\right)+2\right)\\
 & \leq\frac{e^{x}}{e^{x}\left(e^{-x/2}+e^{x/2}\right)^{2}}\left(\frac{x}{2}\times4+2\right)\\
 & \leq\frac{2}{\left(e^{x/2}\right)^{2}}\left(x+1\right)\\
 & =2e^{-x}\left(x+1\right).
\end{align*}
Recall $Z^{t}\beta\sim\mathcal{N}\left(a^{t}\beta,\left\Vert \beta\right\Vert _{2}^{2}\right)$,
\begin{align}
\mathbb{E}\left[\left(\alpha'\left(X^{t}\beta\right)\right)^{2}\right] & =\mathbb{E}\left[\left(\alpha'\left(Z^{t}\beta\right)\right)^{2}\right]\nonumber \\
 & \leq\mathbb{E}\left[4e^{-2Z^{t}\beta}\left(Z^{t}\beta+1\right)^{2}\right]\nonumber \\
 & =4\int_{\mathbb{R}}\left(x+1\right)^{2}e^{-2x}\frac{1}{\sqrt{2\pi\left\Vert \beta\right\Vert _{2}}}\exp\left(-\frac{\left(x-a^{t}\beta\right)^{2}}{2\left\Vert \beta\right\Vert _{2}^{2}}\right)dx.\label{eq:save_computation_time_2}
\end{align}
By denoting $N_{a,\beta}\sim\mathcal{N}\left(a^{t}\beta-2\left\Vert \beta\right\Vert _{2}^{2},\left\Vert \beta\right\Vert _{2}^{2}\right)$,
we get

\begin{align*}
\mathbb{E}\left[\left(\alpha'\left(X^{t}\beta\right)\right)^{2}\right] & =4\int_{\mathbb{R}}\left(x+1\right)^{2}\frac{1}{\sqrt{2\pi\left\Vert \beta\right\Vert _{2}^{2}}}\exp\left(-\frac{\left(x+\left(2\left\Vert \beta\right\Vert _{2}^{2}-a^{t}\beta\right)\right)^{2}+2\left\Vert \beta\right\Vert _{2}^{2}\left(2a^{t}\beta-2\left\Vert \beta\right\Vert _{2}^{2}\right)}{2\left\Vert \beta\right\Vert _{2}^{2}}\right)dx\\
\\
 & =4e^{-2\left(a^{t}\beta-\left\Vert \beta\right\Vert _{2}^{2}\right)}\int_{\mathbb{R}}\left(x+1\right)^{2}\frac{1}{\sqrt{2\pi\left\Vert \beta\right\Vert _{2}^{2}}}\exp\left(-\frac{\left(x-\left(a^{t}\beta-2\left\Vert \beta\right\Vert _{2}^{2}\right)\right)^{2}}{2}\right)dx\\
 & =4e^{-2\left(a^{t}\beta-\left\Vert \beta\right\Vert _{2}^{2}\right)}\mathbb{E}\left[\left(N_{a,\beta}+1\right)^{2}\right]\\
 & =4e^{-2\left(a^{t}\beta-\left\Vert \beta\right\Vert _{2}^{2}\right)}\left(\mathbb{E}\left[N_{a,\beta}^{2}\right]+2\mathbb{E}\left[N_{a,\beta}\right]+1\right)\\
 & =4e^{-2\left(a^{t}\beta-\left\Vert \beta\right\Vert _{2}^{2}\right)}\left(\left[\left\Vert \beta\right\Vert _{2}^{2}+\left[a^{t}\beta-2\left\Vert \beta\right\Vert _{2}^{2}\right]^{2}\right]+2\left[a^{t}\beta-2\left\Vert \beta\right\Vert _{2}^{2}\right]+1\right)
\end{align*}
Finally, we have
\begin{align*}
\left|d_{\beta}^{3}\mathcal{R}\left(u,u,u\right)\right| & \leq\sqrt{8e^{-\left(a^{t}\beta-\left\Vert \beta\right\Vert _{2}^{2}\right)}\sqrt{2\left(\left\Vert a\right\Vert ^{6}+\mathbb{E}\left[N_{0}^{6}\right]\right)\left(\left[\left\Vert \beta\right\Vert _{2}^{2}+\left[a^{t}\beta-2\left\Vert \beta\right\Vert _{2}^{2}\right]^{2}\right]+2\left[a^{t}\beta-2\left\Vert \beta\right\Vert _{2}^{2}\right]+1\right)}},
\end{align*}
which gives the result, according to definition~\ref{def:tensor_operator_norm}.
\end{proof}
\begin{lem}
\label{lem:minoration_excess_risk_locally}Under the condition that
$\left\Vert a\right\Vert _{2}\geq2R$, $R=\sqrt{x_{1}+0.08}$, the
excess risk $\mathcal{E}\left(\cdot,\beta_{0}\right)$ satisfies around
$\beta_{0}$:\\
 $\forall\beta\in B_{2}\left(\beta_{0},\varepsilon\right)\bigcap B_{2}\left(0,R\right)$,

\[
\mathcal{E}\left(\beta,\beta_{0}\right)\geq e^{-\left(\left\Vert a\right\Vert _{2}R-R^{2}/2\right)}\left[\frac{1}{16}\left(1+\left(\left\Vert a\right\Vert _{2}-R\right)^{2}\right)\left\Vert \beta-\beta_{0}\right\Vert _{2}^{2}-24\left\Vert a\right\Vert ^{4}e^{R^{2}/2}e^{\varepsilon\left\Vert a\right\Vert _{2}}\left\Vert \beta-\beta_{0}\right\Vert _{2}^{3}\right].
\]
\end{lem}
\begin{proof}
Fisrt note that $\mathcal{E}\left(\beta_{0},\beta_{0}\right)=0$ by
definition of $\mathcal{E}\left(\cdot,\beta_{0}\right)$. According
to lemma~\ref{lem:Control_risk_linear_op} we can control $\left(d_{\beta}\mathcal{R}\right)\left(\beta-\beta_{0}\right)$
from below.\\
Since $\left\langle \beta-\beta_{0},\beta_{0}\right\rangle \leq0$
and $a^{t}\beta_{0}-2\left\Vert \beta_{0}\right\Vert _{2}^{2}\geq0$,
we have

\begin{align*}
\left(d_{\beta_{0}}\mathcal{R}\right)\left(\beta-\beta_{0}\right) & \geq\frac{1}{8}e^{-\left(2a^{t}\beta_{0}-\left\Vert \beta_{0}\right\Vert _{2}^{2}\right)/2}\left\langle \beta_{0}-\beta,\frac{\beta_{0}}{\left\Vert \beta_{0}\right\Vert _{2}^{2}}\right\rangle \left(\left\Vert \beta_{0}\right\Vert _{2}^{2}+\left(a^{t}\beta_{0}-\left\Vert \beta_{0}\right\Vert _{2}^{2}\right)^{2}\right)\\
 & \geq\frac{1}{8}e^{-\left(\left\Vert a\right\Vert _{2}R-R^{2}/2\right)}\left\langle \beta_{0}-\beta,\beta_{0}\right\rangle \left(1+\left(\left\Vert a\right\Vert _{2}-R\right)^{2}\right).
\end{align*}
Use now Lemma~\ref{lem:comparison_scal_prod_and_sqared_norm},
\begin{align*}
\left(d_{\beta_{0}}\mathcal{R}\right)\left(\beta-\beta_{0}\right) & \geq\frac{1}{16}e^{-\left(\left\Vert a\right\Vert _{2}R-R^{2}/2\right)}\left(1+\left(\left\Vert a\right\Vert _{2}-R\right)^{2}\right)\left\Vert \beta-\beta_{0}\right\Vert _{2}^{2}
\end{align*}
According to lemma~\ref{lem:Condition-2}, we can control $\left(d_{\beta}^{2}\mathcal{R}\right)\left(\beta-\beta_{0},\beta-\beta_{0}\right)$
from below:

\begin{align*}
\left(d_{\beta}^{2}\mathcal{R}\right)\left(\beta-\beta_{0},\beta-\beta_{0}\right) & \geq\Lambda_{min}\left\Vert \beta-\beta_{0}\right\Vert _{2}^{2}\geq0.
\end{align*}
In addition, we can use Lemma~\ref{lem:risk_trilinear_op_norm} to
have{\footnotesize{}
\begin{align*}
 & \int_{0}^{1}\left|\left(d_{t\beta+(1-t)\beta_{0}}^{3}\mathcal{R}\right)\left(\beta-\beta_{0},\beta-\beta_{0},\beta-\beta_{0}\right)\right|dt\\
 & =\left\Vert \beta-\beta_{0}\right\Vert _{2}^{3}\int_{0}^{1}\left|\left(d_{t\beta+(1-t)\beta_{0}}^{3}\mathcal{R}\right)\left(\frac{\beta-\beta_{0}}{\left\Vert \beta-\beta_{0}\right\Vert _{2}},\frac{\beta-\beta_{0}}{\left\Vert \beta-\beta_{0}\right\Vert _{2}},\frac{\beta-\beta_{0}}{\left\Vert \beta-\beta_{0}\right\Vert _{2}}\right)\right|dt\\
 & \leq\left\Vert \beta-\beta_{0}\right\Vert _{2}^{3}\int_{0}^{1}\left\Vert d_{t\beta+(1-t)\beta_{0}}^{3}\mathcal{R}\right\Vert _{op}dt\\
 & \leq8\left\Vert \beta-\beta_{0}\right\Vert _{2}^{3}\int_{0}^{1}\exp\left(-\left(a^{t}\left(t\beta+(1-t)\beta_{0}\right)-\left\Vert t\beta+(1-t)\beta_{0}\right\Vert _{2}^{2}\right)\right)C_{3,a}\left(t\beta+(1-t)\beta_{0}\right)dt,
\end{align*}
}where 
\[
C_{3,a}:\mu\in\mathbb{R}^{d}\mapsto\sqrt{2\left(\left\Vert a\right\Vert _{2}^{6}+\mathbb{E}\left[N_{0}^{6}\right]\right)\left(\left[\left\Vert \mu\right\Vert _{2}^{2}+\left[a^{t}\nu-2\left\Vert \mu\right\Vert _{2}^{2}\right]^{2}\right]+2\left[a^{t}\nu-2\left\Vert \mu\right\Vert _{2}^{2}\right]+1\right)}.
\]
To bound $C_{3,a}$ from above, remark that $\forall\mu\in\Psi_{U}$,
$\left\Vert \mu\right\Vert _{2}^{2}\leq4R^{2}$, $-2R^{2}\leq a^{t}\nu-2\left\Vert \mu\right\Vert _{2}^{2}\leq R\left\Vert a\right\Vert _{2}-2R^{2}\leq R\left\Vert a\right\Vert _{2}$
and remark also that owing to $\left\Vert a\right\Vert \geq2R\approx2.548$
and article~\cite{wilkelbauer2012moments}, one has $\mathbb{E}\left[N_{0}^{6}\right]=15$
with $\mathcal{N}\left(0,1\right)$, so $\mathbb{E}\left[N_{0}^{6}\right]\leq\frac{1}{18}\left\Vert a\right\Vert _{2}^{6}$
. Therefore, all together this leads to

\begin{align}
C_{3,a}\left(\mu\right) & \leq\sqrt{2\left(\left\Vert a\right\Vert _{2}^{6}+\mathbb{E}\left[N_{0}^{6}\right]\right)\left(\left[R^{2}+R^{2}\left[\max\left(2R,\left\Vert a\right\Vert _{2}\right)\right]^{2}\right]+2\left[R\left\Vert a\right\Vert _{2}-2R^{2}\right]+1\right)}\nonumber \\
 & \leq\sqrt{2\left(\left\Vert a\right\Vert _{2}^{6}+\frac{1}{18}\left\Vert a\right\Vert _{2}^{6}\right)\left(\left[R^{2}+R^{2}\left\Vert a\right\Vert _{2}^{2}\right]+2\left[R\left\Vert a\right\Vert _{2}-2R^{2}\right]+1\right)}\nonumber \\
 & \leq\sqrt{\frac{19}{9}\left\Vert a\right\Vert _{2}^{6}\left(R^{2}\left\Vert a\right\Vert _{2}^{2}+2R\left\Vert a\right\Vert _{2}-3R^{2}+1\right)}\nonumber \\
 & \leq3\left\Vert a\right\Vert ^{4}.\label{eq:majoration_C_3aR-1}
\end{align}
Hence

{\footnotesize{}
\begin{align*}
 & \int_{0}^{1}\left|\left(d_{t\beta+(1-t)\beta_{0}}^{3}\mathcal{R}\right)\left(\beta-\beta_{0},\beta-\beta_{0},\beta-\beta_{0}\right)\right|dt\\
 & \leq24\left\Vert a\right\Vert ^{4}\left\Vert \beta-\beta_{0}\right\Vert _{2}^{3}\int_{0}^{1}\exp\left(-\left(a^{t}\left(t\beta+(1-t)\beta_{0}\right)-\left\Vert t\beta+(1-t)\beta_{0}\right\Vert _{2}^{2}\right)\right)dt\\
 & \leq24\left\Vert a\right\Vert ^{4}\left\Vert \beta-\beta_{0}\right\Vert _{2}^{3}\underset{\left\Vert \beta-\beta_{0}\right\Vert _{2}\leq\varepsilon}{\sup}\underset{0\leq t\leq1}{\sup}\exp\left(-\left(a^{t}\left(t\beta+(1-t)\beta_{0}\right)-\left\Vert t\beta+(1-t)\beta_{0}\right\Vert _{2}^{2}\right)\right)\\
 & \leq24\left\Vert a\right\Vert ^{4}\left\Vert \beta-\beta_{0}\right\Vert _{2}^{3}\exp\left(-\underset{\left\Vert \beta-\beta_{0}\right\Vert _{2}\leq\varepsilon}{\inf}\underset{0\leq t\leq1}{\inf}a^{t}\left(t\beta+(1-t)\beta_{0}\right)+\underset{\left\Vert \beta-\beta_{0}\right\Vert _{2}\leq\varepsilon}{\sup}\underset{0\leq t\leq1}{\sup}\left\Vert t\beta+(1-t)\beta_{0}\right\Vert _{2}^{2}\right)\\
 & \leq24\left\Vert a\right\Vert ^{4}\left\Vert \beta-\beta_{0}\right\Vert _{2}^{3}\exp\left(-a^{t}\beta_{0}+\left\Vert a\right\Vert _{2}\varepsilon+R^{2}\right)\\
 & \leq24\left\Vert a\right\Vert ^{4}e^{-\left\Vert a\right\Vert _{2}R+R^{2}}e^{\varepsilon\left\Vert a\right\Vert _{2}}\left\Vert \beta-\beta_{0}\right\Vert _{2}^{3}.
\end{align*}
}Finally, this gives 
\begin{align}
\mathcal{E}\left(\beta,\beta_{0}\right) & >e^{-\left(\left\Vert a\right\Vert _{2}R-R^{2}/2\right)}\left[\frac{1}{16}\left(1+\left(\left\Vert a\right\Vert _{2}-R\right)^{2}\right)\left\Vert \beta-\beta_{0}\right\Vert _{2}^{2}-24\left\Vert a\right\Vert ^{4}e^{R^{2}/2}e^{\varepsilon\left\Vert a\right\Vert _{2}}\left\Vert \beta-\beta_{0}\right\Vert _{2}^{3}\right],\label{eq:taylor_control_excess_risk-1}
\end{align}
as required.
\end{proof}
\begin{lem}
\label{lem:Minoration_min_eigenvalue}Minoration of $\Phi^{c}\left(\left\Vert a\right\Vert _{2}-\frac{x_{1}}{R}\right)-\Phi^{c}\left(\left\Vert a\right\Vert _{2}+\frac{x_{1}}{R}\right)$:
\[
\forall a,R,x_{1},\Phi^{c}\left(\left\Vert a\right\Vert _{2}-\frac{x_{1}}{R}\right)-\Phi^{c}\left(\left\Vert a\right\Vert _{2}+\frac{x_{1}}{R}\right)\geq2\frac{x_{1}}{R}\gamma\left(\left\Vert a\right\Vert _{2}+\frac{x_{1}}{R}\right)
\]
\end{lem}
\begin{proof}
Simple computations give
\begin{align*}
\Phi^{c}\left(\left\Vert a\right\Vert _{2}-\frac{x_{1}}{R}\right)-\Phi^{c}\left(\left\Vert a\right\Vert _{2}+\frac{x_{1}}{R}\right) & =\int_{\left\Vert a\right\Vert _{2}-\frac{x_{1}}{R}}^{\left\Vert a\right\Vert _{2}+\frac{x_{1}}{R}}\gamma\left(x\right)d\lambda\left(x\right)\\
 & \geq\int_{\left\Vert a\right\Vert _{2}-\frac{x_{1}}{R}}^{\left\Vert a\right\Vert _{2}+\frac{x_{1}}{R}}\gamma\left(\left\Vert a\right\Vert _{2}+\frac{x_{1}}{R}\right)d\lambda\left(x\right)\\
 & \geq2\frac{x_{1}}{R}\gamma\left(\left\Vert a\right\Vert _{2}+\frac{x_{1}}{R}\right)
\end{align*}
\end{proof}
\begin{lem}
\label{lem:comparison_scal_prod_and_sqared_norm}$\forall\beta\in B_{2}\left(0,R\right)$,
if $\beta_{0}\in\partial B_{2}\left(0,R\right)$ then $\left\langle \beta_{0}-\beta,\beta_{0}\right\rangle \geq\frac{1}{2}\left\Vert \beta-\beta_{0}\right\Vert _{2}^{2}$.
\end{lem}
\begin{proof}
Decompose $\beta$ as $\beta_{\perp}+\beta_{\parallel}$ on $\beta_{0}^{\perp}\oplus Vect\left(\beta_{0}\right)$
and note that $\exists\lambda_{\beta}\in[-1,1],\beta_{\parallel}=\lambda_{\beta}\beta_{0}$.
We have
\[
\frac{\left\langle \beta_{0}-\beta,\beta_{0}\right\rangle }{\left\Vert \beta-\beta_{0}\right\Vert _{2}^{2}}=\frac{\left\langle \beta_{0}-\beta_{\parallel},\beta_{0}\right\rangle }{\left\Vert \beta_{\perp}\right\Vert _{2}^{2}+\left\Vert \beta_{\parallel}-\beta_{0}\right\Vert _{2}^{2}}=\frac{\left(1-\lambda_{\beta}\right)R^{2}}{\left\Vert \beta_{\perp}\right\Vert _{2}^{2}+\left(1-\lambda_{\beta}\right)^{2}R^{2}}
\]
Furthermore, we have $\left\Vert \beta\right\Vert _{2}^{2}\leq R^{2}$
and by pythagora's theorem $\left\Vert \beta_{\perp}\right\Vert _{2}^{2}\in\left[0,R^{2}-\lambda_{\beta}^{2}R^{2}\right]$.
Therefore,
\begin{align*}
\frac{\left\langle \beta_{0}-\beta,\beta_{0}\right\rangle }{\left\Vert \beta-\beta_{0}\right\Vert _{2}^{2}} & \geq\frac{\left(1-\lambda_{\beta}\right)R^{2}}{\left(1-\lambda_{\beta}^{2}\right)R^{2}+\left(1-\lambda_{\beta}\right)^{2}R^{2}}\\
 & \geq\frac{1-\lambda_{\beta}}{1-\lambda_{\beta}^{2}+\left(1-2\lambda_{\beta}+\lambda_{\beta}^{2}\right)}\\
 & \geq\frac{1}{2}.
\end{align*}
\end{proof}
\begin{defn}
\label{def:empirical_L2_ norm}When one has $P_{n}=\frac{1}{n}\sum_{i=1}^{n}\delta_{x_{i}}$
with $x_{1},\dots,x_{n}\in\mathcal{X}$, define the following ``empirical-$L^{2}$-
norm'' as: 
\[
\forall f:\mathcal{X}\rightarrow\mathbb{R},\left\Vert f\right\Vert _{P_{n}}:=\sqrt{\frac{1}{n}\sum_{i=1}^{n}f^{2}\left(X^{(i)}\right)}
\]
\end{defn}
.
\begin{defn}
\label{def:entropy_of_a_set}For $\delta>0$, the $\delta$-covering
number $N\left(\delta,\mathscr{H},\left\Vert \cdot\right\Vert \right)$
of a set $\mathscr{H}$ is the smallest number of closed balls, with
respect to $\left\Vert \cdot\right\Vert $ with radius $\delta$,
that covers the space. The set of the centers of the balls is called
a $\delta$-covering set. The entropy of $\mathscr{H}$ with respect
to a norm $\left\Vert \cdot\right\Vert $ is $H\left(\cdot,\mathscr{H},\left\Vert \cdot\right\Vert \right)=\log N\left(\cdot,\mathscr{H},\left\Vert \cdot\right\Vert \right)$.
\end{defn}
.
\begin{lem}
\label{lem:entropy_lemma}Define $\Theta\left(\varepsilon\right):=\left\{ \beta\in B_{2}\left(0,R\right):\left\Vert \beta-\beta_{0}\right\Vert _{1}\leq\varepsilon\right\} $
and take 
\[
\mathscr{H}_{\varepsilon,M_{n}}:=\left\{ \left(\rho_{\beta}-\rho_{\beta_{0}}\right)I_{\left\{ G\leq M_{n}\right\} }-\mathbb{E}\left[\left(\rho_{\beta}(X)-\rho_{\beta_{0}}(X)\right)I_{\left\{ G(X)\leq M_{n}\right\} }\right]:\beta\in\Theta\left(\varepsilon\right)\right\} ,
\]
 where $G\left(X\right):=\left\Vert X\right\Vert _{\infty}$. Recall
that $L$ is the Lipschitz constant of $\rho$.\\
Then for all $u>0$ and $M_{n}>0$, the entropy of $\mathscr{H}_{\varepsilon,M_{n}}$
with respect to the empirical-$L^{2}$-norm $\left\Vert \cdot\right\Vert _{P_{n}}$
(see definition~\ref{def:empirical_L2_ norm}) satisfies
\[
H\left(u,\mathscr{H}_{\varepsilon,M_{n}},\left\Vert \cdot\right\Vert _{P_{n}}\right)\leq\left(\frac{4L^{2}M_{n}^{2}\varepsilon^{2}}{u^{2}}+1\right)\log\left(2d\right).
\]
\end{lem}
\begin{proof}
Let $\widehat{X_{i}},\dots,\widehat{X_{i}}$ be i.i.d copies of $X$
and set $\mathscr{B}_{\varepsilon,M_{n}}:=\left\{ f_{\beta,\beta'}:X\mapsto\frac{X^{t}}{M_{n}}\left(\beta-\beta'\right)I_{\left\{ G(X)\leq M_{n}\right\} }:\beta,\beta'\in\Theta\left(\varepsilon\right)\right\} $.\\
One has $\forall\beta,\beta'\in\Theta\left(\varepsilon\right)$, 
\[
\left|\rho_{\beta}\left(X\right)-\rho_{\beta'}\left(X\right)\right|=\left|\rho\left(X^{t}\beta\right)-\rho\left(X^{t}\beta'\right)\right|\leq L\left|X^{t}\beta-X^{t}\beta'\right|.
\]
With $\forall a,b>0,\left(a+b\right)^{2}\leq2\left(a^{2}+b^{2}\right)$,
it follows that
\begin{align*}
 & \left\Vert \rho_{\beta}I_{\left\{ G(\cdot)\leq M_{n}\right\} }-\mathbb{E}\left[\rho_{\beta}I_{\left\{ G(\cdot)\leq M_{n}\right\} }\right]-\rho_{\beta'}I_{\left\{ G(\cdot)\leq M_{n}\right\} }+\mathbb{E}\left[\rho_{\beta'}I_{\left\{ G(\cdot)\leq M_{n}\right\} }\right]\right\Vert _{P_{n}}^{2}\\
 & =\frac{1}{n}\sum_{i=1}^{n}\left(\rho_{\beta}\left(X^{(i)}\right)I_{\left\{ G\left(X^{(i)}\right)\leq M_{n}\right\} }-\mathbb{E}\left[\rho_{\beta}\left(X\right)I_{\left\{ G\left(X\right)\leq M_{n}\right\} }\right]-\left(\rho_{\beta'}\left(X^{(i)}\right)I_{\left\{ G\left(X^{(i)}\right)\leq M_{n}\right\} }-\mathbb{E}\left[\rho_{\beta'}\left(X\right)I_{\left\{ G\left(X\right)\leq M_{n}\right\} }\right]\right)\right)^{2}\\
 & \leq\frac{2}{n}\sum_{i=1}^{n}\left(\rho_{\beta}\left(X^{(i)}\right)-\rho_{\beta'}\left(X^{(i)}\right)\right)^{2}I_{\left\{ G\left(X^{(i)}\right)\leq M_{n}\right\} }+\frac{2}{n}\sum_{i=1}^{n}\left(\mathbb{E}\left[\rho_{\beta'}\left(X\right)I_{\left\{ G\left(X\right)\leq M_{n}\right\} }\right]-\mathbb{E}\left[\rho_{\beta}\left(X\right)I_{\left\{ G\left(X\right)\leq M_{n}\right\} }\right]\right)^{2}\\
 & \leq\frac{2}{n}\sum_{i=1}^{n}\left(\rho_{\beta}\left(X^{(i)}\right)-\rho_{\beta'}\left(X^{(i)}\right)\right)^{2}I_{\left\{ G\left(X^{(i)}\right)\leq M_{n}\right\} }+2\mathbb{E}\left[\left(\rho_{\beta'}\left(X\right)-\rho_{\beta}\left(X\right)\right)^{2}I_{\left\{ G\left(X\right)\leq M_{n}\right\} }\right].
\end{align*}
Furthermore,
\begin{align*}
\frac{2}{n}\sum_{i=1}^{n}\left(\rho_{\beta}\left(X^{(i)}\right)-\rho_{\beta'}\left(X^{(i)}\right)\right)^{2}I_{\left\{ G\left(X^{(i)}\right)\leq M_{n}\right\} } & \leq\frac{2}{n}\sum_{i=1}^{n}L^{2}\left|X^{(i)t}\beta-X^{(i)t}\beta'\right|^{2}I_{\left\{ G\left(X^{(i)}\right)\leq M_{n}\right\} }\\
 & \leq\frac{2}{n}L^{2}\sum_{i=1}^{n}G\left(X^{(i)}\right)^{2}\left\Vert \beta-\beta'\right\Vert _{1}^{2}I_{\left\{ G\left(X^{(i)}\right)\leq M_{n}\right\} }\\
 & \leq2L^{2}M_{n}^{2}\left\Vert \beta-\beta'\right\Vert _{1}^{2}.
\end{align*}
One also has 
\[
\text{\ensuremath{\mathbb{E}\left[\left(\rho_{\beta'}\left(X\right)-\rho_{\beta}\left(X\right)\right)^{2}I_{\left\{  G\left(X\right)\leq M_{n}\right\}  }\right]\leq L^{2}M_{n}^{2}\left\Vert \beta-\beta'\right\Vert _{1}^{2}}}.
\]
Hence
\begin{align}
\left\Vert \rho_{\beta}I_{\left\{ G(\cdot)\leq M_{n}\right\} }-\mathbb{E}\left[\rho_{\beta}I_{\left\{ G(\cdot)\leq M_{n}\right\} }\right]+\rho_{\beta'}I_{\left\{ G(\cdot)\leq M_{n}\right\} }-\mathbb{E}\left[\rho_{\beta'}I_{\left\{ G(\cdot)\leq M_{n}\right\} }\right]\right\Vert _{P_{n}}^{2} & \leq4L^{2}M_{n}^{2}.\left\Vert \beta-\beta'\right\Vert _{1}^{2}.\label{eq:distance_2_empirical_proc}
\end{align}
This relation enables us to state 
\[
H\left(u,\mathscr{H}_{\varepsilon,M_{n}},\left\Vert \cdot\right\Vert _{P_{n}}\right)\leq H\left(\frac{u}{2LM_{n}},\Theta\left(\varepsilon\right),\left\Vert \cdot\right\Vert _{1}\right).
\]
Define the convex hull of a set of vectors $\left\{ e_{j}\right\} _{j=1}^{d}$
as $Conv\left\{ e_{j}\right\} _{j=1}^{d}:=\left\{ \left.\sum_{j=1}^{d}v_{j}e_{j}\right|v_{i}\geq0,\left\Vert v\right\Vert _{1}=1\right\} $
and take in particular the vectors $\left\{ e_{j}\right\} _{j=1}^{d}$
of the canonical basis in $\mathbb{R}^{d}$. Then 
\[
\Theta\left(\varepsilon\right)\subset\beta_{0}+\varepsilon.Conv\left\{ 0,\left\{ \pm e_{j}\right\} _{j=1}^{d}\right\} .
\]
Owing to the definition of $e_{j}$, we have $\forall j,\left\Vert e_{j}\right\Vert _{1}=1$.
so we can use Lemma~14.28 in~\cite{buhlmann2011statistics} to get
\begin{align*}
H\left(u,\Theta\left(\varepsilon\right),\left\Vert \cdot\right\Vert _{1}\right) & \leq H\left(u,\beta_{0}+\varepsilon.Conv\left\{ 0,\left\{ \pm e_{j}\right\} _{j=1}^{d}\right\} ,\left\Vert \cdot\right\Vert _{1}\right)\\
 & \leq H\left(\frac{u}{\varepsilon},Conv\left\{ 0,\left\{ \pm e_{j}\right\} _{j=1}^{d}\right\} ,\left\Vert \cdot\right\Vert _{1}\right)\\
 & \leq H\left(\frac{u}{\varepsilon},Conv\left\{ 0,\left\{ \pm e_{j}\right\} _{j=1}^{d}\right\} ,\left\Vert \cdot\right\Vert _{1}\right)\\
 & \leq\left\lceil \frac{\varepsilon^{2}}{u^{2}}\right\rceil \left(1+\log\left(1+\left(2d+1\right)\frac{u^{2}}{\varepsilon^{2}}\right)\right)\wedge\left\lceil \frac{\varepsilon^{2}}{u^{2}}\right\rceil \log\left(2d\right)\\
 & \leq\left(\frac{\varepsilon^{2}}{u^{2}}+1\right)\log\left(2d\right),
\end{align*}
which gives the result. 
\end{proof}
\begin{lem}
\label{lem:control_empirical_process_trunc}Let $\varepsilon>0$ and
$X^{(1)},...,X^{(i)},...,X^{(n)}$ be i.i.d. copies of $X$. Let also
\[
\mathscr{H}_{\varepsilon,M_{n}}:=\left\{ \left(\rho_{\beta}-\rho_{\beta_{0}}\right)I_{\left\{ G\leq M_{n}\right\} }-\mathbb{E}\left[\left(\rho_{\beta}(X)-\rho_{\beta_{0}}(X)\right)I_{\left\{ G(X)\leq M_{n}\right\} }\right]:\beta\in\Theta\left(\varepsilon\right)\right\} ,
\]
 where $G\left(X\right):=\left\Vert X\right\Vert _{\infty}$ and 
\[
\Theta\left(\varepsilon\right):=\left\{ \beta\in\mathbb{R}^{d}:\beta\in B_{2}\left(0,R\right),\left\Vert \beta-\beta_{0}\right\Vert _{1}\leq\varepsilon\right\} .
\]
Recall that we set $L$, the Lipschitz norm of $\rho$. One has $\forall T\geq1$,
$\forall n\geq2$, 

{\footnotesize{}
\[
P\left(\underset{\beta\in\Theta\left(\varepsilon\right)}{\sup}\frac{\left|V_{n}^{trunc}\left(\beta\right)-V_{n}^{trunc}\left(\beta_{0}\right)\right|}{\varepsilon}\geq\frac{3LM_{n}T\left(5\sqrt{3\log\left(2d\right)}\log n+4\right)}{\sqrt{n}}\right)<\exp\left(-21\left(T-1\right)^{2}\log\left(2d\right)\log^{2}n\right).
\]
}{\footnotesize\par}
\end{lem}
\begin{proof}
According to equation~(\ref{eq:distance_2_empirical_proc}), $\forall\widetilde{\rho}_{\beta}\in\mathscr{H}_{\varepsilon,M_{n}}$,
$\left\Vert \widetilde{\rho}_{\beta}\right\Vert _{P_{n}}\leq2LM_{n}\varepsilon=:R_{n}$
and $\mathbb{E}\left(\widetilde{\rho}_{\beta}\left(X\right)\right)=0$.
Hence, using Lemma~\ref{lem:entropy_lemma} and Definition~\ref{def:entropy_of_a_set},
we have 
\begin{align*}
\log\left(1+N\left(u,\mathscr{H}_{\varepsilon,M_{n}},\left\Vert \cdot\right\Vert _{P_{n}}\right)\right) & \leq1+H\left(u,\mathscr{H}_{\varepsilon,M_{n}},\left\Vert \cdot\right\Vert _{P_{n}}\right)\\
 & \leq1+\left(\frac{4L^{2}M_{n}^{2}\varepsilon^{2}}{u^{2}}+1\right)\log\left(2d\right)\\
 & \leq\left(\frac{4L^{2}M_{n}^{2}\varepsilon^{2}}{u^{2}}+2\right)\log\left(2d\right)
\end{align*}
Take $u:=2^{-s}R_{n}$ where $0\leq s\leq S:=\min\left\{ s\geq1:2^{-s}\leq\frac{4}{\sqrt{n}}\right\} $
(i.e. $u\in\left[\frac{2}{\sqrt{n}}R_{n},R_{n}\right]$),\\
then one has $\forall0\leq s\leq S$,
\begin{align*}
\log\left(1+N\left(2^{-s}R_{n},\mathscr{H}_{\varepsilon,M_{n}},\left\Vert \cdot\right\Vert _{P_{n}}\right)\right) & \leq\left(\frac{4L^{2}M_{n}^{2}\varepsilon^{2}}{2^{-2s}R_{n}^{2}}+2\right)\log\left(2d\right)\\
 & \leq\left(2^{2s}+2\right)\log\left(2d\right)\\
 & \leq2^{2s}\left(1+2^{1-2s}\right)\log\left(2d\right)\\
 & \leq2^{2s}\times3\log\left(2d\right)
\end{align*}
Now one can apply \cite[Corollary 14.4]{buhlmann2011statistics},
where in our case $A:=3\log\left(2d\right)$. Note that $4\log n\leq3\log_{2}n\leq5\log n$.
We get
\[
\mathbb{E}\left[\underset{\widetilde{\rho}\in\mathscr{H}_{\varepsilon,M_{n}}}{\sup}\left|\frac{1}{n}\sum_{i=1}^{n}\widetilde{\rho}_{\beta}\left(X^{(i)}\right)\right|\right]\leq\frac{R_{n}}{\sqrt{n}}\left(5\sqrt{A}\log n+4\right).
\]
One can apply the Massart's concentration inequality, recalled for
instance in~\cite[Theorem 14.2]{buhlmann2011statistics}. ,Then,
$\forall t>0$,
\[
P\left(\underset{\widetilde{\rho}\in\mathscr{H}_{\varepsilon,M_{n}}}{\sup}\left|\frac{1}{n}\sum_{i=1}^{n}\frac{\widetilde{\rho}_{\beta}\left(X^{(i)}\right)}{R_{n}}\right|\geq\mathbb{E}\left[\underset{\widetilde{\rho}\in\mathscr{H}_{\varepsilon,M_{n}}}{\sup}\left|\frac{1}{n}\sum_{i=1}^{n}\frac{\widetilde{\rho}_{\beta}\left(X^{(i)}\right)}{R_{n}}\right|\right]+t\right)\leq e^{-nt^{2}/8},
\]
which gives

\[
P\left(\underset{\widetilde{\rho}\in\mathscr{H}_{\varepsilon,M_{n}}}{\sup}\left|\frac{1}{n}\sum_{i=1}^{n}\widetilde{\rho}_{\beta}\left(X^{(i)}\right)\right|\geq\frac{R_{n}}{\sqrt{n}}\left(5\sqrt{A}\log n+4\right)+R_{n}t\right)\leq e^{-nt^{2}/8}.
\]
A change of variable $t=\frac{1}{\sqrt{n}}\left(T-1\right)\left(5\sqrt{A}\log n+4\right)$
leads to: $\forall T\geq1$, 

\[
P\left(\underset{\widetilde{\rho}\in\mathscr{H}_{\varepsilon,M_{n}}}{\sup}\left|\frac{1}{n}\sum_{i=1}^{n}\widetilde{\rho}_{\beta}\left(X^{(i)}\right)\right|\geq\frac{R_{n}}{\sqrt{n}}T\left(5\sqrt{A}\log n+4\right)\right)<\exp\left(-\frac{\left(T-1\right)^{2}\left(5\sqrt{A}\log n+4\right)^{2}}{8}\right)
\]
Note that $\forall\widetilde{\rho}_{\beta}\in\mathscr{H}_{\varepsilon,M_{n}}$,
{\scriptsize{}
\begin{align*}
\frac{1}{n}\sum_{i=1}^{n}\widetilde{\rho}_{\beta}\left(X^{(i)}\right) & =V_{n}^{trunc}\left(\beta\right)-V_{n}^{trunc}\left(\beta_{0}\right).
\end{align*}
}Consequently, $\forall T\geq1$, 

{\scriptsize{}
\begin{align*}
P\left(\underset{\beta\in\Theta\left(\varepsilon\right)}{\sup}\left|V_{n}^{trunc}\left(\beta\right)-V_{n}^{trunc}\left(\beta_{0}\right)\right|\geq\frac{3LM_{n}\varepsilon T\left(5\sqrt{3\log\left(2d\right)}\log n+4\right)}{\sqrt{n}}\right) & <\exp\left(-\frac{\left(3T/2-1\right)^{2}\left(5\sqrt{3\log\left(2d\right)}\log n+4\right)^{2}}{8}\right)\\
 & <\exp\left(-21\left(T-1\right)^{2}\log\left(2d\right)\log^{2}n\right).
\end{align*}
}{\scriptsize\par}
\end{proof}
\begin{lem}
\label{lem:peeling_device}Grant the notations of Lemma~\ref{lem:control_empirical_process_trunc}
and set $\lambda_{0}:=3LM_{n}\left(5\sqrt{3\log\left(2d\right)}\log n+4\right)n^{-1/2}$.
One has $\forall T\geq1$, $\forall n\geq2$,
\begin{align*}
P\left(\underset{\beta\in B_{2}(0,R)}{\sup}\frac{\left|V_{n}^{trunc}\left(\beta\right)-V_{n}^{trunc}\left(\beta_{0}\right)\right|}{\left\Vert \beta-\beta_{0}\right\Vert _{1}\lor\lambda_{0}}\geq T\lambda_{0}\right) & \leq\frac{3}{4}\log\left(\frac{4R^{2}nd}{L^{2}M_{n}^{2}}\right)\exp\left(-21\left(T-1\right)^{2}\log\left(2d\right)\log^{2}n\right).
\end{align*}
\end{lem}
\begin{proof}
Let $\lambda_{0}>0$, $n\geq2$ and $T\geq1$. Let us use a peeling:
define $\Theta:=B_{2}\left(0,R\right)$ and divide it into slices
as follows:
\[
\Theta_{j}:=\left\{ \beta\in B_{2}\left(0,R\right):2^{-j-1}\leq\left\Vert \beta-\beta_{0}\right\Vert _{1}\leq2^{-j}\right\} .
\]
Note that $\exists j_{inf},j_{sup}\in\mathbb{Z}$, $\exists r>0$,
$2^{-j_{sup}-1}\leq\lambda_{0}\leq2^{-j_{sup}}$ and $2^{-j_{inf}-1}\leq r:=2R\sqrt{d}\leq2^{-j_{inf}}$
with 
\[
\Theta\subset\bigcup_{j=j_{inf}}^{j_{sup}}\Theta_{j}\bigcup B_{1}\left(\beta_{0},2^{-j_{sup}}\right)
\]
 and $\Theta$$\subset$$B_{1}\left(\beta_{0},2^{-j_{inf}-1}\right)$.
One can also prove that $j_{inf}=\left\lfloor -\log_{2}r\right\rfloor =\left\lfloor -\log_{2}2R\sqrt{d}\right\rfloor \leq-1$
because $R,d\geq1$ and $j_{sup}=\left\lfloor -\log_{2}\left(\lambda_{0}\right)\right\rfloor $
. Hence{\scriptsize{}
\begin{align*}
P\left(\underset{\beta\in\Theta}{\sup}\frac{\left|V_{n}^{trunc}\left(\beta\right)-V_{n}^{trunc}\left(\beta_{0}\right)\right|}{\left\Vert \beta-\beta_{0}\right\Vert _{1}\lor\lambda_{0}}\geq T\lambda_{0}\right) & \leq\sum_{j=j_{inf}}^{j_{sup}}P\left(\underset{\beta\in\Theta_{j}}{\sup}\frac{\left|V_{n}^{trunc}\left(\beta\right)-V_{n}^{trunc}\left(\beta_{0}\right)\right|}{\left\Vert \beta-\beta_{0}\right\Vert _{1}\lor\lambda_{0}}\geq T\lambda_{0}\right)\\
 & \quad+P\left(\underset{\beta\in B_{1}\left(\beta_{0},2^{-j_{sup}}\right)}{\sup}\frac{\left|V_{n}^{trunc}\left(\beta\right)-V_{n}^{trunc}\left(\beta_{0}\right)\right|}{\left\Vert \beta-\beta_{0}\right\Vert _{1}\lor\lambda_{0}}\geq T\lambda_{0}\right).
\end{align*}
}Use the fact that $\forall j\in\left\llbracket j_{inf},j_{sup}\right\rrbracket $,
$\lambda_{0}\leq2^{-j}$ and $\forall\beta\in B_{1}\left(\beta_{0},2^{-j_{sup}}\right)$,
$\left\Vert \beta-\beta_{0}\right\Vert _{1}\lor\lambda_{0}\leq2^{-j_{sup}}$:

{\scriptsize{}
\begin{align*}
P\left(\underset{\beta\in\Theta}{\sup}\frac{\left|V_{n}^{trunc}\left(\beta\right)-V_{n}^{trunc}\left(\beta_{0}\right)\right|}{\left\Vert \beta-\beta_{0}\right\Vert _{1}\lor\lambda_{0}}\geq T\lambda_{0}\right) & \leq\sum_{j=j_{inf}}^{j_{sup}}P\left(\underset{\beta\in\Theta_{j}}{\sup}\frac{\left|V_{n}^{trunc}\left(\beta\right)-V_{n}^{trunc}\left(\beta_{0}\right)\right|}{2^{-j}}\geq T\lambda_{0}\right)\\
 & +P\left(\underset{\beta\in B_{1}\left(\beta_{0},2^{-j_{sup}}\right)}{\sup}\frac{\left|V_{n}^{trunc}\left(\beta\right)-V_{n}^{trunc}\left(\beta_{0}\right)\right|}{2^{-j_{sup}}}\geq T\lambda_{0}\right).
\end{align*}
}By applying Lemma~\ref{lem:control_empirical_process_trunc} with
$\lambda_{0}=3LM_{n}\left(5\sqrt{3\log\left(2d\right)}\log n+4\right)n^{-1/2}$,
we get

{\footnotesize{}
\begin{align*}
P\left(\underset{\beta\in\Theta_{j}}{\sup}\frac{\left|V_{n}^{trunc}\left(\beta\right)-V_{n}^{trunc}\left(\beta_{0}\right)\right|}{2^{-j}}\geq T\lambda_{0}\right) & <\exp\left(-21\left(T-1\right)^{2}\log\left(2d\right)\log^{2}n\right).
\end{align*}
}Then

{\scriptsize{}
\begin{align*}
P\left(\underset{\beta\in\Theta}{\sup}\frac{\left|V_{n}^{trunc}\left(\beta\right)-V_{n}^{trunc}\left(\beta_{0}\right)\right|}{\left\Vert \beta-\beta_{0}\right\Vert _{1}\lor\lambda_{0}}\geq T\lambda_{0}\right) & \leq\sum_{j=j_{inf}}^{j_{sup}}\exp\left(-21\left(T-1\right)^{2}\log\left(2d\right)\log^{2}n\right)\\
 & +\exp\left(-21\left(T-1\right)^{2}\log\left(2d\right)\log^{2}n\right)\\
\\
 & \leq\left(j_{sup}-j_{inf}+2\right)\exp\left(-21\left(T-1\right)^{2}\log\left(2d\right)\log^{2}n\right).
\end{align*}
}Simplify now the expression of $j_{sup}-j_{inf}+2$,
\begin{align*}
j_{sup}-j_{inf}+2 & =\left\lceil \log_{2}2R\sqrt{d}\right\rceil -\left\lceil \log_{2}\left(\lambda_{0}\right)\right\rceil +2\\
 & \leq\log_{2}8R\sqrt{d}+1-\log_{2}\left(\lambda_{0}\right)\\
 & \leq\log_{2}\frac{16R\sqrt{d}}{\lambda_{0}}\\
 & \leq\log_{2}\left(\frac{2R\sqrt{nd}}{LM_{n}\sqrt{3\log\left(2d\right)}\log n}\right)\\
 & \leq\log_{2}\left(\frac{2R\sqrt{nd}}{LM_{n}}\right)\\
 & \leq\frac{3}{2}\log\left(\sqrt{\frac{4R^{2}nd}{L^{2}M_{n}^{2}}}\right)\\
 & \leq\frac{3}{4}\log\left(\frac{4R^{2}nd}{L^{2}M_{n}^{2}}\right).
\end{align*}
This finally gives the result.
\end{proof}
\begin{lem}
\label{lem:control_E(G_trunc)}With $G\left(X\right):=\left\Vert X\right\Vert _{\infty}$
and $a$ and $X$ defined in the section Notations: if $M_{n}=\left\Vert a\right\Vert _{\infty}+\sqrt{2\log d}+\sqrt{2\log n}$
then
\[
\mathbb{E}\left(G\left(X\right)I_{\left\{ G\left(X\right)>M_{n}\right\} }\right)\leq2\left(M_{n}+1\right)\frac{e^{-2\sqrt{\log d\log\left(1+n\right)}}}{n}
\]
and 

\[
\mathbb{E}\left(G\left(X\right)^{2}I_{\left\{ G\left(X\right)>M_{n}\right\} }\right)\leq2\left(M_{n}^{2}+\left\Vert a\right\Vert _{\infty}+1\right)\frac{e^{-2\sqrt{\log d\log\left(1+n\right)}}}{n}.
\]
\end{lem}
\begin{proof}
First note that for $y\geq\left\Vert a\right\Vert _{\infty}$,
\begin{align*}
P\left(G\left(X\right)>y\right) & \leq P\left(\max_{j}\left|Z_{j}\right|>y-\left\Vert a\right\Vert _{\infty}\right)\\
 & \leq dP\left(\left|Z_{1}\right|>y-\left\Vert a\right\Vert _{\infty}\right)\\
 & \leq2de^{-\frac{\left(y-\left\Vert a\right\Vert _{\infty}\right)^{2}}{2}}.
\end{align*}
Take now $M\geq\left\Vert a\right\Vert _{\infty}+1$, we have
\begin{align*}
\mathbb{E}\left(G\left(X\right)I_{\left\{ G\left(X\right)>M\right\} }\right) & =-\int_{M}^{\infty}y\frac{dP\left(G\left(X\right)>y\right)}{dy}dy\\
 & =\left[-yP\left(G\left(X\right)>y\right)\right]_{M}^{\infty}+\int_{M}^{\infty}1\times P\left(G\left(X\right)>y\right)dy\\
 & \leq2Mde^{-\frac{\left(M-\left\Vert a\right\Vert _{\infty}\right)^{2}}{2}}+\int_{M}^{\infty}\left(y-\left\Vert a\right\Vert _{\infty}\right)\times P\left(G\left(X\right)>y\right)dy\\
 & =2Mde^{-\frac{\left(M-\left\Vert a\right\Vert _{\infty}\right)^{2}}{2}}+2\int_{M}^{\infty}\left(y-\left\Vert a\right\Vert _{\infty}\right)\times de^{-\frac{\left(y-\left\Vert a\right\Vert _{\infty}\right)^{2}}{2}}dy\\
 & =2Mde^{-\frac{\left(M-\left\Vert a\right\Vert _{\infty}\right)^{2}}{2}}+2\left[-de^{-\frac{\left(y-\left\Vert a\right\Vert _{\infty}\right)^{2}}{2}}\right]_{M}^{\infty}\\
 & =2Mde^{-\frac{\left(M-\left\Vert a\right\Vert _{\infty}\right)^{2}}{2}}+2de^{-\frac{\left(M-\left\Vert a\right\Vert _{\infty}\right)^{2}}{2}}\\
 & =2\left(M+1\right)de^{-\frac{\left(M-\left\Vert a\right\Vert _{\infty}\right)^{2}}{2}}
\end{align*}
and

{\small{}
\begin{align*}
\mathbb{E}\left(G\left(X\right)^{2}I_{\left\{ G\left(X\right)>M\right\} }\right) & =-\int_{M}^{\infty}y^{2}\frac{dP\left(G\left(X\right)>y\right)}{dy}dy\\
 & =\left[-y^{2}P\left(G\left(X\right)>y\right)\right]_{M}^{\infty}+\int_{M}^{\infty}y\times P\left(G\left(X\right)>y\right)dy\\
 & =2M^{2}de^{-\frac{\left(M-\left\Vert a\right\Vert _{\infty}\right)^{2}}{2}}+\int_{M}^{\infty}\left(y-\left\Vert a\right\Vert _{\infty}\right)\times P\left(G\left(X\right)>y\right)dy+\int_{M}^{\infty}\left\Vert a\right\Vert _{\infty}\times P\left(G\left(X\right)>y\right)dy\\
 & \leq2M^{2}de^{-\frac{\left(M-\left\Vert a\right\Vert _{\infty}\right)^{2}}{2}}+2de^{-\frac{\left(M-\left\Vert a\right\Vert _{\infty}\right)^{2}}{2}}+2\left\Vert a\right\Vert _{\infty}\int_{M}^{\infty}\underset{\geq1}{\underbrace{\left(y-\left\Vert a\right\Vert _{\infty}\right)}}\times de^{-\frac{\left(y-\left\Vert a\right\Vert _{\infty}\right)^{2}}{2}}dy\\
 & \leq M^{2}de^{-\frac{\left(M-\left\Vert a\right\Vert _{\infty}\right)^{2}}{2}}+de^{-\frac{\left(M-\left\Vert a\right\Vert _{\infty}\right)^{2}}{2}}+2\left\Vert a\right\Vert _{\infty}de^{-\frac{\left(M-\left\Vert a\right\Vert _{\infty}\right)^{2}}{2}}\\
 & =2\left(M^{2}+\left\Vert a\right\Vert _{\infty}+1\right)de^{-\frac{\left(M-\left\Vert a\right\Vert _{\infty}\right)^{2}}{2}}.
\end{align*}
}Hence, for $M_{n}:=\left\Vert a\right\Vert _{\infty}+\sqrt{2\log d}+\sqrt{2\log\left(1+n\right)}\geq\left\Vert a\right\Vert _{\infty}+1$,
we have
\begin{align*}
\mathbb{E}\left(G\left(X\right)I_{\left\{ G\left(X\right)>M_{n}\right\} }\right) & \leq2\left(M_{n}+1\right)de^{-\frac{\left(M_{n}-\left\Vert a\right\Vert _{\infty}\right)^{2}}{2}}\\
 & \leq2\left(M_{n}+1\right)de^{-\frac{\left(\sqrt{2\log d}+\sqrt{2\log\left(1+n\right)}\right)^{2}}{2}}\\
 & \leq2\left(M_{n}+1\right)de^{-\log d-2\sqrt{\log d\log\left(1+n\right)}-\log\left(1+n\right)}\\
 & \leq2\left(M_{n}+1\right)\frac{e^{-2\sqrt{\log d\log\left(1+n\right)}}}{1+n}\\
 & \leq2\left(M_{n}+1\right)\frac{e^{-2\sqrt{\log d\log\left(1+n\right)}}}{n}
\end{align*}
and
\begin{align*}
\mathbb{E}\left(G\left(X\right)^{2}I_{\left\{ G\left(X\right)>M_{n}\right\} }\right) & \leq2\left(M_{n}^{2}+\left\Vert a\right\Vert _{\infty}+1\right)de^{-\frac{\left(M_{n}-\left\Vert a\right\Vert _{\infty}\right)^{2}}{2}}\\
 & =2\left(M_{n}^{2}+\left\Vert a\right\Vert _{\infty}+1\right)\frac{e^{-2\sqrt{\log d\log\left(1+n\right)}}}{1+n}\\
 & \leq2\left(M_{n}^{2}+\left\Vert a\right\Vert _{\infty}+1\right)\frac{e^{-2\sqrt{\log d\log\left(1+n\right)}}}{n}.
\end{align*}
\end{proof}
\begin{lem}
\label{lem:control_of_F(X)} Assume that $\left\Vert a\right\Vert _{2}\geq2R\approx2.548$.
Set 
\[
F\left(X\right):=G\left(X\right)I_{\left\{ G\left(X\right)>M_{n}\right\} }+\mathbb{E}\left[G\left(X\right)I_{\left\{ G\left(X\right)>M_{n}\right\} }\right],
\]
where $G\left(X\right)=\left\Vert X\right\Vert _{\infty}$. Moreover,
take the following constants: $M_{n}:=\left\Vert a\right\Vert _{\infty}+\sqrt{2\log d}+\sqrt{2\log\left(1+n\right)}$,
$\lambda_{0}:=3LM_{n}n^{-1/2}\left(5\sqrt{3\log\left(2d\right)}\log n+4\right)$.
It holds: $\forall T>0$,
\[
P\left(\frac{1}{n}\sum_{i=1}^{n}F\left(X^{(i)}\right)\geq\frac{\lambda_{0}T}{L}\right)\leq4\frac{L^{2}}{\lambda_{0}^{2}T^{2}}\frac{M_{n}^{2}+\left\Vert a\right\Vert _{\infty}+1}{n^{2}}.
\]
\end{lem}
\begin{proof}
Note that with our choice of $\lambda_{0}$, we have by Lemma \ref{lem:control_E(G_trunc)}:$\lambda_{0}T/L\geq2\mathbb{E}\left[G\left(X\right)I_{\left\{ G\left(X\right)>M_{n}\right\} }\right]$.
Hence,{\scriptsize{}
\begin{align*}
P\left(\frac{1}{n}\sum_{i=1}^{n}F\left(X^{(i)}\right)\geq\frac{\lambda_{0}T}{L}\right) & =P\left(\frac{1}{n}\sum_{i=1}^{n}\left(G\left(X^{(i)}\right)I_{\left\{ G\left(X^{(i)}\right)>M_{n}\right\} }+\mathbb{E}\left[G\left(X\right)I_{\left\{ G\left(X\right)>M_{n}\right\} }\right]\right)\geq\frac{\lambda_{0}T}{L}\right)\\
 & =P\left(\frac{1}{n}\sum_{i=1}^{n}\left(G\left(X^{(i)}\right)I_{\left\{ G\left(X^{(i)}\right)>M_{n}\right\} }-\mathbb{E}\left[G\left(X\right)I_{\left\{ G\left(X\right)>M_{n}\right\} }\right]\right)\geq\frac{\lambda_{0}T}{L}-2\mathbb{E}\left[G\left(X\right)I_{\left\{ G\left(X\right)>M_{n}\right\} }\right]\right)\\
 & \leq\frac{\mathbb{V}\left(\frac{1}{n}\sum_{i=1}^{n}G\left(X^{(i)}\right)I_{\left\{ G\left(X^{(i)}\right)>M_{n}\right\} }\right)}{\left(\frac{\lambda_{0}T}{L}-2\mathbb{E}\left[G\left(X\right)I_{\left\{ G\left(X\right)>M_{n}\right\} }\right]\right)^{2}}\\
 & \leq\frac{\mathbb{E}\left(G\left(X\right)^{2}I_{\left\{ G\left(X\right)>M_{n}\right\} }\right)}{n\left(\frac{\lambda_{0}T}{L}-2\mathbb{E}\left[G\left(X\right)I_{\left\{ G\left(X\right)>M_{n}\right\} }\right]\right)^{2}}.
\end{align*}
}From Lemma~\ref{lem:control_E(G_trunc)}, we get

\begin{align*}
P\left(\frac{1}{n}\sum_{i=1}^{n}F\left(X^{(i)}\right)\geq\frac{\lambda_{0}T}{L}\right) & \leq\frac{\mathbb{E}\left(G\left(X\right)^{2}I_{\left\{ G\left(X\right)>M_{n}\right\} }\right)}{n\left(\frac{\lambda_{0}T}{L}-2\mathbb{E}\left[G\left(X\right)I_{\left\{ G\left(X\right)>M_{n}\right\} }\right]\right)^{2}}.\\
 & \leq\frac{2\left(M_{n}^{2}+\left\Vert a\right\Vert _{\infty}+1\right)\frac{e^{-2\sqrt{\log d\log\left(1+n\right)}}}{n}}{n\left(\frac{\lambda_{0}T}{L}-4\left(M_{n}+1\right)\frac{e^{-2\sqrt{\log d\log\left(1+n\right)}}}{n}\right)^{2}}\\
 & \leq2L^{2}\frac{M_{n}^{2}+\left\Vert a\right\Vert _{\infty}+1}{n^{2}\lambda_{0}^{2}T^{2}}\frac{e^{-2\sqrt{\log d\log\left(1+n\right)}}}{\left(1-4L\frac{M_{n}+1}{n\lambda_{0}T}e^{-2\sqrt{\log d\log\left(1+n\right)}}\right)^{2}}.
\end{align*}
It holds, for $n\geq2$,

\begin{align*}
L\frac{M_{n}+1}{n\lambda_{0}T}e^{-2\sqrt{\log d\log\left(1+n\right)}} & \leq L\frac{M_{n}+1}{n\lambda_{0}}\\
 & =L\frac{M_{n}+1}{nT.\frac{3LM_{n}\left(5\sqrt{3\log\left(2d\right)}\log n+4\right)}{\sqrt{n}}}\\
 & =\frac{1+\frac{1}{M_{n}}}{3\sqrt{n}T\left(5\sqrt{3\log\left(2d\right)}\log n+4\right)}\\
 & \leq\frac{1}{3}\frac{1+\frac{1}{\sqrt{2\log3}}}{\sqrt{2}\left(5\sqrt{3\log2}.\log2+4\right)}\\
 & <\frac{1}{8}.
\end{align*}
Finally, we conclude that 
\begin{align*}
P\left(\frac{1}{n}\sum_{i=1}^{n}F\left(X^{(i)}\right)\geq\frac{\lambda_{0}T}{L}\right) & <4L^{2}\frac{M_{n}^{2}+\left\Vert a\right\Vert _{\infty}+1}{n^{2}\lambda_{0}^{2}T^{2}}.
\end{align*}
\end{proof}
\begin{lem}
\label{lem:control_V_tail}Recall from Lemma~\ref{lem:control_empirical_process_trunc}
that $V_{n}^{trunc}\left(\beta\right)=(P_{n}-P)(\rho_{\beta}I_{\left\{ G\leq M_{n}\right\} })$
and $V_{n}\left(\beta\right)=\left(P_{n}-P\right)\rho_{\beta}$. Recall
also from Lemma~\ref{lem:control_of_F(X)} that $F\left(X\right)=G\left(X\right)I_{\left\{ G\left(X\right)>M_{n}\right\} }+\mathbb{E}\left[G\left(X\right)I_{\left\{ G\left(X\right)>M_{n}\right\} }\right]$
with $G\left(X\right)=\left\Vert X\right\Vert _{\infty}$. It holds
true that $\forall T\geq1$,
\[
P\left(\sup_{\beta\in B_{2}\left(0,R\right)}\frac{\left|V_{n}^{trunc}\left(\beta\right)-V_{n}^{trunc}\left(\beta_{0}\right)-\left(V_{n}\left(\beta\right)-V_{n}\left(\beta_{0}\right)\right)\right|}{\left\Vert \beta-\beta_{0}\right\Vert _{1}\lor\lambda_{0}}>T\lambda_{0}\right)\leq P\left(\frac{1}{n}\sum_{i=1}^{n}F\left(X^{(i)}\right)>\frac{T\lambda_{0}}{L}\right).
\]
\end{lem}
\begin{proof}
Basic computations and H{\"o}lder's inequality give
\begin{align*}
 & \left|V_{n}^{trunc}\left(\beta\right)-V_{n}^{trunc}\left(\beta_{0}\right)-\left(V_{n}\left(\beta\right)-V_{n}\left(\beta_{0}\right)\right)\right|\\
 & =\left|\left(P_{n}-P\right)\left(\rho_{\beta}I_{\left\{ G>M_{n}\right\} }\right)-\left(P_{n}-P\right)\left(\rho_{\beta_{0}}I_{\left\{ G>M_{n}\right\} }\right)\right|\\
 & \leq\left|P_{n}\left[\left(\rho_{\beta}-\rho_{\beta_{0}}\right)I_{\left\{ G>M_{n}\right\} }\right]\right|+\left|P\left[\left(\rho_{\beta}\left(X\right)-\rho_{\beta_{0}}\left(X\right)\right)I_{\left\{ G\left(X\right)>M_{n}\right\} }\right]\right|\\
 & \leq\frac{1}{n}\sum_{i=1}^{n}L\left|X^{(i)}\left(\beta-\beta_{0}\right)\right|I_{\left\{ G\left(X^{(i)}\right)>M_{n}\right\} }+\mathbb{E}\left[L\left|X\left(\beta-\beta_{0}\right)\right|I_{\left\{ G\left(X\right)>M_{n}\right\} }\right]\\
 & \leq L\left(\frac{1}{n}\sum_{i=1}^{n}\left\Vert X^{(i)}\right\Vert _{\infty}I_{\left\{ G\left(X^{(i)}\right)>M_{n}\right\} }\left\Vert \beta-\beta_{0}\right\Vert _{1}+\mathbb{E}\left[\left\Vert X\right\Vert _{\infty}I_{\left\{ G\left(X\right)>M_{n}\right\} }\right]\left\Vert \beta-\beta_{0}\right\Vert _{1}\right)\\
 & \leq\frac{L\left\Vert \beta-\beta_{0}\right\Vert _{1}}{n}\sum_{i=1}^{n}F\left(X^{(i)}\right),
\end{align*}
and the result directly follows.
\end{proof}
\bibliographystyle{plain}
\bibliography{biblio_clustering_grande_dim,mybibfile,Slope_heuristics_regression_13,bibliKMOM}

\end{document}